\documentclass[12pt]{amsart}
\usepackage{graphicx}
\usepackage{graphics, epsfig}
\newtheorem{thm}{Theorem}[section]
\newtheorem{cor}[thm]{Corollary}
\newtheorem{cond}[thm]{Condition}
\newtheorem{lem}[thm]{Lemma}
\newtheorem{prop}[thm]{Proposition}
\theoremstyle{definition}

\theoremstyle{remark}
\newtheorem{rem}[thm]{Remark}
\numberwithin{equation}{section}

\begin{document}

\title[Optimal control of a large dam]{Optimal control
of a large dam\\ with compound Poisson input\\ and costs depending
on water levels}%
\author{Vyacheslav M. Abramov}%
\address{Monash University, School of Mathematical Sciences, Wellington Road, Clayton Campus, Clayton, Victoria-3800, Australia}%

\email{vabramov126@gmail.com}%

\begin{abstract}
This paper studies a discrete model of a large dam where the
difference between lower and upper levels, $L$, is assumed to be
large. Passage across the levels leads to damage, and the damage
costs of crossing the lower or upper level are proportional to the
large parameter $L$. Input stream of water is described by compound
Poisson process, and the water cost depends upon current level of
water in the dam. The aim of the paper is to choose the parameters
of output stream (specifically defined in the paper) minimizing the
long-run expenses that include the damage costs and water costs. The
present paper addresses the important question \textit{How does the
structure of water costs affect the optimal solution?}  We prove the
existence and uniqueness of a solution. A special attention is
attracted to the case of linear structure of the costs.

As well, the paper contributes to the theory of state-dependent queueing systems. The inter-relations between important characteristics of a state-dependent queueing system are established, and their asymptotic analysis that involves analytic techniques of Tauberian theory and heavy traffic approximations is provided.
\end{abstract}

%
%
\maketitle

\tableofcontents

\section{Introduction}\label{Introduction}

\subsection{Description of the system and formulation of the problem}
The paper studies a discrete model of a large dam. A large dam is
specified by the parameters $L^{\mathrm{lower}}$ and
$L^{\mathrm{upper}}$, which are, respectively, the lower and upper
levels of the dam. If the current water level is between these
bounds, the dam is assumed to be in a normal state. The difference
$L = L^{\mathrm{upper}}-L^{\mathrm{lower}}$ is large, and this is
the reason for calling the dam \textit{large}. This feature enables
us to use asymptotic analysis as $L\to\infty$ and solve different
problems of optimal control. (A direct way without an asymptotic
analysis is very hard.)

Let $L_t$ denote the water level at time $t$. If
$L^{\mathrm{lower}}<L_t\leq L^{\mathrm{upper}}$, then the state of
the dam is called \textit{normal}. Passage across upper level
(flooding) or reaching lower level (emptiness) lead to damage. The
costs per unit time of this damage are
\begin{equation}\label{0.1}
J_1=j_1L
\end{equation} for the lower level and,
respectively,
\begin{equation}\label{0.2}J_2=j_2L\end{equation}
for the upper level, where $j_1$ and $j_2$ are given real constants.
(In real dams that are large, the damage costs are proportional to
the capacity of dam.)

The water inflow is described by a compound Poisson process. Namely,
the probability generating function of input amount of water (which
is assumed to be an integer-valued random variable) in an interval
$t$ is given by
 \begin{equation}\label{pgfArrival}
 f_t(z)=\exp\left\{-\lambda t\left(1-\sum_{i=1}^{\infty}r_iz^i\right)\right\},
 \end{equation}
where $r_i$ is the probability that at a specified moment of Poisson
arrival the amount of water will increase by $i$ units. In practice,
this means that the arrival of water is registered at random
instants $t_1$, $t_2$, \ldots; the times between consecutive
instants are mutually independent and exponentially distributed with
parameter $\lambda$, and quantities of water (number of water units)
of input flow are specified as a quantity $i$ with probability $r_i$
($r_1+r_2+\ldots=1$). Clearly that this assumption is more
applicable to real world problems than the assumption of
\cite{Abramov 2007} where the inter-arrival times of water units are
exponentially distributed with parameter $\lambda$. For example, the
assumption made in the present paper enables us to approach a
continuous dam model, assuming that the water levels $L_t$ take the
discrete values $\{j\Delta\}$, where $j$ is a positive integer and
step $\Delta$ is a positive small real constant. In the paper, the
water levels $L_t$ are assumed to be integer-valued. The
aforementioned set of values $\{j\Delta\}$ for water levels can be
obtained by scaling.

The outflow of water is state-dependent as follows. If the level of
water is between $L^{\mathrm{lower}}$ and $L^{\mathrm{upper}}$, then
an interval between departures of water units has the probability
distribution function $B_1(x, C)$ (depending on parameter $C$, the
meaning of which will become clear later). If the level of water
exceeds $L^{\mathrm{upper}}$, then an interval between departures of
water units has the probability distribution function $B_2(x)$. The
probability distribution function $B_2(x)$ is assumed to obey the
condition $\int_0^\infty
x\mbox{d}B_2(x)<{1}/(\lambda\mathsf{E}\varsigma)$, where
$\mathsf{E}\varsigma$ is the mean batch size.  If the level of water
is $L^{\mathrm{lower}}$ exactly, then output of water is frozen, and
it resumes again as soon as the level of water exceeds the level
$L^{\mathrm{lower}}$. (The exact mathematical formulation of the
problem taking into account some specific details is given below.)

Let $c_{L_t}$ denote the cost of water at level $L_t$. The sequence
$c_i$ is assumed to be positive and non-increasing. The problem of
the present paper is to choose the parameter $C$ (and, specifically,
$\int_0^\infty x\mbox{d}B_1(x, C)$ and $\int_0^\infty
x^2\mbox{d}B_1(x, C)$) of the dam in the normal state minimizing the
objective function
\begin{equation}\label{I1}
J = p_1J_1 + p_2J_2 +
\sum_{i=L^{\mathrm{lower}}+1}^{L^{\mathrm{upper}}}c_iq_i,
\end{equation}
where
\begin{eqnarray}
p_1&=&\lim_{t\to\infty}\mathsf{Pr}\{L_t=L^{\mathrm{lower}}\},\label{I2}\\
p_2&=&\lim_{t\to\infty}\mathsf{Pr}\{L_t>L^{\mathrm{upper}}\},\label{I3}\\
q_i&=&\lim_{t\to\infty}\mathsf{Pr}\{L_t=L^{\mathrm{lower}}+i\}, \
i=1,2,\ldots,L\label{I4},
\end{eqnarray}
and $J_1$ and $J_2$ are defined in \eqref{0.1} and \eqref{0.2}.

Usually $L^{\mathrm{lower}}$ is identified with an empty queue (i.e.
$L^{\mathrm{lower}}:=0$ and $L^{\mathrm{upper}}:=L$). Just this
assumption is made in the paper, and the dam model is the following
queueing system with service depending on queue-length. If immediately before
moment of a service start the number of customers in the system exceeds the level
$L$, then the customer is served by the probability distribution
function $B_2(x)$. Otherwise, the service time distribution is
$B_1(x)$. The value $p_1$ is the stationary probability of an empty
system, the value $p_2$ is the stationary probability that a
customer is served by probability distribution $B_2(x)$, and $q_i$,
$i=1,2,\ldots,L$, are the stationary probabilities of the
queue-length process, so $p_1+p_2+\sum_{i=1}^L q_i=1$. (For
the described queueing system, the right-hand side limits in
relations \eqref{I2}--\eqref{I4} do exist.)

So, in the present paper we assume that $L^{\mathrm{lower}}=0$ and
$L^{\mathrm{upper}}=L$. We also assume that initial water level is 0.

In our study, the parameter $L$ increases unboundedly, and we deal
with the series of queueing systems. The above parameters, such as
$p_1$, $p_2$, $J_1$, $J_2$ as well as other parameters are functions
of $L$. The argument $L$ will be often omitted in these functions.

\subsection{Motivation, discussion of the study and review of
related literature} Similarly to \cite{Abramov 2007}, it is assumed
that the input parameter $\lambda$, the probabilities $r_1$,
$r_2$,\ldots and probability distribution function $B_2(x)$ are
given, while the appropriate probability distribution function
$B_1(x,C)$ should be found from the specified parametric family of
functions $B_1(x,C)$, where
\begin{equation}\label{Old_def_of_C}
C=\lim_{L\to\infty}L\delta(L),
\end{equation}
where $\delta(L)$ is a specified nonnegative vanishing parameter of
the system as $L\to\infty$.

The reason of solving this specific problem, where the probability
distribution function $B_2(x)$ is given while the probability
distribution functioned $B_1(x, C)$ is controlled, is that in
practical problems of the water resources planning, it is important
to know how much water should be used per unit time in order to
minimize the risks of the disasters such as emptiness or flooding
the dam.

 The outflow rate,
should be chosen such that to minimize the objective function of
\eqref{I1} with respect to the parameter $C$, which results in
choice of the corresponding probability distribution function
$B_1(x,C)$ of that family.

A particular problem have been studied in \cite{Abramov 2007}. A
circle of problems associated with the results of \cite{Abramov
2007} are discussed in a review paper \cite{Abramov 2009}.

The simplest model with Poisson input stream and the objective
function having the form $J= p_1J_1 + p_2J_2$ (i.e. the water costs
are not taken into account), has been studied in \cite{Abramov
2007}. Denote $\rho_2=\lambda\int_0^\infty x\mathrm{d}B_2(x)$ and
$\rho_1=\rho_1(C)=\lambda\int_0^\infty x\mathrm{d}B_1(x,C)$. (The
optimal value of $C$ is unique in a minimization problem precisely
formulated in \cite{Abramov 2007}.) It was shown in \cite{Abramov
2007} that the unique solution to the control problem has one of the
three forms given there in the formulation of Theorem 5.1. The
aforementioned three forms in the formulation of Theorem 5.1 in
\cite{Abramov 2007} fall into the category of a large area of heavy
traffic analysis in queueing theory.  We mention the books of Chen
and Yao \cite{Chen and Yao 2001} and Whitt \cite{Whitt 2001}, where
a reader can find many other references. It has also been shown in
\cite{Abramov 2007} that the solution to the control problem is
insensitive to the type of probability distributions $B_1(x, C)$ and
$B_2(x)$. Specifically, it is expressed via the first moment of
$B_2(x)$ and the first two moments of $B_1(x, C)$.

Compared to the earlier studies in \cite{Abramov 2007}, the solution
of the problems in the present paper requires a much deepen and
delicate analysis. The results of \cite{Abramov 2007} are extended
in two directions: (1) the arrival process is compound Poisson
rather than Poisson, and (2) structure of water costs in dependence
of the level of water in the dam is included.

The first extension leads to new techniques of stochastic analysis.
The main challenge in \cite{Abramov 2007} was reducing the certain
characteristics of the system during a busy period to the
convolution type recurrence relation such as $Q_n=\sum_{i=0}^n
Q_{n-i+1}f_i$ ($Q_0\neq0$), where $f_0>0$, $f_i\geq0$ for all
$i\geq1$, $\sum_{i=0}^\infty f_i=1$ and then using the known results
on the asymptotic behaviour of $Q_n$ as $n\to\infty$. In the case
when arrivals are compound Poisson, the same characteristics of the
system cannot be reduced to the aforementioned convolution type of
recurrence relation. Instead, we obtain a more general scheme
including as a part the aforementioned recurrence relation. In this
case, asymptotic analysis of the required characteristics becomes
very challenging. It is based on special stochastic domination
methods.

The second extension leads to new analytic techniques of asymptotic
analysis. Asymptotic methods of \cite{Abramov 2007} are no longer
working, and more delicate techniques should be used instead. That
is, instead of Tak\'acs' asymptotic theorems \cite{Takacs 1967}, p.
22-23, special Tauberian theorems with remainder by Postnikov
\cite{Postnikov 1980}, Sect. 25 should be used. For different
applications of the aforementioned Tak\'acs' asymptotic theorems and
Tauberian theorems of Postnikov see \cite{Abramov 2009}.

Another challenging problem for the dam model in the present paper
is the solution to the control problem, that is, the proof of a
uniqueness of the optimal solution. In the case of the model in
\cite{Abramov 2007} the existence and uniqueness of a solution
follows automatically from the explicit representations of the
functionals obtained there. (The existence of a solution follows
from the fact that in the case $\rho_1=1$ we get a bounded value of
the functional, while in the cases $\rho_1<1$ and $\rho_1>1$ the
functional is unbounded. Then the uniqueness of a solution reduces
to elementary minimization problem for smooth convex functions.)

In the case of the model in the present paper, the solution of the
problem with extended criteria \eqref{I1} is related to the same
set of solutions as in \cite{Abramov 2007}. That is, it must be
either $\rho_1=1$ or one of the two limits of $\rho_1=1+\delta(L)$,
$\rho_1=1-\delta(L)$ for positive small vanishing $\delta(L)$ as the
series parameter $L$ increases unboundedly, and $L\delta(L)\to C$.
In the present study, it is convenient to define the parameter $C$
as
\begin{equation}\label{Def_of_C}
C=\lim_{L\to\infty}L[\rho_1(L)-1],
\end{equation}
and use $C(L)=L[\rho_1(L)-1]$ as a series parameter. Hence, it is
quite natural to consider $\rho_1(L)$ as a control sequence, while
$C(L)$ is a sequence derivative from $\rho_1(L)$. Furthermore, in
this case $C(L)$ is expressed uniquely via $\rho_1(L)$ and vice
versa.

The definition of the parameter $C$ given here differs from that
definition given in \cite{Abramov 2007}. Unlike \cite{Abramov 2007},
where $C$ was defined as a nonnegative control parameter (see
\eqref{Old_def_of_C}), in the present definition \eqref{Def_of_C}
the value of $C$ can be either positive or negative.

Unlike \cite{Abramov 2007}, we use the notation
$\rho_{1,l}(L)=\lambda^l\int_0^\infty x^l\mbox{d}B_1(x, C(L))$,
$l=2,3$. The existence of $\rho_{1,3}(L)$ (i.e. the moments of the
third order of $B_1(x, C(L))$) will be specially assumed in the
formulations of the statements corresponding to the case studies.

It is assumed in the present paper that $c_i$ is a nonincreasing
sequence. If the cost sequence $c_i$ were an arbitrary bounded
sequence, then a richer set of possible cases could be studied.
However, in the case of arbitrary cost sequence, the solution does not need
be unique.

A nonincreasing sequence $c_i$ depends on $L$ in series. This means
that as $L$ changes (increasing to infinity) we have different not
increasing sequences (see example in Section \ref{Examples}). The
initial value $c_1$ and final value $c_L$ are taken fixed and
strictly positive, and the limit of $c_L$ as $L\to\infty$ is assumed
to be positive as well.

Realistic models arising in practice assume that the probability
distribution function $B_1(x, C)$ should also depend on $i$, i.e
have the representation $B_{1,i}(x, C)$. The model of the present
paper, where $B_1(x, C)$ is the same for all $i$, under appropriate
additional information can approximate those more general models.
Namely, one can suppose that the stationary service time
distribution $B_1(x, C)$ has the representation $B_1(x,
C)=\sum_{i=1}^Lq_iB_{1,i}(x, C)$ ($q_i$, $i=1,2,\ldots,L$ are the
state probabilities), and the solution to the control problem for
$B_1(x, C)$ enables us to find then the approximate solution to the
control problem for $B_{1,i}(x, C)$, $i=1,2,\ldots,L$. The
additional information about $B_{1,i}(x, C)$, $i=1,2,\ldots,L$,
might be that all these distributions belong to the same parametric
family of distributions with known relationships between the values
of a parameter. For instance, the simplest model can be of the form
$B_1(x, C)=aB_{1}^*(x, C)+bB_1^{**}(x, C)$, where the distributions
$B_{1}^*(x, C)$ and $B_{1}^{**}(x, C)$ are of the same type (say,
two-phase Erlang distributions), $a:=\sum_{i=1}^{L^0}q_i$ ($L^0<L$),
and, respectively, $b:=\sum_{i=L^0+1}^{L}q_i$, and the relationship
between the means,
$$
\gamma=\frac{\int_0^\infty x\mathrm{d}B_1^*(x,C)}{\int_0^\infty
x\mathrm{d}B_1^{**}(x,C)},
$$
is known.

\smallskip
In the present paper we address the following questions.
\medskip

$\bullet$ Uniqueness of an optimal solution and its structure.

\smallskip
$\bullet$ Interrelation between the parameters $j_1$, $j_2$,
$\rho_2$, $c_i$ ($i=1,2,\ldots,L$) when the optimal solution is
$\rho_1=1$.

\medskip
The uniqueness of an optimal solution is given by Theorem
\ref{thm3*}. In the case of the model considered in \cite{Abramov
2007} the condition, when the optimal solution is $\rho_1=1$, the
interrelation between the parameters $j_1$, $j_2$ and $\rho_2$ is
$j_1=j_2{\rho_2}/(1-\rho_2)$. In the case of the model considered in
this paper, when the optimal solution is $\rho_1=1$, the
interrelation between the aforementioned and some additional
parameters gives us the inequality $j_1\leq j_2{\rho_2}/(1-\rho_2)$
(see Section \ref{Solution}, Corollary \ref{cor2}).

A more exact result is obtained in the particular case of linearly
decreasing costs as the level of water increases (for brevity, this
case is called \textit{linear costs}). In this case, a numerical
solution of the problem is given.

The solution to the control problem enables us to find optimal
initial condition of the system. The steady state distribution,
under which the optimal value of the control parameter $\rho_1(L)$
is achieved, is associated with the optimal initial level of water
in the dam.

\subsection{Organization of the paper}
The rest of the paper is organized as follows. In Section
\ref{Methodology} the main ideas and methods of asymptotic analysis
are given. In Section \ref{MG1}, we recall the basic methods related
to state dependent queueing system with ordinary Poisson input that
have been used in \cite{Abramov 2007}. Then in Section \ref{MXG1},
extensions of these methods for the model considered in this paper
are given. Specifically, the methodology of constructing linear
representations between mean characteristics given during a busy
period is explained. Main results of the paper are formulated in
Section \ref{Main results}.

The sections following after Section \ref{Main results} are of two
types. The first type of the results is presented in Sections
\ref{Stationary probabilities} and \ref{Q stationary probabilities}.
These large sections present the preliminary results of the paper
and study the characteristics of the system, their asymptotic
behaviour and specifically the asymptotic behaviour of different
stationary probabilities. The second type of the results is
presented in Sections \ref{ObFunction}, \ref{Solution},
\ref{Examples} and \ref{Numerical results} and related to the
solution to the control problem and further study of its properties.

In Section \ref{Stationary probabilities}, the asymptotic behavior
of the stationary probabilities $p_1$ and $p_2$ for specific sets of
states are studied. In Section \ref{Preliminaries}, known Tauberian
theorems that are used in the asymptotic analysis in the paper are
recalled. Section \ref{Exact formulae} establishes explicit formulae
for the probabilities $p_1$ and $p_2$. In Section \ref{Preliminary
asymp} some preliminary results are established for the further
study of asymptotic behaviour of stationary probabilities $p_1$ and
$p_2$ as $L\to\infty$ in Sections \ref{Final asymp} and \ref{Final
asymp 2}. Section \ref{Q stationary probabilities} is devoted to
asymptotic analysis of the stationary probabilities $q_{L-i}$,
$i=1,2,\ldots$. In Section \ref{Explicit q}, the explicit
representation for the stationary probabilities $q_i$ is derived. On
the basis of this explicit representation and Tauberian theorems, in
following Sections \ref{Case 1}, \ref{Case 2} and \ref{Case 3}
asymptotic theorems for these stationary probabilities are
established in the cases $\rho_1=1$, $\rho_1=1+\delta(L)$ and
$\rho_1=1-\delta(L)$ correspondingly, where positive $\delta(L)$ is
assumed vanishing such that $L[\rho_1(L)-1] \to C$ as $L\to\infty$.
In Section \ref{ObFunction} the objective function given in
\eqref{I1} is studied. In following Sections \ref{Case 1O},
\ref{Case 2O} and \ref{Case 3O}, the asymptotic theorems for this
objective function are established for the cases $\rho_1=1$,
$\rho_1=1+\delta(L)$ and $\rho_1=1-\delta(L)$, correspondingly. In
Section \ref{Solution}, the theorem on existence and uniqueness of a
solution is proved. In Section \ref{Examples}, the case of linear
costs is studied. Numerical results relevant to Section
\ref{Examples} are provided in Section \ref{Numerical results}.
Section \ref{Proofs} contains long proofs of the lemmas, theorems
and propositions formulated in the paper.

\section{Methodology of analysis}\label{Methodology}

In this section we describe the methodology used in the present
paper. This is very important for the following two reasons. The
first reason is that the standard approach of a diffusion
approximation (transient) of a dam process with the following
computation of the stationary distribution of the diffusion is hard,
because in that case we should deal with the interchange of the
order of limits (see discussion on the page 514 of Whitt \cite{Whitt
2004} as well as in Whitt \cite{Whitt 2005}). The second reason is that the earlier methods of
\cite{Abramov 2007} do not work for this extended model and, hence,
need in substantial revision.

We start from the model where arrivals are Poisson, and then we
explain how the methods should be developed for the model where an
arrival process is compound Poisson. In this and later sections we
write $B_1(x)$ rather than $B_1(x,C)$. As well, the parameter $L$
will be omitted from the related notation for the characteristics of
the system.

\subsection{State dependent queueing system with Poisson input and
its characteristics}\label{MG1}

In this section, we consider the simplest model in which arrival
flow is Poisson with parameter $\lambda$. The service time depends upon queue-length as follows. If immediately before a service of a customer, the number of customers in the system is not greater than $L$, then the probability distribution function is $B_1(x)$. Otherwise, if the number of customers in the system exceeds $L$, then the probability distribution function is $B_2(x)$. Note, that in the case when $L=0$, then the only first customer arrived in a busy period has the probability distribution function is $B_1(x)$; all other has the probability distribution function $B_2(x)$.

 Let $T_L$ denote the
length of a busy period of this system, and let $T_L^{(1)}$,
$T_L^{(2)}$ denote the cumulative times spent for service of
customers arrived during that busy period with probability
distribution functions $B_1(x)$ and $B_2(x)$ correspondingly. For
$k=1,2$, the expectations of service times will be denoted by
${1}/{\mu_k}=\int_0^\infty x\mathrm{d}B_k(x)$, and the loads by
$\rho_k={\lambda}/{\mu_k}$. Let $\nu_L$, $\nu_L^{(1)}$ and
$\nu_L^{(2)}$ denote correspondingly the number of served customers
during a busy period, and the numbers of those customers served by
the probability distribution functions $B_1(x)$ and $B_2(x)$. The
random variable $T_L^{(1)}$ coincides in distribution with a busy
period of the $M/G/1/L$ queueing system ($L$ is the number of
waiting places excluding the place for server). The elementary
explanation of this fact is based on a property of level crossings
and the property of the lack of memory of exponential distribution
(e.g. \cite{Abramov 2007}), so the analytic representation for
$\mathsf{E}T_L^{(1)}$ is the same as this for the expected busy
period of the $M/G/1/L$ queueing system. The recurrence relation for
the Laplace-Stieltjes transform and consequently that for the
expected busy period of the $M/G/1/L$ queueing system has been
derived by Tomko \cite{Tomko} (see also Cooper and Tilt \cite{CT}). So, for $\mathsf{E}T_L^{(1)}$ the
following recurrence relation is satisfied:
\begin{equation}\label{MG1.1}
\mathsf{E}T_L^{(1)}=\sum_{i=0}^L\mathsf{E}T_{L-i+1}^{(1)}\int_0^\infty\mathrm{e}^{-\lambda
x}\frac{(\lambda x)^i}{i!} \mathrm{d}B_1(x),
\end{equation}
where $\mathsf{E}T_0^{(1)}={1}/{\mu_1}$.

\begin{rem}The random variable $T_i^{(1)}$ is defined similarly to that of
$T_L^{(1)}$. In that case the parameter $i$ is the threshold value
of the model, and the set $\{\mathsf{E}T_i^{(1)}\}$ may be thought
as the set of mean busy periods of $M/G/1/i$ queueing systems with
the same parameter of Poisson input, the same probability
distribution of service time, but different number of waiting
places.
\end{rem}

Recurrence relation \eqref{MG1.1} is a particular form of the
recurrence relation
\begin{equation}\label{MG1.1.1}
Q_n=\sum_{i=0}^n Q_{n-i+1}f_i,
\end{equation}
where $Q_0\neq0$, $f_0>0$, $f_i\geq0$, $i=1,2,\ldots$ and
$\sum_{i=0}^\infty f_i=1$ (see Tak\'acs \cite{Takacs 1967}).

Using the obvious system of equations given by (2.1) and (2.2) in
\cite{Abramov 2007} and Wald's equations (see \cite{Feller 1966},
p.384) given by (2.3) and (2.4) in [1]
one can express the quantities $\mathsf{E}T_L$, $\mathsf{E}\nu_L$,
$\mathsf{E}T_L^{(2)}$, $\mathsf{E}\nu_L^{(1)}$ and
$\mathsf{E}\nu_L^{(2)}$ all via $\mathsf{E}T_{L}^{(1)}$ as the
linear functions. Since equations (2.1) -- (2.4) of \cite{Abramov
2007} should be mentioned many times in this paper, we find
convenient to list them here for following direct references:
\begin{eqnarray}
\mathsf{E}T_L&=&\mathsf{E}T_{L}^{(1)}+\mathsf{E}T_{L}^{(2)},\label{MG1.2}\\
\mathsf{E}\nu_L&=&\mathsf{E}\nu_{L}^{(1)}+\mathsf{E}\nu_{L}^{(2)},\label{MG1.3}\\
\mathsf{E}T_L^{(1)}&=&\frac{1}{\mu_1}\mathsf{E}\nu_L^{(1)},\label{MG1.4}\\
\mathsf{E}T_L^{(2)}&=&\frac{1}{\mu_2}\mathsf{E}\nu_L^{(2)}.\label{MG1.5}
\end{eqnarray}

Then, to obtain the desired linear representations note that the
number of arrivals during a busy cycle coincides with the total
number of customers served during a busy period. That is, for their
expectations we have
\begin{equation}\label{MG1.5.1}
\lambda\left(\mathsf{E}T_L+\frac{1}{\lambda}\right)=\lambda\mathsf{E}T_L+1=\mathsf{E}\nu_L,
\end{equation}
which together with the aforementioned relations \eqref{MG1.2} --
\eqref{MG1.5} yields the linear representations required.

For example,
\begin{equation}\label{MG1.5.2}
\mathsf{E}\nu_L^{(2)}=\frac{1}{1-\rho_2}-\mu_1\cdot\frac{1-\rho_1}{1-\rho_2}\mathsf{E}T_L^{(1)},
\end{equation}
and
\begin{equation}\label{MG1.5.3}
\mathsf{E}T_L^{(2)}=\frac{\rho_2}{\lambda(1-\rho_2)}-\frac{\rho_2}{\rho_1}
\cdot\frac{1-\rho_1}{1-\rho_2}\mathsf{E}T_L^{(1)}.
\end{equation}
As a result, the stationary probabilities $p_1$ and $p_2$ both are
expressed via $\mathsf{E}\nu_{L}^{(1)}$ as follows:
\begin{equation*}
p_1=\frac{1-\rho_2}{1+(\rho_1-\rho_2)\mathsf{E}\nu_L^{(1)}},
\end{equation*}
\begin{equation*}
p_2=\frac{\rho_2+\rho_2(\rho_1-1)\mathsf{E}\nu_L^{(1)}}{1+(\rho_1-\rho_2)\mathsf{E}\nu_L^{(1)}}.
\end{equation*}
It is interesting to note that the coefficients in the linear
representations all are insensitive to the probability distribution
functions $B_1(x)$ and $B_2(x)$ and are only expressed via
parameters such as $\mu_1$, $\mu_2$ and $\lambda$.

The asymptotic behaviour of $\mathsf{E}T_L^{(1)}$ as $L\to\infty$
that given by \eqref{MG1.1} is established on the basis of the known
asymptotic behaviour of the sequence $Q_n$ as $n\to\infty$ that
given by \eqref{MG1.1.1} (see \cite{Takacs 1967}, p.22,
\cite{Postnikov 1980} as well as recent paper \cite{Abramov 2009}).
To make the paper self-contained, the necessary results about the
asymptotic behaviour of $Q_n$ as $n\to\infty$ are given in Section
\ref{Preliminaries}.

\subsection{State dependent queueing system with
compound Poisson input and its characteristics}\label{MXG1}

\subsubsection{Historic background}

For $M^X/G/1/L$ queues, certain characteristics associated with busy
periods have been studied by Rosenlund \cite{Rosenlund 1973}.
Developing the results of Tomko \cite{Tomko}, Rosenlund
\cite{Rosenlund 1973} has derived the recurrence relations for the
joint Laplace-Stieltjes and $z$-transform of two-dimensional
distributions of a generalized busy period and the number of
customers served during that period. In turn, both of these
approaches \cite{Tomko} and \cite{Rosenlund 1973} are based on a
well-known Tak\'acs' method (see \cite{Takacs 1955} or \cite{Takacs
1962}).

For further analysis, \cite{Rosenlund 1973} used matrix-analytic
techniques and techniques of the theory of analytic functions. This
type of analysis is very hard and seems cannot be easily adapted for
the purposes of the present paper, where a more general model than
that from \cite{Rosenlund 1973} is studied. Busy periods and
loss characteristics during busy periods for $M^X/G/1/n$ systems
have also been studied by Pacheco and Ribeiro \cite{Pacheco and
Ribeiro 1}, \cite{Pacheco and Ribeiro 2} and Ferreira, Pacheco and
Ribeiro \cite{Ferreira}.

Along with previously mentioned paper \cite{Abramov 2007}, the
method of asymptotic analysis closely related to subject matter of
this paper have been considered by Abramov \cite{Abramov 2002} and
\cite{Abramov 2004} and further reviewed in \cite{Abramov 2009}.

The first studies of single server queueing systems with Poisson input and service depending on queue-length were due to Suzuki \cite{Suzuki1}, \cite{Suzuki2}, and the paper by Suzuki and Ebe \cite{SuzukiEbe} was probably the first one to consider decision rules problems. Since then there have been numerous publications related to state-dependent queueing systems, e.g. Knessl \textit{et al} \cite{KMST},\cite{KMST2}, Mandelbaum and Pats \cite{MP},  Miller \cite{M 2009} and Miller and McInnes \cite{MillerMcInnes}. Some of these publications include control problems as well. Basic results for a single-server queueing system with Poisson input and service depending on queue-length can also be found in Abramov \cite{Abramov 1991}.

\subsubsection{Structure of busy periods, and reasoning for the recurrence
relations of convolution type}\label{S2.2.2} In this section we
explain how the method of Section \ref{MG1} can be extended, and how
the characteristics of the system can be expressed via the similar
convolution type recurrence relations.

To explain the origin of the convolution type recurrence relations
in queueing systems and further representation for the required mean
characteristics, we recall what the structure of busy period in
systems with Poisson input is, and how this structure is extended
from relatively simple systems to more complicated ones.

Let us first recall the structure of a busy period in the $M/G/1$
queueing system (e.g. Tak\'acs \cite{Takacs 1955}, \cite{Takacs
1962}).

The busy period starts upon arrival of a customer in the idle
system. If during the service time of the customer no arrival
occurs, then the length of the busy period coincides with the length
of the service. If at least one arrival occurs, then the structure
of busy period is as follows. Suppose that during a service time
there are $n$ arrivals. Then, the time interval from the moment of
service beginning when there are $n$ customers in the system until
the moment when there remain only $n-1$ customers in the system at
the first time after the interval start, coincides in distribution
with a busy period. In Figure 1, the typical structure of the $M/G/1$
busy period is shown.

\begin{figure}
\includegraphics[width=10cm, height=15cm]{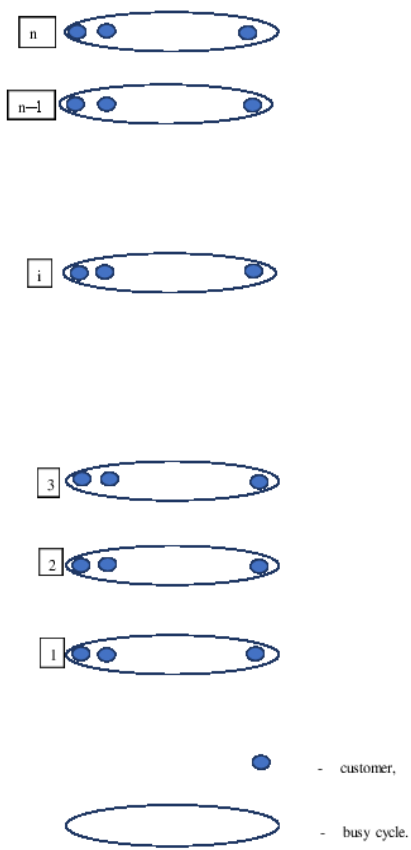}
\caption{General scheme of a busy period in the $M/G/1$ queue.}
\end{figure}

Consider now the busy period in the $M/G/1$ queueing system with
threshold level $L$. In this system, the service time distribution
depends on the number of customers in the queue as follows. If immediately before the
service start the number of customers in the system is
greater than $L$, then the service time distribution of that
customer is $B_2(x)$. Otherwise, it is $B_1(x)$. The typical
structure of a busy period is indicated in Figure 2, where customers
that are over the threshold level (served by probability
distribution $B_2(x)$) are indicated with white color, and all other
customers (served by probability distribution $B_1(x)$) are
indicated with dark color.

\begin{figure}
\includegraphics[width=10cm, height=15cm]{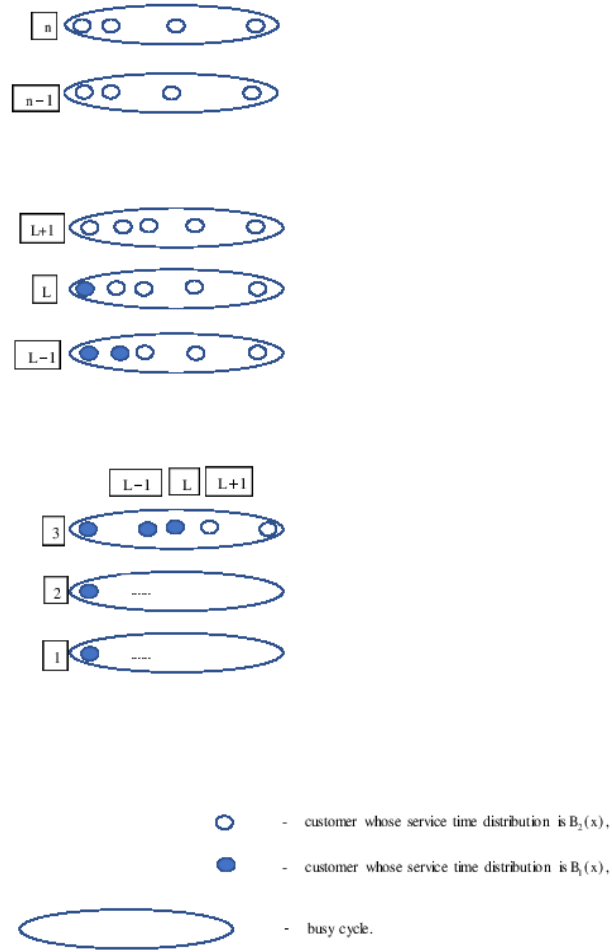}
\caption{General scheme of a busy period in the state-dependent
$M/G/1$ queue, where $L$ is the threshold level.}
\end{figure}

Unlike the case of the standard $M/G/1$ queue (without threshold),
in the case of the state-dependent queueing systems with service
depending on queue-length, a time interval from the moment of
service beginning when the number of customers in the system is $i$
until the time moment when the number of customers becomes $i-1$ at
the first time since its start, depends on $i$. Specifically, if
$i\leq L$, then the length of the aforementioned cycle is
distributed as the state-dependent queueing system with the biased
threshold equal to $L-i+1$. But if $i>L$, then the length of a busy
cycle coincides with the length of a busy period of the standard
$M/G/1$ queue, the service time of customers which is $B_2(x)$.
Then, the recurrence relations for the mean busy periods are
understandable.

Let us now consider the standard $M^X/G/1$ queueing system (without
threshold). The structure of a busy period for this queueing system
is given in Figure 3.

\begin{figure}
\includegraphics[width=10cm, height=15cm]{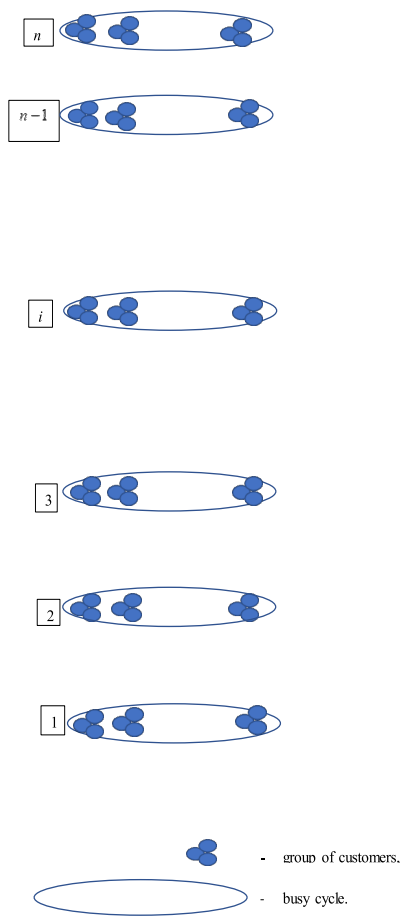}
\caption{General scheme of a busy period in the $M^X/G/1$ queue.}
\end{figure}

As we can see from the comparison of Figures 1 and 3, the only
difference between them is that in the case of the system with batch
arrivals, the elements of busy cycles that are indicated in Figure
3, are groups of customers rather than isolated customers as in
Figure 1. That is, the number of cycles indicated in Figure 3 is
associated with the number of batches arrived during the service
time of the first batch of arrived customers.

However, in the case
of the present queueing system that is with batch arrivals and
threshold, the structure of a busy period is much more complicated
than that in the previous cases. We consider the state-dependent queueing system with batch arrivals in the following formulation. The arrival of a batch is Poisson. If immediately before a service start the queue-length is not greater than $L$, then the probability distribution function of the customer is $B_1(x)$. Otherwise, it is $B_2(x)$. The description of the system implies that the first customer in a busy period is served with probability distribution function $B_1(x)$. This situation makes the system artificial. In the more natural situation, when the service depends on queue-length \textit{at the moment of a service start} rather than \textit{immediately before the service start}  the structure of the process is more complicated, and its analysis is technically harder. However, the asymptotic behavior, as $L$ increases to infinity, is the same. To avoid the technical complications, the problem is considered in the aforementioned simplified formulation.

Then, the linear
representations are similar to those derived for the state dependent
queueing system with ordinary Poisson input. Indeed, equations
\eqref{MG1.2} -- \eqref{MG1.5} all hold in the case of the present
queueing system as well. The first two, \eqref{MG1.2} and
\eqref{MG1.3} are obvious, and \eqref{MG1.4} and \eqref{MG1.5}
follow from the same Wald's identities as in the case of ordinary
Poisson arrivals. However, instead of \eqref{MG1.5.1} given in
Section \ref{MG1}, the relation between $\mathsf{E}T_L$ and
$\mathsf{E}\nu_L$ in the case of batch arrivals is slightly different. Specifically,
\begin{equation}\label{MXG1.0}
\lambda\mathsf{E}\varsigma\left(\mathsf{E}T_L+\frac{1}{\lambda}\right)=\lambda\mathsf{E}\varsigma\mathsf{E}T_L+\mathsf{E}\varsigma=\mathsf{E}\nu_L.
\end{equation}
Note, that the left-hand side of \eqref{MXG1.0} can be rewritten in the different way:
\begin{equation}\label{MXG1.9}
\begin{aligned}
&\lambda\mathsf{E}\varsigma\left(\mathsf{E}T_L-\frac{1}{\mu_1}+\frac{1}{\mu_1}+\frac{1}{\lambda}\right)\\
&=\lambda\mathsf{E}\varsigma\left(\mathsf{E}T_L-\frac{1}{\mu_1}\right)+\rho_1+\mathsf{E}\varsigma_1,
\end{aligned}
\end{equation}
where $\varsigma_1$ is the first batch that starts a busy period.
From \eqref{MXG1.9} we have
\begin{equation}\label{MXG1.10}
\lambda\mathsf{E}\varsigma\left(\mathsf{E}T_L-\frac{1}{\mu_1}\right)+(\rho_1-1+\mathsf{E}\varsigma_1)=\mathsf{E}\nu_L-1.
\end{equation}
The meaning of the quantity $\rho_1-1+\mathsf{E}\varsigma_1$ on the left-hand side of \eqref{MXG1.10} is the expected number of customers in the system after the service completion of the first customer in the busy period. That is, the expected number of independent busy cycles after the service completion of the first customer in the busy period is
\begin{equation}\label{MXG1.11}
\mathsf{E}\zeta_1=\rho_1-1+\mathsf{E}\varsigma_1.
\end{equation}

The main difficulty, however, is that the recurrence relation for
$\mathsf{E}T_L^{(1)}$ (or that for the corresponding quantity
$\mathsf{E}\nu_L^{(1)}$) cannot be presented as a convolution type recurrence
relation in simple terms as \eqref{MG1.1.1}, since, as it was indicated, the
structure of a busy period in the case of batch arrivals and
threshold becomes very complicated, and is not quite similar to that
given in Figure 2. This is because the size of the first batch is
random, and this is essentially affected to the complexity. However,
under the assumption that the first batch that starts a busy period contains
a single customer only, the structure of the busy period will become
similar to that given in Figure 2.

In following Figure 4, the typical structure of such busy period is
indicated.

\begin{figure}
\includegraphics[width=10cm, height=15cm]{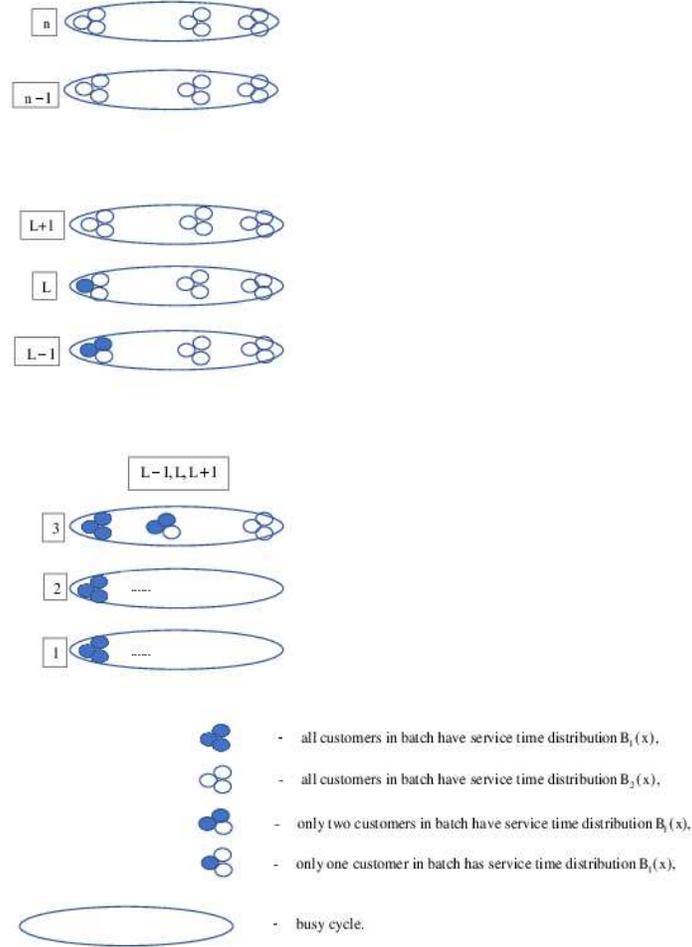}
\caption{Particular scheme of a busy period in the state-dependent
$M^X/G/1$ queue when the first batch contains only one customer.
($L$ is the threshold level.)}
\end{figure}

The number of cycles that are indicated in Figure 4 coincides with
the number of customers arrived during the service time of the first
customer. For instance, suppose that during the service time of the
first customer, three batches of customers arrived. If the
corresponding numbers of customers in those batches are 2, 1 and 3,
then the total number of cycles is 6. Each of these 6 customers is
considered as a tagged customer in the corresponding cycle, and the
structure of the busy period becomes similar to that in Figure 2.

\subsubsection{Analysis of the busy period} We first start from the
structure indicated in Figure 4. For this model, let $\widetilde
T_j$, $j=1,2,\ldots, L$, denote the time interval starting from a
moment when there are $L-j+1$ customers in the system until the
moment when there remain $L-j$ customers for the first time since
beginning of that interval. Similarly to the notation used in
Section \ref{MG1}, we introduce the random variables
$\widetilde{T}_j^{(1)}$, $\widetilde{T}_j^{(2)}$, $\widetilde
\nu_j$, $\widetilde{\nu}_j^{(1)}$, $\widetilde{\nu}_j^{(2)}$,
$j=1,2,\ldots, L$, which have the same meaning as before.
Specifically, when $j$ takes the value $L$, $\widetilde T_L$ is the
length a busy period starting from a single customer (1-busy
period); $\widetilde\nu_L$ is the number of customers that served
during a 1-busy period, and so on.

Comparison of the busy periods structures  in Figures 2 and 4
enables us to conclude that Tak\'acs' method \cite{Takacs 1955},
\cite{Takacs 1962} that was applied previously to the $M/G/1$
state-dependent queueing system is applicable to the state-dependent
$M^X/G/1$ queueing system, in which the first batch contains a
single customer. Based on the aforementioned Tak\'{a}cs' method, the
recurrence relation similar to that of \eqref{MG1.1} is
\begin{equation}\label{MXG1.1}
\mathsf{E}\widetilde{T}_L^{(1)}=\sum_{i=0}^{L}\mathsf{E}\widetilde{T}_{L-i+1}^{(1)}\int_0^\infty
\frac{1}{i!}\frac{\mathrm{d}^{i}f_x(z)}{\mathrm{d}z^i}\Big|_{z=0}
\mathrm{d}B_1(x),
\end{equation}
where $\mathsf{E}\widetilde{T}_0^{(1)}={1}/{\mu_1}$, and the
generating function $f_x(z)$ is given by \eqref{pgfArrival}.

So, the only difference between \eqref{MG1.1} and \eqref{MXG1.1} is
in their integrands on the right-hand side of \eqref{MXG1.1}, and in
the particular case $r_1=1$, $r_i=0$, $i\geq2$ we clearly arrive at
the same expression as \eqref{MG1.1}.

The explicit results associated with recurrence relation
\eqref{MXG1.1} is given later in the paper. Apparently, the similar
system of equations as \eqref{MG1.2} -- \eqref{MG1.5} is satisfied
for the characteristics of the state dependent queueing system
$M^X/G/1$. Namely,
\begin{eqnarray}
\mathsf{E}\widetilde T_L&=&\mathsf{E}\widetilde T_{L}^{(1)}+\mathsf{E}\widetilde T_{L}^{(2)},\label{MXG1.2}\\
\mathsf{E}\widetilde\nu_L&=&\mathsf{E}\widetilde\nu_{L}^{(1)}+\mathsf{E}\widetilde\nu_{L}^{(2)},\label{MXG1.3}
\end{eqnarray}
\begin{eqnarray}
\mathsf{E}\widetilde T_L^{(1)}&=&\frac{1}{\mu_1}\mathsf{E}\widetilde\nu_L^{(1)},\label{MXG1.4}\\
\mathsf{E}\widetilde
T_L^{(2)}&=&\frac{1}{\mu_2}\mathsf{E}\widetilde\nu_L^{(2)}.\label{MXG1.5}
\end{eqnarray}
Therefore, the same linear representations via
$\mathsf{E}\widetilde{T}_L^{(1)}$ hold for characteristics of these
systems, where by $\rho_1$ and $\rho_2$ one now should mean the
expected numbers of \textit{arrived customers} per service time (not
the expected number of arrivals) having the probability distribution
function $B_1(x)$ and, respectively, $B_2(x)$.
In other words, for all $L=1,2,\ldots$, we have $\mathsf{E}\widetilde{T}_L=a+b\mathsf{E}\widetilde{T}_L^{(1)}$ with
\begin{equation}\label{MXG2.0}
a=\frac{\rho_2}{\lambda(1-\rho_2)}, \quad
b=\frac{\rho_1-\rho_2}{\rho_1(1-\rho_2)}.
\end{equation}

Let us now consider the length of a busy period $T_L$ and associated
random variables $T_{L}^{(1)}$, $T_{L}^{(2)}$, $\nu_L$,
$\nu_{L}^{(1)}$ and $\nu_{L}^{(2)}$. In the following consideration, all these characteristics are associated with queueing models, in which a batch size that starts a busy period (initial batch size) can be different. The original batch size that starts a busy period, $\varsigma_1$, has the distribution $\mathsf{Pr}\{\varsigma_1=i\}=r_i$. Then, the random variable $\zeta_1=\kappa_1-1+\varsigma_1$ is the total number of customers in the system after the service completion of the first customer in a busy period (or the number of independent busy cycles after the service completion of the first customer in a busy period), where the random variable $\kappa_1$ denotes the total number of customers arrived during the service time of the first customer in a busy period. Recall (see relation \eqref{MXG1.11}) that $\mathsf{E}\zeta_1=\rho_1-1+\mathsf{E}\varsigma_1$.

Based on this, a busy period will be denoted $T_L(\zeta_1)$ and the basic characteristics of the queueing system associated with the busy period will be denoted $T_{L}^{(1)}(\zeta_1)$, $T_{L}^{(2)}(\zeta_1)$, $\nu_L(\zeta_1)$,
$\nu_{L}^{(1)}(\zeta_1)$ and $\nu_{L}^{(2)}(\zeta_1)$.

Another value of the initial characteristic of these random functions that is considered below is $\zeta_1\wedge L$, where $a\wedge b$ denotes $\min\{a,b\}$. That is, instead of the argument $\zeta_1$ in the random functions we will consider the argument $\zeta_1\wedge L$, by restricting the space of possible events in which
$\varsigma_1+\kappa_1\leq L+1$. Then, the queueing models with different initial characteristics $\zeta_1$ and $\zeta_1\wedge L$ are assumed to be given on the same probability space, and the corresponding notation for the characteristics of queueing system, in which that argument is
$\zeta_1\wedge L$, is $T_L(\zeta_1\wedge L)$, $T_{L}^{(1)}(\zeta_1\wedge L)$, $T_{L}^{(2)}(\zeta_1\wedge L)$, $\nu_L(\zeta_1\wedge L)$,
$\nu_{L}^{(1)}(\zeta_1\wedge L)$ and $\nu_{L}^{(2)}(\zeta_1\wedge L)$.

The busy period $T_L(\zeta_1)$ can be represented
\begin{equation}\label{MXG1.6}
T_L(\zeta_1){\buildrel d\over =}\chi_1+\sum_{i=1}^{\zeta_1\wedge
L}\widetilde{T}_{L-i+1}+\sum_{i=1}^{\zeta_1-L}\widetilde{T}_{0,i},
\end{equation}

where $\chi_1$ is the service time of the first customer;

\smallskip
1-busy periods $\widetilde{T}_{L-i+1}$, $i=1,2,\ldots,L$ are
mutually independent, and $\varsigma_1$ and $\kappa_1$ are independent of these 1-busy periods; hence, $\zeta_1$ is also independent of the aforementioned 1-busy periods;

\smallskip
$\widetilde{T}_{0,i}$, $i=1,2,\ldots$, is a sequence of independent
and identically distributed 1-busy periods of the $M^X/G/1$ queueing
system, the service times of which all are independent and
identically distributed random variables having the probability
distribution function $B_2(x)$, and the distributions of
interarrival times and batch sizes are the same as in the original
state dependent queueing system;

\smallskip
${\buildrel d\over =}$ denotes the equality in distribution;

\smallskip
in the case where $\zeta_1-L\leq0$, the empty sum in
\eqref{MXG1.6} is assumed to be zero.

\medskip
In turn, the representation for $T_L^{(1)}(\zeta_1)$ is as follows:
\begin{equation}\label{MXG1.7}
T_L^{(1)}(\zeta_1){\buildrel d\over=}\chi_1+\sum_{i=1}^{\zeta_1\wedge
L}\widetilde{T}_{L-i+1}^{(1)}.
\end{equation}
Notice, that along with \eqref{MXG1.7} we also have
\begin{equation}\label{MXG1.7*}
T_L^{(1)}(\zeta_1\wedge L){\buildrel d\over=}\chi_1+\sum_{i=1}^{(\zeta_1\wedge
L)\wedge L}\widetilde{T}_{L-i+1}^{(1)}=\chi_1+\sum_{i=1}^{\zeta_1\wedge
L}\widetilde{T}_{L-i+1}^{(1)}.
\end{equation}
That is, $T_L^{(1)}(\zeta_1)$ and $T_L^{(1)}(\zeta_1\wedge L)$ coincide in distribution.

Whereas $\mathsf{E}\widetilde{T}_L^{(1)}(\zeta_1)$ is determined by
recurrence relation \eqref{MXG1.1}, which is a particular case of
\eqref{MG1.1.1},  the convolution type recurrence relation, as it is
mentioned in Section \ref{S2.2.2}, is no longer valid for
$\mathsf{E}T_L^{(1)}(\zeta_1)$.

\subsubsection{Techniques of asymptotic analysis for
$\mathsf{E}T_L^{(1)}$ when $L$ large and associated characteristics}
For the following asymptotic analysis of $\mathsf{E}T_L^{(1)}$ and
other mean characteristics such as $\mathsf{E}\nu_L^{(1)}$,
$\mathsf{E}\nu_L^{(2)}$ we will use the following techniques. Let
$\mathcal{F}_L$ denotes the $\sigma$-algebra of the random variable
$\zeta_1\wedge L$. Then we have an increasing family of
$\sigma$-algebras
$\mathcal{F}_1\subset\mathcal{F}_2\subset\ldots\subset\mathcal{F}$,
where $\mathcal{F}$ is the $\sigma$-algebra of the random variable
$\zeta_1$.

Apparently, $\mathsf{Pr}\{\zeta_1\wedge
n>N\}\leq\mathsf{Pr}\{\zeta_1\wedge (n+1)>N\}$ for all
$n=1,2,\ldots$ and any fixed $N$, and hence,
$$
\lim_{L\to\infty}\mathsf{Pr}\{\zeta_1\wedge L\leq
N\}=\mathsf{Pr}\{\zeta_1\leq N\}
$$
and consequently, for all $n=1,2,\ldots$ we have
$\mathsf{E}\{\zeta_1\wedge n\}\leq\mathsf{E}\{\zeta_1\wedge
(n+1)\}$, and consequently
$$
\lim_{L\to\infty}\mathsf{E}\{\zeta_1\wedge
L\}=\mathsf{E}\zeta_1.
$$

From \eqref{MXG1.7} and \eqref{MXG1.7*} we have
\begin{equation}\label{MXG1.14}
\mathsf{E}T_L^{(1)}(\zeta_1\wedge L)=\mathsf{E}T_L^{(1)}(\zeta_1)
\end{equation}
for any $L\geq 1$.

Based on \eqref{MXG1.14} let us now study the mean characteristics $\mathsf{E}T_L(\zeta_1\wedge L)$, $\mathsf{E}T_L^{(1)}(\zeta_1\wedge L)$, $\mathsf{E}T_L^{(2)}(\zeta_1\wedge L)$, $\mathsf{E}\nu_L^{(1)}(\zeta_1\wedge L)$ and $\mathsf{E}\nu_L^{(2)}(\zeta_1\wedge L)$ showing first the justice of the linear representations that are similar to those \eqref{MG1.5.2} and \eqref{MG1.5.3}. Writing $\mathsf{E}\widetilde{T}_i=a+b\mathsf{E}\widetilde{T}_i^{(1)}$, $i=1,2,\ldots,L$, where $a$ and $b$ are specified constants, by the total expectation formula we obtain:
\begin{equation*}\label{MXG1.8*}
\begin{aligned}
\mathsf{E}T_L(\zeta_1\wedge L)&=\mathsf{EE}\{T_L(\zeta_1\wedge L)|\mathcal{F}_L\}\\
&=\frac{1}{\mu_1}+\sum_{i=1}^L\mathsf{Pr}\{\zeta_1\wedge L=i\}\sum_{j=1}^{i}\mathsf{E}\widetilde T_{L-j+1}\\
&=\frac{1}{\mu_1}+\sum_{i=1}^L\mathsf{Pr}\{\zeta_1\wedge L=i\}\sum_{j=1}^i(a+b\mathsf{E}\widetilde T_{L-i+1}^{(1)})\\
&=\frac{1}{\mu_1}+a\sum_{i=1}^L i\mathsf{Pr}\{\zeta_1\wedge L=i\}\\
&\ \ \ \ \ \ \ \ +b\sum_{i=1}^L
\mathsf{Pr}\{\zeta_1\wedge L=i\}
\sum_{j=1}^i\mathsf{E}\widetilde T_{L-i+1}^{(1)}\\
&=\frac{1}{\mu_1}+a\mathsf{E}(\zeta_1\wedge
L)+b\mathsf{EE}\{T_{L}^{(1)}(\zeta_1\wedge L)|\mathcal{F}_L\}\\
&=\frac{1}{\mu_1}+a\mathsf{E}(\zeta_1\wedge L)+b\mathsf{E}T_{L}^{(1)}(\zeta_1\wedge L).
\end{aligned}
\end{equation*}

Hence, we obtained the representation
\begin{equation}\label{MXG1.8}
\mathsf{E}T_L(\zeta_1\wedge L) =\frac{1}{\mu_1}+a\mathsf{E}(\zeta_1\wedge
L)+b\mathsf{E}T_{L}^{(1)}(\zeta_1\wedge L).
\end{equation}
Keeping in mind \eqref{MXG2.0} we obtain
\begin{equation*}
\mathsf{E}T_L(\zeta_1\wedge L) =\frac{1}{\mu_1}+\frac{\rho_2}{\lambda(1-\rho_2)}\mathsf{E}(\zeta_1\wedge
L)+\frac{\rho_1-\rho_2}{\rho_1(1-\rho_2)}\mathsf{E}T_{L}^{(1)}(\zeta_1\wedge L).
\end{equation*}

\section{Main result}\label{Main results}
In this section, we formulate the main result of this paper. The main result of the paper is based on heavy-traffic conditions. They are motivated by the fact (mentioned later in Remark \ref{R1} and based on the statement of Theorem \ref{M asymp}) that the only case $\rho_1=1$ gives finite limit of the functional $J(L)$, as $L$ increases to infinity. In all other cases where $\rho_1$ is fixed, the functional is not bounded in limit, and only in the cases where $L[\rho_1(L)-1]$ converges to finite limit, that is positive, negative or zero, may give the optimal solution to the control problem.

So, we
start from the heavy traffic conditions and explicit representations
for the objective function under these conditions.

\subsection{Heavy traffic conditions}

\begin{cond}\label{cond1} Assume that
$L[\rho_1(L)-1]\to C>0$, as $L\to\infty$. Assume also that
$\rho_{1,3}(L)$ is a bounded sequence,
$\mathsf{E}\varsigma^{3}<\infty$ and the limit
$\lim_{L\to\infty}\rho_{1,2}(L)=\widetilde\rho_{1,2}$ exists.
\end{cond}

\begin{cond}\label{cond2}Assume
that $L[\rho_1(L)-1]\to C<0$, as $L\to\infty$. Assume also that
$\rho_{1,3}(L)$ is a bounded sequence,
$\mathsf{E}\varsigma^{3}<\infty$ and the limit
$\lim_{L\to\infty}\rho_{1,2}(L)=\widetilde\rho_{1,2}$ exists.
\end{cond}

\begin{rem}
The case when $C=0$ is also considered and is related to both of these conditions.
\end{rem}

Let
$\widehat{B}_1(s)=\int_0^\infty\mathrm{e}^{-sx}\mathrm{d}B_1(x)$,
$s\geq0$.

\begin{cond}\label{cond3}
Under Condition \ref{cond2} let $\delta(L)=1-\rho_1(L)$, and assume that there exists $\delta_0>0$
such that for all $\delta(L)<\delta_0$ as $L\to\infty$, each of the
functional equations $z=\widehat{B}_1(\lambda-\lambda z)$ (depending
on the parameter $\delta$) has a unique solution in the interval
$(1, \infty)$.
\end{cond}

\begin{rem} Conditions \ref{cond1} and \ref{cond2} contain some
technical assumptions such as $\rho_{1,3}(L)<\infty$,
$\mathsf{E}\varsigma^{3}<\infty$ and the existence of the limit
$\lim_{L\to\infty}\rho_{1,2}(L)=\widetilde\rho_{1,2}$.    The aforementioned assumptions
are originated from the analytic approach that is based
on Taylor's expansion and application of Tauberian theorems. We do
think that these assumptions can be avoided by using the direct
approach that considers a diffusion approximation of the transient
dam process and computes then the stationary distribution of the
diffusion. On this way, these technical assumptions can be avoided,
however one would have deal with interchange of limits (large $L$
vs. large $t$). This problem is very hard in general (see discussion of a similar problem in Whitt \cite{Whitt 2004}).
For solution of the limits interchange problem in generalized Jackson networks in heavy traffic see Gamarnik and Zeevi \cite{GZ 2006} and Braverman, Dai and Miyazawa \cite{BDM 2017}. The further references can be found in \cite{BDM 2017}.

Condition \ref{cond3} is originated from an application of the
analytic method of \cite{Willmot 1988}. It requires to consider the
class of probability distributions, the Laplace-Stieljes transform
of which is analytic in some negative area of $\mathrm{Re}s$ and use
Taylor's expansion for small values of $\delta$. This class of
distributions is smaller than that under Conditions \ref{cond1} or
\ref{cond2} and implies the existence of all moments of the
distributions.
\end{rem}

\subsection{Series of objective functions}\label{S3.2}

Let $\widehat C_L(z)=\sum_{j=0}^{L-1} c_{L-j}z^j$ denote a
backward generating cost function, and let
$$
C(L)=L[\rho_1(L)-1]
$$
be the function of $L$.

We introduce the following series of objective functions corresponding to
the cases
$$
 C(L)>0,\leqno(\mathrm{i})
$$
$$
C(L)<0,\leqno(\mathrm{ii})
$$
and
$$
C(L)=0,\leqno(\mathrm{iii})
$$
which are subject to minimization.

Below we define three series of objective functions
$J^{\mathrm{upper}}[L,C(L)]$ and $J^{\mathrm{lower}}[L,C(L)]$ and $J^0(L,C(L))$
depending on large parameter $L$ and the objective function in the aforementioned marginal case.  The first series,
$J^{\mathrm{upper}}[L,C(L)]$, is associated with condition (i),
the second one, $J^{\mathrm{lower}}[L,C(L)]$, with condition (ii) and the last case is associated with condition (iii).

The problem is to find a function $\rho_1(L)$ under which, as
$L\to\infty$, the functionals $J^{\mathrm{upper}}[L,C(L)]$ or
$J^{\mathrm{lower}}[L,C(L)]$ (in dependence which of the conditions
is satisfied) converges in limit, as $L\to\infty$, to minimum. If $\rho_1(L)$ converges to 1, then we arrive at the limiting case associated with (iii), where $\widetilde\rho_{1,2}=\lim_{L\to\infty}\rho_{1,2}(L)$ and $c^*=\lim_{L\to\infty}c^0(L)$ (see also Remark \ref{R2} below).

\subsubsection{Series of objective functions corresponding to the case (i)}

\begin{equation}\label{SP11.4*}
\begin{aligned}
&J^{\mathrm{upper}}[L,C(L)]\\
&=C(L)\left[j_1\frac{1}{\exp\left(\frac{2C(L)
\mathsf{E}\varsigma}{{\rho}_{1,2}(L)(\mathsf{E}\varsigma)^{3}+\mathsf{E}\varsigma^{2}-\mathsf{E}\varsigma}\right)-1}\right.\\
&\ \ \
\left.+j_2\frac{\rho_2\exp\left(\frac{2C(L)\mathsf{E}\varsigma}{{\rho}_{1,2}(L)(\mathsf{E}\varsigma)^{3}+\mathsf{E}\varsigma^{2}-\mathsf{E}\varsigma}\right)}
{(1-\rho_2)\left({\exp\left(\frac{2C(L)\mathsf{E}\varsigma}{{\rho}_{1,2}(L)
(\mathsf{E}\varsigma)^{3}+\mathsf{E}\varsigma^{2}-\mathsf{E}\varsigma}\right)-1}\right)}\right]\\
&\ \ \ +c^{\mathrm{upper}}[L,C(L)],
\end{aligned}
\end{equation}
where
\begin{equation}\label{SP11.5*}
\begin{aligned}
&c^{\mathrm{upper}}[L,C(L)]\\
&=\frac{2C(L)\mathsf{E}\varsigma}
{{\rho}_{1,2}(L)(\mathsf{E}\varsigma)^{3}+\mathsf{E}\varsigma^{2}-\mathsf{E}\varsigma}\cdot
\frac{\exp\left(\frac{2C(L)\mathsf{E}\varsigma}
{{\rho}_{1,2}(L)(\mathsf{E}\varsigma)^{3}+\mathsf{E}\varsigma^{2}-\mathsf{E}\varsigma}\right)}
{\exp\left(\frac{2C(L)\mathsf{E}\varsigma}{{\rho}_{1,2}(L)(\mathsf{E}\varsigma)^{3}+\mathsf{E}
\varsigma^{2}-\mathsf{E}\varsigma}\right)-1}\\
&\ \ \
\times\frac{1}{L}~\widehat{C}_L\left(1-\frac{2C(L)\mathsf{E}\varsigma}
{[{\rho}_{1,2}(L)(\mathsf{E}\varsigma)^{3}
+\mathsf{E}\varsigma^{2}-\mathsf{E}\varsigma]L}\right).
\end{aligned}
\end{equation}

The series of objective functions given by \eqref{SP11.4*} and \eqref{SP11.5*} is used in Proposition \ref{prop2}.

\subsubsection{Series of objective functions corresponding to the case (ii)}

\begin{equation}\label{SP12.1*}
\begin{aligned}
&J^{\mathrm{lower}}[L,C(L)]\\
&=-C(L)\left[j_1\exp\left(-\frac{{\rho}_{1,2}(L)
(\mathsf{E}\varsigma)^{3}+\mathsf{E}\varsigma^{2}-\mathsf{E}\varsigma}
{2C(L)\mathsf{E}\varsigma}\right)\right.\\
&\ \ \ \left.+j_2\frac{\rho_2}{1-\rho_2}
\left(\exp\left(-\frac{{\rho}_{1,2}(L)(\mathsf{E}\varsigma)^{3}+\mathsf{E}\varsigma^{2}-\mathsf{E}\varsigma}
{2C(L)\mathsf{E}\varsigma}\right)-1\right)\right]\\
&\ \ \ +c^{\mathrm{lower}}[L,C(L)],
\end{aligned}
\end{equation}
where
\begin{equation}\label{SP12.2*}
\begin{aligned}
&c^{\mathrm{lower}}[L,C(L)]\\
&=-\frac{2C(L)\mathsf{E}\varsigma}
{{\rho}_{1,2}(L)(\mathsf{E}\varsigma)^{3}+\mathsf{E}\varsigma^{2}-\mathsf{E}\varsigma}
\cdot\frac{1}{\exp\left(-\frac{2C(L)\mathsf{E}\varsigma}
{{\rho}_{1,2}(L)(\mathsf{E}\varsigma)^{3}+\mathsf{E}\varsigma^{2}-\mathsf{E}\varsigma}\right)-1}\\
&\ \ \ \times
\frac{1}{L}~\widehat{C}_{L}\left(1-\frac{2C(L)\mathsf{E}\varsigma}
{[{\rho}_{1,2}(L)(\mathsf{E}\varsigma)^{3}+
\mathsf{E}\varsigma^{2}-\mathsf{E}\varsigma]L}\right).
\end{aligned}
\end{equation}

The series of objective functions given by \eqref{SP12.1*} and \eqref{SP12.2*} is used in Proposition \ref{prop3}.

\subsubsection{Series of objective functions corresponding to the case (iii)}
\begin{equation}\label{SP12.3*}
\begin{aligned}
J^0(L)=&j_1\frac{\rho_{1,2}(L)(\mathsf{E}\varsigma)^3
+\mathsf{E}\varsigma^2-\mathsf{E}\varsigma}{2}\\
&+j_2\frac{\rho_2}{1-\rho_2}
\frac{\rho_{1,2}(L)(\mathsf{E}\varsigma)^3+\mathsf{E}\varsigma^2-\mathsf{E}\varsigma}{2}\\
&+c^0(L),
\end{aligned}
\end{equation}
where
$$
c^0(L)=\frac{1}{L}\sum_{i=1}^{L}c_i
$$

\begin{rem}\label{R2}
The series of the second moments $\rho_{1,2}(L)$ that are used in
\eqref{SP11.4*}, \eqref{SP11.5*}, \eqref{SP12.1*},
\eqref{SP12.2*} and \eqref{SP12.3*} is assumed to converge to the limit
$\widetilde{\rho}_{1,2}$.
\end{rem}

\subsection{Formulation of the main result}

\begin{thm}\label{thm3*}
Under the assumption that the costs $c_i$ are nonincreasing, and
under Conditions \ref{cond1}, \ref{cond2} and \ref{cond3}, a
solution to the control problem do exist and unique in the sense
explained below. The solution to the control problem in \eqref{I1} -- \eqref{I4} is defined as
follows.

Let $\overline{J}$ be the minimum value of the possible limits
$$\overline{J}^{\mathrm{upper}}(C)=\lim_{L\to\infty}J^{\mathrm{upper}}[L,C(L)]$$
for the series of objective function $J^{\mathrm{upper}}[L,C(L)]$ defined in
\eqref{SP11.4*} and \eqref{SP11.5*} and, respectively, let
$\underline{J}$ be the minimum value of the possible limits
$$\underline{J}^{\mathrm{lower}}(C)=\lim_{L\to\infty}J^{\mathrm{lower}}[L,C(L)]$$
or $$\underline{J}^0(0)=\lim_{L\to\infty}J^{0}(L)$$
for the series of objective functions
$J^{\mathrm{lower}}[L,C(L)]$ defined in \eqref{SP12.1*} and
\eqref{SP12.2*} or $J^0(L)$ defined in \eqref{SP12.3*}. Then there is a function $\rho_1(L)$ satisfying the
following properties.

If $j_1>j_2\rho_2/(1-\rho_2)$ is satisfied, then only for a positive
limit
$$
\lim_{L\to\infty}L[\rho_1(L)-1]=\overline{C}>0,
$$
the optimal value of the objective function, $\overline{J}$, is
reached.

Otherwise, for the optimal value of the objective function, the limit
$$
\lim_{L\to\infty}L[\rho_1(L)-1]
$$
can be positive, negative or zero.
\end{thm}

\section{Asymptotic theorems for the stationary probabilities $p_1$ and $p_2$}\label{Stationary
probabilities}

In this section, the explicit expressions are derived for the
stationary probabilities, and their asymptotic behavior is studied.
These results will be used in our further findings of the optimal
solution.

\subsection{Preliminaries}\label{Preliminaries}

In this section we recall the main properties of recurrence relation
\eqref{MG1.1.1}. The detailed theory of these recurrence relations
can be found in Tak\'{a}cs \cite{Takacs 1967}. For the generating
function $Q(z)=\sum_{j=0}^\infty Q_jz^j$, $|z|\leq1$, we have
\begin{equation}\label{SP4+1}
Q(z)=\frac{Q_0 F(z)}{F(z)-z},
\end{equation}
where $F(z)=\sum_{j=0}^\infty f_jz^j$.

Asymptotic behavior of $Q_n$ as $n\to\infty$ has been studied by
Tak\'{a}cs \cite{Takacs 1967} and Postnikov \cite{Postnikov 1980}.
Recall the theorems that are needed in this paper.

Denote
$\gamma_m=\lim_{z\uparrow1}{\mathrm{d}^mF(z)}/{\mathrm{d}z^m}$.

\begin{lem}\label{lem.T} \texttt{(Tak\'acs \cite{Takacs 1967}, p.22-23).} If $\gamma_1<1$ then
\begin{equation}\label{T.1}
\lim_{n\to\infty} Q_n=\frac{Q_0}{1-\gamma_1}.
\end{equation}
If $\gamma_1=1$ and $\gamma_2<\infty$, then
$$
\lim_{n\to\infty}\frac{Q_n}{n}=\frac{2Q_0}{\gamma_2}.
$$
If $\gamma_1>1$, then
\begin{equation}\label{T.3}
\lim_{n\to\infty}\left(Q_n-\frac{Q_0}{\delta^n[1-F^\prime(\delta)]}\right)=\frac{Q_0}{1-\gamma_1},
\end{equation}
where $\delta$ is the least in absolute value root of the functional
equation $z=F(z)$ and $F^\prime(z)$ is the derivative of $F(z)$.
\end{lem}

\begin{lem}\label{lem.P} \texttt{(Postnikov \cite{Postnikov 1980}, Sect.25).} Let $\gamma_1=1$, $\gamma_2<\infty$ and $f_0+f_1<1$. Then, as $n\to\infty$,
\begin{equation}\label{P.1}
Q_{n+1}-Q_n=\frac{2Q_0}{\gamma_2}+o(1).
\end{equation}
\end{lem}

\subsection{Extension of Tak\'acs' lemma}
For the following considerations we also need in extended version of Tak\'acs' Lemma \ref{lem.T} in the case when $\gamma_1>1$.

Let $Q_n(L)$ be a series of number satisfying for each $L$ the system of recurrence relations
\begin{equation}\label{eq.0.20}
Q_n(L)=\sum_{i=0}^{n}Q_{n-i+1}(L)f_{n}(L),
\end{equation}
where $Q_0(L)$ is an arbitrary positive number, and $\sum_{i=0}^{\infty}f_i(L)=1$, $f_i(L)\geq0$, and let
\begin{equation}\label{eq.0.22}
f_{n}^*=\lim_{L\to\infty}f_{n}(L)
\end{equation}
exist.
So, \eqref{eq.0.20} is an extended version of recurrence relations given by \eqref{MG1.1.1}, where the series parameter $L$ is added, with the limiting sequence given by \eqref{eq.0.22}.

Assume that for all $L$
$$
\gamma_1(L)=\sum_{i=1}^{\infty}if_i(L)>1.
$$
So, as $L$ increases to infinity, from \eqref{eq.0.22} we have
\begin{equation}\label{eq.0.21}
\lim_{L\to\infty}\gamma_1(L)=\gamma_1^*.
\end{equation}
Denote by $F_L(z)$ the series $F_L(z)=\sum_{i=0}^{\infty}f_i(L)z^i$, and by $\delta(L)$ the least in absolute value root of the functional equation $z=F_L(z)$.

\begin{lem}\label{lem9}
Assume that $\gamma_1^*>1$,
\begin{equation}\label{eq.0.23}
\lim_{L\to\infty}Q_0(L)=Q_0^*.
\end{equation}
Then,
\begin{equation}\label{eq.0.24}
\lim_{L\to\infty}\lim_{n\to\infty}Q_n(L)[\delta(L)]^n=\lim_{n\to\infty}Q_n^*(\delta^*)^n,
\end{equation}
where $Q_n^*=\lim_{L\to\infty}Q_n(L)$ and $\delta^*$ is the least in absolute value root of the functional equation $z=F^*(z)$, $F^*(z)=\lim_{L\to\infty}F_L(z)$.
\end{lem}

\begin{proof}
It follows from \eqref{eq.0.20} that $Q_n(L)$ are defined for all $L$, and there exists the limit $Q_n^*$ of $Q_n(L)$ as $L\to\infty$. As well, the sequence $\delta(L)$ converges to its limit $\delta^*$ which is, due to the assumption $\gamma_1^*>1$, strictly less than 1. The limit
$$
\lim_{n\to\infty}Q_n(L)[\delta(L)]^n
$$
is defined and, according to \eqref{T.3} of Lemma \ref{lem.T}, expressed via the quantity $Q_0(L)$ divided by $1-F_L^\prime[\delta(L)]$, where $F_L^\prime(z)$ is the derivative of $F_L(z)$. Hence, taking into account assumption \eqref{eq.0.23} and the convergence $F^*(z)=\lim_{L\to\infty}F_L(z)$, we obtain
\begin{equation}\label{eq.0.25}
\lim_{n\to\infty}\lim_{L\to\infty}Q_n(L)[\delta(L)]^n=\lim_{n\to\infty}Q_n^*(\delta^*)^n.
\end{equation}
The convergence in $L$, depending on $\gamma_1^*$ and $Q_0^*$ only,  is uniform.
Hence,
the Moore-Osgood theorem on interchanging order of limits is applicable, and \eqref{eq.0.24} follows from \eqref{eq.0.25}.
\end{proof}

\subsection{Exact formulae for $p_1$ and $p_2$}\label{Exact formulae}
In this section we derive the exact formulae for $p_1$ and $p_2$.
These formulae follow from the following two steps (see the proof of
Lemma \ref{Prel asymp}).

\smallskip
1. We establish the linear representations for $\mathsf{E}\nu_L^{(2)}(\zeta_1)$
in terms of $\mathsf{E}\nu_L^{(1)}(\zeta_1)$.

\smallskip
2. Then, the explicit formulae for $p_1$ and $p_2$ follow from
renewal reward theorem.

\begin{lem}\label{Prel asymp}
We have:
\begin{equation}\label{SP4.19}
p_1=\frac{(1-\rho_2)\mathsf{E}\zeta_1}{\mathsf{E}\zeta_1+(\rho_1-\rho_2)\left[\mathsf{E}\nu_L^{(1)}(\zeta_1)-1\right]},
\end{equation}
and
\begin{equation}\label{SP4.20}
p_2=\frac{\rho_2\mathsf{E}\zeta_1+\rho_2(\rho_1-1)\left[\mathsf{E}\nu_L^{(1)}(\zeta_1)-1\right]}
{\mathsf{E}\zeta_1+(\rho_1-\rho_2)\left[\mathsf{E}\nu_L^{(1)}(\zeta_1)-1\right]},
\end{equation}
 where $\rho_1$ and $\rho_2$ mean the load parameters of the system, that is, the
expected numbers of arrived customers per service time  having the probability distribution
function $B_1(x)$ and, respectively, $B_2(x)$.
\end{lem}

\begin{proof}
First, derive the linear representation of $\mathsf{E}\nu_L^{(2)}(\zeta_1)$
via $\mathsf{E}\nu_L^{(1)}(\zeta_1)$. From relation \eqref{MXG1.0} and
equations \eqref{MG1.2} -- \eqref{MG1.5}, which also hold true in
the case of the present queueing system with batch arrivals, we
obtain:
\begin{equation}\label{SP4.23}
\mathsf{E}\nu_L^{(2)}(\zeta_1)=\frac{\mathsf{E}\zeta_1}{1-\rho_2}-\frac{1-\rho_1}{1-\rho_2}\left[\mathsf{E}\nu_L^{(1)}(\zeta_1)-1\right].
\end{equation}
Using renewal arguments (e.g. \cite{Ross 2000}) and relation
\eqref{MXG1.0}, we have:
\begin{equation}\label{SP4.21}
p_1=\frac{\frac{1}{\lambda}}{\mathsf{E}T_L^{(1)}(\zeta_1)+\mathsf{E}T_L^{(2)}(\zeta_1)+\frac{1}{\lambda}}
=\frac{\mathsf{E}\zeta_1}
{\mathsf{E}\nu_L^{(1)}(\zeta_1)+\mathsf{E}\nu_L^{(2)}(\zeta_1)}
\end{equation}
and
\begin{equation}\label{SP4.22}
p_2=\frac{\mathsf{E}T_L^{(2)}(\zeta_1)}{\mathsf{E}T_L^{(1)}(\zeta_1)+\mathsf{E}T_L^{(2)}(\zeta_1)+\frac{1}{\lambda}}=
\frac{\rho_2\mathsf{E}\nu_L^{(2)}(\zeta_1)}{\mathsf{E}\nu_L^{(1)}(\zeta_1)+\mathsf{E}\nu_L^{(2)}(\zeta_1)}.
\end{equation}
Now, substituting \eqref{SP4.23} for the right sides of
\eqref{SP4.21} and \eqref{SP4.22} we obtain relations \eqref{SP4.19}
and \eqref{SP4.20} of this lemma.
\end{proof}

\subsection{Preliminary asymptotic
expansions for large $L$}\label{Preliminary asymp} The explicit
results for $p_1$ and $p_2$ that are established in Lemma \ref{Prel
asymp} are expressed via the unknown quantity
$\mathsf{E}\nu_L^{(1)}(\zeta_1)$. So, the aim is to find the asymptotic
behaviour of $\mathsf{E}\nu_L^{(1)}(\zeta_1)$ as $L$ increases to infinity
and thus to find the asymptotic behaviour of $p_1$ and $p_2$.

In this section we obtain some preliminary asymptotic
representations that follow from the explicit results. Those
asymptotic representations will be used in the sequel. The results
of this section are as follows.

\smallskip
1. We first establish the linear representation for
$\mathsf{E}\widetilde\nu_L^{(2)}$ in terms of
$\mathsf{E}\widetilde\nu_L^{(1)}$ (see Lemma \ref{representation}).

\smallskip
2. By similar way, we derive the representation for
$\mathsf{E}\nu_L^{(2)}(\zeta_1\wedge L)$ in terms of $\mathsf{E}\nu_L^{(1)}(\zeta_1\wedge L)$ (see
Lemma \ref{extended representation}).

\smallskip
3. Based on that representation, we prove that
$\mathsf{E}\nu_L^{(2)}(\zeta_1)-\mathsf{E}\nu_L^{(2)}(\zeta_1\wedge
L)=o(1)$ as $L\to\infty$ (see Lemma \ref{Asymp estimate}).

\begin{lem}\label{representation}
For $\mathsf{E}\widetilde{\nu}_L^{(2)}$, $L=1,2,\ldots$, we have the
following representation
\begin{equation}\label{SP4.0}
\mathsf{E}\widetilde{\nu}_L^{(2)}=\frac{\mathsf{E}\varsigma}{1-\rho_2}-\frac{1-\rho_1}{1-\rho_2}\mathsf{E}\widetilde\nu_L^{(1)},
\end{equation}
where $\rho_1={\lambda\mathsf{E}\varsigma}/{\mu_1}$ and
$\rho_2={\lambda\mathsf{E}\varsigma}/{\mu_2}<1$, and
$\mathsf{E}\widetilde\nu_L^{(1)}$ is given by
\begin{equation}\label{SP4.1}
\begin{aligned}
\mathsf{E}\widetilde{\nu}_L^{(1)}
&=\sum_{i=0}^{L}\mathsf{E}\widetilde{\nu}_{L-i+1}^{(1)}\int_0^\infty
\frac{1}{i!}\frac{\mathrm{d}^{i}f_x(z)}{\mathrm{d}z^i}\Big|_{z=0}
\mathrm{d}B_1(x),
\end{aligned}
\end{equation}
$\mathsf{E}\widetilde\nu_0^{(1)}=1$.
\end{lem}

\begin{proof}
Taking into account that the number of arrivals during 1-busy cycle
(1-busy period plus idle period) coincides with the number of
customers served during the same 1-busy period, according to Wald's
identity we have:
$$
\lambda\mathsf{E}\varsigma\left(\mathsf{E}\widetilde{T}_L+\frac{1}{\lambda}\right)
=\lambda\mathsf{E}\varsigma\mathsf{E}\widetilde{T}_L+\mathsf{E}\varsigma
=\mathsf{E}\widetilde\nu_L=\mathsf{E}\widetilde\nu_L^{(1)}+\mathsf{E}\widetilde\nu_L^{(2)}.
$$
This equality together with \eqref{MXG1.2} -- \eqref{MXG1.5} yields
the desired statement of the lemma, where \eqref{SP4.1} in turn
follows from \eqref{MXG1.1} and Wald's identity \eqref{MXG1.4}.
\end{proof}

The next step is to derive representations for
$\mathsf{E}\nu_L^{(1)}(\zeta_1\wedge L)$ and
$\mathsf{E}\nu_L^{(2)}(\zeta_1\wedge L)$. We have the following
lemma.
\begin{lem}\label{extended representation}
For $\mathsf{E}\nu_L^{(2)}(\zeta_1\wedge L)$ we have
\begin{equation}\label{SP4.2}
\mathsf{E}\nu_L^{(2)}(\zeta_1\wedge
L)=\frac{\mathsf{E}(\zeta_1\wedge
L)}{1-\rho_2}-\frac{1-\rho_1}{1-\rho_2}\left[\mathsf{E}\nu_L^{(1)}(\zeta_1\wedge
L)-1\right],
\end{equation}
where similarly to \eqref{MXG1.7*}
\begin{equation}\label{MXG1.15}
\mathsf{E}\nu_L^{(1)}(\zeta_1\wedge L)=1+\mathsf{E}\sum_{i=1}^{\zeta_1\wedge
L}\widetilde{\nu}_{L-i+1}^{(1)},
\end{equation}
and $\mathsf{E}\widetilde{\nu}_{L-i+1}^{(1)}$, $i=1,2,\ldots,L$, are
given by \eqref{SP4.1}.
\end{lem}

\begin{proof} Following the same arguments as in the proof of \eqref{MXG1.8},
one can write
\begin{equation}\label{MXG1.13}
\mathsf{E}\nu_L^{(2)}(\zeta_1\wedge L)
=a\mathsf{E}(\zeta_1\wedge
L)+b\left[\mathsf{E}\nu_L^{(1)}(\zeta_1\wedge L)-1\right]
\end{equation}
 for the specified
constants $a$ and $b$, for which the linear representation
$\mathsf{E}\widetilde\nu_L^{(2)}=a+b\mathsf{E}\widetilde\nu_L^{(1)}$
is satisfied. Hence, according to relation \eqref{SP4.0} of Lemma
\ref{representation}, $a={1}/({1-\rho_2})$ and
$b=-({1-\rho_1})/({1-\rho_2})$. The proof is completed.
\end{proof}

\smallskip
The following lemma yields an estimate for the difference
$\mathsf{E}\nu_L^{(2)}(\varsigma_1)-\mathsf{E}\nu_L^{(2)}(\varsigma_1\wedge L)$.

\begin{lem}\label{Asymp estimate} As $L\to\infty$,
\begin{equation}\label{SP4.18}
\mathsf{E}\nu_L^{(2)}(\zeta_1)-\mathsf{E}\nu_L^{(2)}(\zeta_1\wedge L)=o(1).
\end{equation}
\end{lem}

\begin{proof}
It follows from \eqref{MXG1.14} and Wald's identity that
$$
\mathsf{E}\nu_L^{(1)}=\mathsf{E}\nu_L^{(1)}(\zeta_1\wedge L).
$$
Hence, \eqref{MXG1.13} can be rewritten in the form
$$
\mathsf{E}\nu_L^{(2)}(\zeta_1\wedge L)
=\frac{1}{1-\rho_2}\mathsf{E}(\zeta_1\wedge
L)-\frac{1-\rho_1}{1-\rho_2}\left[\mathsf{E}\nu_L^{(1)}(\zeta_1)-1\right],
$$
and \eqref{SP4.18} follows.
\end{proof}

\subsection{Asymptotic theorems for $p_1$ and $p_2$ under `usual assumptions'}\label{Final asymp}
By `usual assumption' we mean the standard cases as $\rho_1<1$ or
$\rho_1>1$ for the asymptotic behaviour as $L\to\infty$. In the
following sections the heavy traffic assumptions are assumed. By the
heavy-traffic assumptions, we mean such a case where for some
positive constant $c$
\begin{equation}\label{SP5.15}
-c<\lim_{L\to\infty}L[\rho_1(L)-1]<c,
\end{equation}
and $\rho_2<1$.

 The main result of Section \ref{Exact formulae} is Lemma
\ref{Prel asymp}, where the stationary probabilities $p_1$ and $p_2$
are expressed explicitly via $\mathsf{E}\nu_L^{(1)}(\zeta_1)$. The aim of
this section is to obtain the analogue of asymptotic Theorem 3.1 of
\cite{Abramov 2007}. To this end, we do as follows.

\smallskip
1. We first study the asymptotic behavior of
$\mathsf{E}\widetilde\nu_L^{(1)}$ as $L\to\infty$. For this purpose
derive the representation for the generating function
$\sum_{j=0}^\infty\mathsf{E}\widetilde\nu_j^{(1)}z^j$. Then,
we obtain the asymptotic behaviour of
$\mathsf{E}\widetilde\nu_L^{(1)}$ as $L\to\infty$ by using Tak\'acs'
theorem (Lemma \ref{lem.T} or Lemma \ref{lem9} containing an extension of Tak\'acs' theorem) and Postnikov's theorem (Lemma
\ref{lem.P}).

\smallskip
2. Then, we derive asymptotic representation for
$\mathsf{E}\nu_L^{(1)}(\zeta_1\wedge L)$ as $L\to\infty$, and
on the basis of this representation and renewal reward theorem (e.g.
\cite{Ross 2000}) we find asymptotic behaviour of stationary
probabilities $p_1$ and $p_2$ as $L\to\infty$.

\smallskip
To derive the generative function
$\sum_{j=0}^\infty\mathsf{E}\widetilde\nu_j^{(1)}z^j$, we use
representation \eqref{SP4.1}. This yields:

\begin{equation}\label{SP5.0}
\begin{aligned}
\sum_{j=0}^\infty\mathsf{E}\widetilde\nu_j^{(1)}z^j&=\sum_{j=0}^\infty
z^j\sum_{i=0}^{j}
\mathsf{E}\widetilde{\nu}_{L-i+1}^{(1)}\int_0^\infty
\frac{1}{i!}\frac{\mathrm{d}^{i}f_x(u)}{\mathrm{d}u^i}\Big|_{u=0}
\mathrm{d}B_1(x)\\
&=\frac{U(z)}{U(z)-z},
\end{aligned}
\end{equation}
where
\begin{equation}\label{SP5.1}
\begin{aligned}
U(z)&=\int_0^{\infty}\exp\left\{-\lambda x\left(1-\sum_{i=1}^{\infty}r_iz^i\right)\right\}\mbox{d}B_1(x)\\
&=\widehat{B}_1(\lambda-\lambda\widehat{R}(z)).
\end{aligned}
\end{equation}
(By $\widehat{B}_1(s)$, $s\geq0$, we denote the Laplace-Stieltjes
transform of $B_1(x)$, and $\widehat{R}(z)=\sum_{i=1}^\infty
r_iz^i$, $|z|\leq1$.) Hence, from \eqref{SP5.1} and \eqref{SP5.0} we
obtain:
\begin{equation}\label{SP5.2}
\sum_{j=0}^\infty\mathsf{E}\widetilde\nu_j^{(1)}z^j=\frac{\widehat{B}_1(\lambda-\lambda\widehat{R}(z))}
{\widehat{B}_1(\lambda-\lambda\widehat{R}(z))-z}.
\end{equation}

Notice, that the right-hand side of \eqref{SP5.0} and, hence, that
of \eqref{SP5.2} has the same form as \eqref{SP4+1}. Therefore we
can use Lemmas \ref{lem.T} and \ref{lem.P}, and according to these
lemmas, the asymptotic behaviour of
$\mathsf{E}\widetilde{\nu}_L^{(1)}$, as $L\to\infty$, is given by
the following statements.
\begin{lem}\label{Asymp behav 1}
If $\rho_1<1$, then
\begin{equation}\label{SP5.3}
\lim_{L\to\infty}\mathsf{E}\widetilde\nu_L^{(1)}=\frac{1}{1-\rho_1}.
\end{equation}
If $\rho_1=1$, and additionally $\rho_{1,2}=\int_0^\infty
(\lambda x)^2\mbox{d}B_1(x)<\infty$ and
$\mathsf{E}\varsigma^2<\infty$, then
\begin{equation}\label{SP5.4}
 \mathsf{E}\widetilde\nu_L^{(1)}-\mathsf{E}\widetilde\nu_{L-1}^{(1)}
=\frac{2\mathsf{E}\varsigma}{\rho_{1,2}(\mathsf{E}\varsigma)^{3}+\mathsf{E}\varsigma^{2}-\mathsf{E}\varsigma
}+o(1).
\end{equation}
If $\rho_1>1$, then
\begin{equation}\label{SP5.5}
\lim_{L\to\infty}\left[\mathsf{E}\widetilde\nu_L^{(1)}-\frac{1}{\varphi^L[1+\lambda
\widehat{B}_1^\prime(\lambda-\lambda\widehat{R}(\varphi))\widehat{R}^\prime(\varphi)]}\right]=\frac{1}{1-\rho_1},
\end{equation}
where $\varphi<1$ is the least positive root of the functional
equation $z=\widehat{B}_1(\lambda-\lambda\widehat{R}(z))$.
\end{lem}

The proof of this lemma is given in Section \ref{Proofs}.

\begin{rem}
In the case when $\rho_1$ is the function of $L$, relation \eqref{SP5.5} of Lemma \ref{Asymp behav 1} is rewritten as follows.
\begin{equation}\label{eq.0.26}
\lim_{L\to\infty}\varphi^L\mathsf{E}\widetilde\nu_L^{(1)}=\frac{1}{\varphi^L[1+\lambda
\widehat{B}_1^\prime(\lambda-\lambda\widehat{R}(\varphi))\widehat{R}^\prime(\varphi)]}.
\end{equation}
where $\varphi$ is the root of functional equation associated with the limiting functions depending of $L$, as $L\to\infty$. Limit relation \eqref{eq.0.26} is based on application of Lemma \ref{lem9} instead of Lemma \ref{lem.T}.
\end{rem}

With the aid of Lemma \ref{Asymp behav 1} one can obtain the
statements on asymptotic behavior of $\mathsf{E}\nu_L^{(1)}(\zeta_1)$,
$\mathsf{E}\nu_L^{(1)}(\zeta_1\wedge L)$ and, consequently,
$p_1$ and $p_2$. The theorem below characterizes asymptotic behavior
of the probabilities $p_1$ and $p_2$ as $L\to\infty$.

\begin{thm}\label{M asymp}
If $\rho_1<1$, then
\begin{eqnarray}
\lim_{L\to\infty} p_1(L)&=&1-\rho_1,\label{SP5.6}\\
\lim_{L\to\infty} p_2(L)&=&0.\label{SP5.7}
\end{eqnarray}
If $\rho_1=1$, and additionally $\rho_{1,2}=\int_0^\infty
(\lambda x)^2\mbox{d}B_1(x)<\infty$ and
$\mathsf{E}\varsigma^{2}<\infty$, then
\begin{eqnarray}
\lim_{L\to\infty}
Lp_1(L)&=&\frac{\rho_{1,2}(\mathsf{E}\varsigma)^{3}+\mathsf{E}\varsigma^{2}-\mathsf{E}\varsigma
}{2\mathsf{E}\varsigma},\label{SP5.8}\\
\lim_{L\to\infty}
Lp_2(L)&=&\frac{\rho_2}{1-\rho_2}\cdot\frac{\rho_{1,2}(\mathsf{E}\varsigma)^{3}+\mathsf{E}\varsigma^{2}-\mathsf{E}\varsigma
}{2\mathsf{E}\varsigma}.\label{SP5.9}
\end{eqnarray}
If $\rho_1>1$, then
\begin{equation}\label{SP5.10}
\lim_{L\to\infty}\frac{p_1(L)}{\varphi^L}
=\frac{(1-\rho_2)[1+\lambda\widehat{B}_1^\prime(\lambda-\lambda\widehat{R}(\varphi))\widehat{R}^\prime(\varphi)]}
{(\rho_1-\rho_2)},
\end{equation}
\begin{equation}\label{SP5.11}
\lim_{L\to\infty}p_2(L)=\frac{\rho_2(\rho_1-1)}{\rho_1-\rho_2},
\end{equation}
where $\varphi$ is defined in the formulation of Lemma \ref{Asymp
behav 1}.
\end{thm}

The proof of this theorem is given in Section \ref{Proofs}.

\begin{rem}
In the case when $\rho_1$ is the function of $L$, relations \eqref{SP5.10} and \eqref{SP5.11} of Lemma \ref{M asymp} remain the same, since an application of Lemma \ref{lem9} instead of Lemma \ref{lem.T} gives the same result.
In this case, $\varphi$ is the root of functional equation associated with the limiting functions depending of $L$, as $L\to\infty$.
\end{rem}

\begin{rem}\label{R1} It follows from Theorem \ref{M asymp}, under the
assumption $\rho_1<1$ we have \eqref{SP5.6} and \eqref{SP5.7}. The
probability $p_1(L)$ tends to the positive limit as $L\to\infty$,
while the probability $p_2(L)$ vanishes as $L\to\infty$. Then, for
large $L$, the functional $J=J(L)$ in \eqref{I1} is estimated as
$J\approx (1-\rho_1)J_1=(1-\rho_1)j_1L$, that is, it increases
proportionally to large parameter $L$.

Under the assumption $\rho_1>1$ we have \eqref{SP5.10} and
\eqref{SP5.11}. Then, for large $L$, the probability $p_1(L)$ is
estimated as
$$
p_1(L)\asymp\frac{(1-\rho_2)[1+\lambda\widehat{B}_1^\prime(\lambda-\lambda\widehat{R}(\varphi))\widehat{R}^\prime(\varphi)]}
{(\rho_1-\rho_2)}\varphi^L,
$$
that is, tends to zero since $\varphi<1$. The probability $p_2(L)$
converges to the positive limit as $L\to\infty$. This means, that
for large $L$, the functional $J=J(L)$ in \eqref{I1} is estimated as
$J\approx
\rho_2(\rho_1-1)/(\rho_1-\rho_2)J_2=\rho_2(\rho_1-1)/(\rho_1-\rho_2)j_2L$.
That is, similarly to the case $\rho_1<1$ it tends to infinity with
the rate proportional to $L$.

Under the assumption $\rho_1=1$ following the limits \eqref{SP5.8}
and \eqref{SP5.9} the limit of $J(L)$ is finite. So, the only case
$\rho_1=1$ among these three cases $\rho_1<1$, $\rho_1=1$ and
$\rho_1>1$ can be a ``candidate" to the optimal solution. In fact,
the case $\rho_1=1$ belongs to the class of heavy traffic conditions
given by \eqref{SP5.15}, an optimal solution must belong to the
set of traffic parameters $\rho_1(L)$ such as there is the limit
$L\rho_1(L)$ depending on parameters $j_1$, $j_2$ and the sequence
of costs $\{c_i\}$ depending of water levels in the dam.
\end{rem}

\subsection{Asymptotic theorems for $p_1$ and $p_2$ under special heavy traffic conditions}\label{Final asymp 2}

In this section we establish asymptotic theorems for $p_1$ and $p_2$
under heavy traffic assumptions where (j) $\rho_1=1+\delta(L)$ or
(jj) $\rho_1=1-\delta(L)$, and $\delta(L)$ is a vanishing positive
parameter as $L\to\infty$. The theorems presented in this section
are analogues of the theorems of \cite{Abramov 2007} given in
Section 4 of that paper. The conditions are special, because these
heavy traffic conditions include the change of the parameter
$\rho_1$ as $L$ increases to infinity and $\delta(L)$ vanishes, but
the other load parameter $\rho_2$ remains unchanged.

\smallskip
In case (j) we have the following two theorems.

\begin{thm}\label{i1} Under Condition \ref{cond1} we have
\begin{eqnarray}
Lp_1&=&\frac{C}{\exp\left(\frac{2C\mathsf{E}\varsigma}
{\widetilde{\rho}_{1,2}(\mathsf{E}\varsigma)^{3}+\mathsf{E}\varsigma^{2}-\mathsf{E}\varsigma}\right)-1}[1+o(1)],\label{SP6.1}\\
Lp_2&=&\frac{C\rho_2\exp\left(\frac{2C\mathsf{E}\varsigma}
{\widetilde{\rho}_{1,2}(\mathsf{E}\varsigma)^{3}+\mathsf{E}\varsigma^{2}-\mathsf{E}\varsigma}\right)}
{(1-\rho_2)\left[\exp\left(\frac{2C\mathsf{E}\varsigma}
{\widetilde{\rho}_{1,2}(\mathsf{E}\varsigma)^{3}+\mathsf{E}\varsigma^{2}-\mathsf{E}\varsigma}\right)-1\right]}
[1+o(1)].\label{SP6.2}
\end{eqnarray}
\end{thm}

The proof of this theorem is given in Section \ref{Proofs}.

\begin{thm}\label{i2} Under Condition \ref{cond1}
assume that $L\delta(L)\to0$. Then,
\begin{eqnarray}
\lim_{L\to\infty}Lp_1(L)&=&\frac
{\widetilde{\rho}_{1,2}(\mathsf{E}\varsigma)^{3}+\mathsf{E}\varsigma^{2}-\mathsf{E}\varsigma}{2\mathsf{E}\varsigma},\label{SP6.9}\\
\lim_{L\to\infty}Lp_2(L)&=&\frac{\rho_2}{1-\rho_2}\frac
{\widetilde{\rho}_{1,2}(\mathsf{E}\varsigma)^{3}+\mathsf{E}\varsigma^{2}-\mathsf{E}\varsigma}{2\mathsf{E}\varsigma}.\label{SP6.10}
\end{eqnarray}
\end{thm}

\begin{proof}
The statement of the theorem follows by expanding the main terms of
asymptotic relations \eqref{SP6.1} and \eqref{SP6.2} for small $C$.
\end{proof}

\smallskip
In case (jj) we have the following two theorems.

\begin{thm}\label{i3} Under Condition \ref{cond2} we have:
\begin{eqnarray}
p_1&=&\delta\exp\left(-\frac
{\widetilde{\rho}_{1,2}(\mathsf{E}\varsigma)^{3}+\mathsf{E}\varsigma^{2}-\mathsf{E}\varsigma}{2C\mathsf{E}\varsigma}\right)[1+o(1)],
\label{SP6.11}\\
p_2&=&\frac{\delta\rho_2\left[\exp\left(-\frac
{\widetilde{\rho}_{1,2}(\mathsf{E}\varsigma)^{3}+\mathsf{E}\varsigma^{2}-\mathsf{E}\varsigma}{2C\mathsf{E}\varsigma}\right)-1\right]}
{1-\rho_2}[1+o(1)].\label{SP6.12}
\end{eqnarray}
\end{thm}

The proof of this theorem is given in Section \ref{Proofs}.

\begin{thm}\label{i4}
Under Condition \ref{cond2} assume that $L\delta\to0$. Then we have
\eqref{SP6.9} and \eqref{SP6.10}.
\end{thm}

\begin{proof}
The proof of the theorem follows by expanding the main terms of the
asymptotic relations \eqref{SP6.11} and \eqref{SP6.12} for small
$C$.
\end{proof}

\section{Asymptotic theorems for the stationary probabilities $q_i$}\label{Q stationary
probabilities}

The aim of this section is asymptotic analysis of the stationary
probabilities $q_i$, $i=1,2,\ldots,L$ as $L\to\infty$. The challenge
is to first obtain the explicit representation for $q_i$ in terms of
$\mathsf{E}\nu_i^{(1)}(\zeta_1)$, and then to study the asymptotic behavior
of $q_i$ as $L\to\infty$ on the basis of the known asymptotic
results for $\mathsf{E}\nu_L^{(1)}(\zeta_1)$ as $L\to\infty$.

To find asymptotic behaviour of stationary probabilities we need the
following major steps.

\smallskip
1. First, we derive representation for stationary probabilities in
terms of $\mathsf{E}\nu_i^{(1)}(\zeta_1)$, $i=0,1,\ldots$, where
$\mathsf{E}\nu_i^{(1)}(\zeta_1)$ with lower index $i$ has the similar meaning
as $\mathsf{E}\nu_L^{(1)}(\zeta_1)$ with the replacement of the level $L$ by
$i$ (see Lemma \ref{Explicit qi}).

\smallskip
2. Then, we study asymptotic behaviour of the difference
$\mathsf{E}\nu_{L-j}^{(1)}(\zeta_1)-\mathsf{E}\nu_{L-j-1}^{(1)}(\zeta_1)$ as
$L\to\infty$. This difference is an important part of the formula
that defines the asymptotic behaviour of the stationary
probabilities.

\smallskip
3. The asymptotic behaviour of the aforementioned difference should
be studied for the following three cases
 $\rho_1=1$, $\rho_1=1+\delta(L)$ and $\rho_1=1-\delta(L)$, where
$\delta(L)$ is a positive vanishing value. The first two cases are
based on a standard study based on asymptotic behaviour of the
difference $\mathsf{E}\nu_{L-j}^{(1)}(\zeta_1)-\mathsf{E}\nu_{L-j-1}^{(1)}(\zeta_1)$,
and the main results for this study are Theorems \ref{lem3} and
\ref{thm1}. The third case is more complicated and involves special
results on asymptotic behaviour of this type of sequences
\cite{Willmot 1988}. As well, some special additional assumptions
are required (see Theorem \ref{thm4}).

\subsection{Explicit representation for the stationary probabilities $q_i$}\label{Explicit q}
The aim of this section is to prove the following statement.

\begin{lem}\label{Explicit qi} For $i=1,2,\ldots,L$ we have
\begin{equation}\label{SP7.1}
q_i=\rho_1p_1\left(\mathsf{E}\nu_i^{(1)}(\zeta_1)-\mathsf{E}\nu_{i-1}^{(1)}(\zeta_1)\right).
\end{equation}
\end{lem}

\begin{proof} Rewriting relation \eqref{MXG1.10} in the form
\begin{equation}\label{SP7.4}
\mathsf{E}T_L(\zeta_1)-\frac{1}{\mu_1}+\frac{\mathsf{E}\zeta_1}{\lambda\mathsf{E}\varsigma}=\frac{\left[\mathsf{E}\nu_L(\zeta_1)-1\right]}{\lambda\mathsf{E}\varsigma},
\end{equation}
and now using renewal arguments (e.g. \cite{Ross 2000}), relation
\eqref{SP7.4} and Wald's identities
\begin{equation*}
\mathsf{E}T_i^{(1)}(\zeta_1)-\frac{1}{\mu_1}=\frac{\rho_1}{\lambda\mathsf{E}\varsigma}\left[\mathsf{E}\nu_i^{(1)}(\zeta_1)-1\right],
\ i=1,2,\ldots,L,
\end{equation*}
 we obtain:
\begin{equation}\label{SP7.2}
\begin{aligned}
&q_i=\frac{\mathsf{E}T_i^{(1)}(\zeta_1)-\mathsf{E}T_{i-1}^{(1)}(\zeta_1)}{\mathsf{E}T_L(\zeta_1)+\frac{1}{\lambda}}=\rho_1
\frac{\mathsf{E}\nu_i^{(1)}(\zeta_1)-\mathsf{E}\nu_{i-1}^{(1)}(\zeta_1)}{\mathsf{E}\nu_L(\zeta_1)},\\
&i=1,2,\ldots,L.
\end{aligned}
\end{equation}
Taking into account that
$\mathsf{E}\nu_L(\zeta_1)=\mathsf{E}\nu_L^{(1)}(\zeta_1)+\mathsf{E}\nu_L^{(2)}(\zeta_1)$ and
then applying the linear representation for $\mathsf{E}\nu_L^{(2)}(\zeta_1)$
given by \eqref{SP4.23}, from \eqref{SP7.2} we obtain:
\begin{equation*}\label{SP7.3}
\begin{aligned}
&q_i=\frac{\rho_1(1-\rho_2)\mathsf{E}\zeta_1}{\mathsf{E}\zeta_1+(\rho_1-\rho_2)\left[\mathsf{E}\nu_L^{(1)}(\zeta_1)-1\right]}
\left(\mathsf{E}\nu_i^{(1)}(\zeta_1)-\mathsf{E}\nu_{i-1}^{(1)}(\zeta_1)\right),\\
&i=1,2,\ldots,L.
\end{aligned}
\end{equation*}
Hence, representation \eqref{SP7.1} follows from \eqref{SP4.19}
(see Lemma \ref{Prel asymp}), and Lemma \ref{Explicit qi} is proved.
\end{proof}

\subsection{Asymptotic analysis of the stationary probabilities $q_i$: The case
$\rho_1=1$}\label{Case 1} Let us study asymptotic behavior of the
stationary probabilities $q_i$. We start from the following modified
version of \eqref{SP5.4} (Lemma \ref{Asymp behav 1}):
\begin{equation}\label{SP8.1}
 \mathsf{E}\widetilde\nu_{L-j}^{(1)}-\mathsf{E}\widetilde\nu_{L-j-1}^{(1)}
=\frac
{2\mathsf{E}\varsigma}{\rho_{1,2}(\mathsf{E}\varsigma)^{3}+\mathsf{E}\varsigma^{2}-\mathsf{E}\varsigma}+o(1),
\end{equation}
which is assumed to be satisfied under the conditions
$\rho_{1,2}<\infty$ and $\mathrm{E}\varsigma^{2}<\infty$. Under the
same conditions, similarly to \eqref{SP5.13} we obtain:
\begin{equation}\label{SP8.2}
\begin{aligned}
\mathsf{E}\nu_{L-j}^{(1)}(\zeta_1\wedge L)-&\mathsf{E}\nu_{L-j-1}^{(1)}(\zeta_1\wedge L)\\ &= \frac
{2\mathsf{E}\varsigma}{\rho_{1,2}(\mathsf{E}\varsigma)^{3}+\mathsf{E}\varsigma^{2}-\mathsf{E}\varsigma}\\
&\ \ \ \times
\sum_{i=1}^{L}i\mathsf{Pr}\{\zeta_1\wedge L=i\}+o(1)\\
&=\frac
{2\mathsf{E}\varsigma\mathsf{E}\zeta_1}{\rho_{1,2}(\mathsf{E}\varsigma)^{3}+\mathsf{E}\varsigma^{2}-\mathsf{E}\varsigma}+o(1).
\end{aligned}
\end{equation}
Hence, since for any $j<L$,
$$
\mathsf{E}\nu_{L-j}^{(1)}(\zeta_1\wedge L)=\mathsf{E}\nu_{L-j}^{(1)}(\zeta_1\wedge (L-j))=\mathsf{E}\nu_{L-j}^{(1)}(\zeta_1),
$$
then from \eqref{SP8.2} we have the estimate
\begin{equation}\label{SP8.3}
\begin{aligned}
\mathsf{E}\nu_{L-j}^{(1)}(\zeta_1)-\mathsf{E}\nu_{L-j-1}^{(1)}(\zeta_1) &=\frac
{2\mathsf{E}\varsigma\mathsf{E}\zeta_1}{\rho_{1,2}(\mathsf{E}\varsigma)^{3}+\mathsf{E}\varsigma^{2}-\mathsf{E}\varsigma}+o(1).
\end{aligned}
\end{equation}

Asymptotic relations \eqref{SP8.3}, \eqref{SP5.8} together with
explicit relation \eqref{SP7.1} of Lemma \ref{Explicit qi} leads to
the following theorem.

\begin{thm}\label{lem3} In the case $\rho_1=1$ under the additional conditions $\rho_{1,2}<\infty$
and $\mathrm{E}\varsigma^{2}<\infty$ for large values $j$ such that
${j}/{L}\to1$ as $L\to\infty$, we have
\begin{equation}\label{SP8.4}
\lim_{L\to\infty}Lq_{j}=1.
\end{equation}
\end{thm}

Note, that the asymptotic relation given by \eqref{SP8.4} is not
expressed via $\mathsf{E}\varsigma$ and, therefore, it is invariant
and hence the same as that for the queueing system with ordinary
Poisson arrivals.

\subsection{Asymptotic analysis of the stationary probabilities $q_i$:
The case $\rho_1=1+\delta(L)$} \label{Case 2} In the case
$\rho_1=1+\delta(L)$, $\delta>0$ the asymptotic behaviour of $q_i$
is specified by the following theorem.

\begin{thm}\label{thm1}
Assume that Condition \ref{cond1} is satisfied and
$\mathrm{E}\varsigma^{3}<\infty$. Then, for all $j\geq0$, we have
\begin{equation}\label{SP9.1}
\begin{aligned}
q_{L-j}&=\frac{\exp\left(\frac{2C\mathsf{E}\varsigma}{\widetilde{\rho}_{1,2}(\mathsf{E}\varsigma)^{3}+\mathsf{E}\varsigma^{2}-\mathsf{E}\varsigma}\right)}
{\exp\left(\frac{2C\mathsf{E}\varsigma}{\widetilde{\rho}_{1,2}(\mathsf{E}\varsigma)^{3}+\mathsf{E}\varsigma^{2}-\mathsf{E}\varsigma}\right)-1}\\
&\ \ \ \times
\left(1-\frac{2\delta\mathsf{E}\varsigma}{\widetilde{\rho}_{1,2}
(\mathsf{E}\varsigma)^{3}+\mathsf{E}\varsigma^{2}-\mathsf{E}\varsigma}\right)^j\\
&\ \ \
\times\frac{2\delta\mathsf{E}\varsigma}{\widetilde{\rho}_{1,2}
(\mathsf{E}\varsigma)^{3}+\mathsf{E}\varsigma^{2}-\mathsf{E}\varsigma}
+o(\delta).
\end{aligned}
\end{equation}
\end{thm}

The proof of this theorem is given in Section \ref{Proofs}.

\subsection{Asymptotic analysis of the stationary probabilities $q_i$:
The case $\rho_1=1-\delta(L)$}\label{Case 3} In the case
$\rho_1=1-\delta(L)$, $\delta>0$, the study is more delicate and
based on special analysis. The additional assumption of this case is
that the class of probability distribution functions $\{B_1(x)\}$
and $\mathsf{Pr}\{\varsigma=i\}$ are given such that there exists a
unique root $\tau\in(1,\infty)$ of the equation
\begin{equation}\label{SP10.1}
z=\widehat{B}_1(\lambda-\lambda \widehat{R}(z)),
\end{equation}
and there exists the first derivative
$\widehat{B}_1^\prime(\lambda-\lambda \widehat{R}(\tau))$.

Under the assumption that $\rho_1<1$ the root of
\eqref{SP10.1} is not necessarily exists and if exists it is not necessarily unique. Such type of condition has
been considered by Willmot \cite{Willmot 1988} to obtain the
asymptotic behavior for high queue-level probabilities in stationary
$M/G/1$ queues.

The analysis provided here consists of the following three parts.

\smallskip
1. We consider $M^X/G/1$ queueing system with $\rho_1<1$, and under the assumption that the first batch in a busy period consists of only one customer, we derive the asymptotic formula for the stationary probability $q_i$ for large $i$. The idea is first to extend the asymptotic result of Willmot \cite{Willmot 1988} obtained for the $M/G/1$ queueing system. Similarly to Willmot \cite{Willmot 1988}, we show that the stationary probability $q_i$ is presented as $c\tau^{-i}[1+o(1)]$ with exact explicit expression for the constant $c$.

\smallskip
2. Based on this result, we then express $\tau$ via the numerical characteristics $\mathsf{E}\widetilde\nu_{L-j}^{(1)}$, $L\to\infty$, which in turn were studied in Section \ref{Stationary probabilities}.

\smallskip
3. Then we provide asymptotic study of $\tau$ as $\rho_1$ approaches 1. The study is based on Taylor's expansion for an explicit analytic expression.

\smallskip

Denote the stationary probabilities in the $M/G/1$
queueing system by $q_i[M/G/1]$, $i=0,1,\ldots$. It was shown in
\cite{Willmot 1988} that
\begin{equation}\label{SP10.2}
q_i[M/G/1]=\frac{(1-\rho_1)(1-\tau)}{\tau^i[1+\lambda\widehat{B}_1^\prime(\lambda-\lambda\tau)]}[1+o(1)]
\ \mbox{as} \ i\to\infty,
\end{equation}
where $\widehat{B}_1(s)$ denotes the Laplace-Stieltjes transform of
the service time distribution in the $M/G/1$ queueing system, and
$\tau$ denotes a root of the equation
$z=\widehat{B}_1(\lambda-\lambda z)$ greater than 1, which is
assumed to be unique. On the other hand, according to the
Pollaczek-Khintchine formula (e.g. Tak\'acs \cite{Takacs 1962},
p.242), $q_i[M/G/1]$ can be represented explicitly
\begin{equation}\label{SP10.3}
q_i[M/G/1]=(1-\rho_1)\left(\mathsf{E}\nu_i^{(1)}-\mathsf{E}\nu_{i-1}^{(1)}\right),
i=1,2,\ldots,
\end{equation}
where the random variable $\nu_i^{(1)}$ in this formula is
associated with the number of served customers during a busy period
of the state dependent $M/G/1$ queueing system, where the value of
the system parameter, where the service is changed, is $i$ (see
Section \ref{MG1}). Representation \eqref{SP10.3} can be easily
checked, since in this case
\begin{equation}\label{SP10.4}
\sum_{j=0}^\infty\mathsf{E}\nu_j^{(1)}z^j=\frac{\widehat{B}_1(\lambda-\lambda
z)}{\widehat{B}_1(\lambda-\lambda z)-z},
\end{equation}
and multiplication of the right-hand side of \eqref{SP10.4} by
$(1-\rho_1)(1-z)$ leads to the well-known Pollaczek-Khintchine
formula. Then, from \eqref{SP10.2} and \eqref{SP10.3} there is the
asymptotic proportion for large $L$ and any $j\geq0$:
\begin{equation}\label{SP10.5}
\frac{\mathsf{E}\nu_{L-j}^{(1)}-\mathsf{E}\nu_{L-j-1}^{(1)}}{\mathsf{E}\nu_{L}^{(1)}-\mathsf{E}\nu_{L-1}^{(1)}}
=\tau^j[1+o(1)].
\end{equation}

In the case of batch arrivals the results are similar. One can prove
that the same proportion as \eqref{SP10.5} holds in this case as
well, where $\tau$ in the case of batch arrivals denotes a unique
real root of the equation of \eqref{SP10.1}, which is greater than
1. (Recall that our convention is an existence of a unique real
solution of \eqref{SP10.1} greater than 1.) Indeed, the arguments of
\cite{Willmot 1988} are elementary extended for the queueing system
with batch arrivals. The simplest way to extend these results
straightforwardly is to consider the stationary queueing system with
batch Poisson arrivals, in which the first batch in each busy period
is equal to 1. Denote this system by $M^{1,X}/G/1$. For this
specific system, similarly to \eqref{SP10.2} we obtain:
\begin{equation}\label{SP10.6}
\begin{aligned}
&q_i[M^{1,X}/GI/1]=
\frac{(1-\rho_1)(1-\tau)}{\tau^i[1+\lambda\widehat{B}_1^\prime(\lambda-\lambda
\widehat{R}(\tau))\widehat{R}^\prime(\tau)]}[1+o(1)],\\
& \ \mbox{as} \ i\to\infty,
\end{aligned}
\end{equation}
where $q_i[M^{1,X}/GI/1]$, $i=0,1,\ldots$,  denotes the stationary
probabilities in this system. Then, taking into account
\eqref{SP5.2}, similarly to \eqref{SP10.3} one can write
\begin{equation}\label{SP10.7}
q_i[M^{1,X}/GI/1]=(1-\rho_1)\left(\mathsf{E}\widetilde\nu_i^{(1)}-\mathsf{E}\widetilde\nu_{i-1}^{(1)}\right),
\ i=1,2,\ldots.
\end{equation}
From \eqref{SP10.6} and \eqref{SP10.7} we obtain
\begin{equation}\label{SP10.8}
\frac{\mathsf{E}\widetilde\nu_{L-j}^{(1)}-\mathsf{E}\widetilde\nu_{L-j-1}^{(1)}}
{\mathsf{E}\widetilde\nu_{L}^{(1)}-\mathsf{E}\widetilde\nu_{L-1}^{(1)}}
=\tau^j[1+o(1)].
\end{equation}
From \eqref{SP10.8} and the results of Sections \ref{Preliminary
asymp} and \ref{Final asymp} (see \eqref{SP5.3}, the statement on
asymptotic behaviour of $\mathsf{E}\nu_L^{(1)}(\zeta_1\wedge L)$ as $L\to\infty$ given in Section \ref{Proofs} (relation
\eqref{SP5.12}) and the equality
$\mathsf{E}\nu_L^{(1)}(\zeta_1\wedge L)=\mathsf{E}\nu_L^{(1)}(\zeta_1)$), we also have the estimate
\begin{equation}\label{SP10.9}
\frac{\mathsf{E}\nu_{L-j}^{(1)}(\zeta_1)-\mathsf{E}\nu_{L-j-1}^{(1)}(\zeta_1)}
{\mathsf{E}\nu_{L}^{(1)}(\zeta_1)-\mathsf{E}\nu_{L-1}^{(1)}(\zeta_1)} =\tau^j[1+o(1)],
\end{equation}
which coincides with \eqref{SP10.5}.

\smallskip
Now we formulate and prove a theorem on asymptotic behavior of the
stationary probabilities $q_i$ in the case $\rho_1=1-\delta$,
$\delta>0$. The special assumption in this theorem is that the class
of probability distributions $\{B_1(x)\}$ is defined according to
the above convention. More precisely, in the case $\rho_1=1-\delta$,
$\delta>0$, and vanishing $\delta$ as $L\to\infty$ this means that
Condition \ref{cond3} should be satisfied.

\begin{thm}\label{thm4} Assume that Conditions \ref{cond2} and \ref{cond3}
are satisfied and $\mathrm{E}\varsigma^{3}<\infty$. Then,
\begin{equation}\label{SP10.12}
\begin{aligned}
q_{L-j}&=\frac{1}{\exp\left(\frac{2C\mathsf{E}\varsigma}{\widetilde{\rho}_{1,2}(\mathsf{E}\varsigma)^{3}+\mathsf{E}\varsigma^{2}-\mathsf{E}\varsigma}\right)-1}\\
&\ \ \ \times
\frac{2\delta\mathsf{E}\varsigma}{\widetilde{\rho}_{1,2}(\mathsf{E}\varsigma)^{3}+
\mathsf{E}\varsigma^{2}-\mathsf{E}\varsigma}\\
&\ \ \ \times
\left(1+\frac{2\delta\mathsf{E}\varsigma}{\widetilde{\rho}_{1,2}
(\mathsf{E}\varsigma)^{3}+\mathsf{E}\varsigma^{2}-\mathsf{E}\varsigma}\right)^j[1+o(1)],
\end{aligned}
\end{equation}
for any $j\geq0.$
\end{thm}

The proof of this theorem is given in Section \ref{Proofs}.

\section{Derivations for the objective function}\label{ObFunction}

\subsection{The case $\rho_1=1$}\label{Case 1O} In this section we prove the following result.

\begin{prop}\label{prop1}
In the case $\rho_1=1$, under the additional conditions
$\rho_{1,2}<\infty$ and $\mathrm{E}\varsigma^{2}<\infty$ we have:
\begin{equation}\label{SP11.1}
\begin{aligned}
\lim_{L\to\infty}J(L)=&j_1\frac{{\rho}_{1,2}(\mathsf{E}\varsigma)^{3}
+\mathsf{E}\varsigma^{2}-\mathsf{E}\varsigma}{2\mathsf{E}\varsigma}\\
&+j_2\frac{\rho_2}{1-\rho_2}\cdot
\frac{{\rho}_{1,2}(\mathsf{E}\varsigma)^{3}+\mathsf{E}\varsigma^{2}-\mathsf{E}\varsigma}{2\mathsf{E}\varsigma}+c^{*},
\end{aligned}
\end{equation}
where
\begin{equation*}\label{SP11.2}
c^{*}=\lim_{L\to\infty}\frac{1}{L}\sum_{i=1}^L c_i.
\end{equation*}
\end{prop}

\begin{proof}
The first two terms in the right-hand side of \eqref{SP11.1} follow
from asymptotic relations \eqref{SP5.8} and \eqref{SP5.9} (Theorem
\ref{M asymp}). The last term $c^*$ of the right-hand side of
\eqref{SP11.1} follows from \eqref{SP8.4} (Theorem \ref{lem3}),
since
\begin{equation*}\label{SP11.3}
\lim_{L\to\infty}\sum_{i=1}^Lq_ic_i=\lim_{L\to\infty}\frac{1}{L}\sum_{i=1}^L
c_i=c^*.
\end{equation*}
\end{proof}

\subsection{The case $\rho_1=1+\delta(L)$}\label{Case 2O}
 In the case $\rho_1=1+\delta(L)$, $\delta>0$ we have the following statement.

\begin{prop}\label{prop2} Under the assumptions of Theorem \ref{thm1}
for the series of objective functions $J(L)$ we have representation
\eqref{SP11.4*}. Then, a solution to the control problem is found in
the set of possible limits $\overline{J}^{\mathrm{upper}}(C)$.
\end{prop}

The proof of Proposition \ref{prop2} is given in Section
\ref{Proofs}.

\subsection{The case $\rho_1=1-\delta(L)$}\label{Case 3O}
 In the case $\rho_1=1-\delta(L)$, $\delta>0$ we have the following statement.

\begin{prop}\label{prop3}
Under the assumptions of Theorem \ref{thm4} for the series of objective
functions $J(L)$ we have representation \eqref{SP12.1*}. Then, a
solution to the control problem is found in the set of possible
limits $\underline{J}^{\mathrm{lower}}(C)$.
\end{prop}

The proof of Proposition \ref{prop3} is given in Section
\ref{Proofs}.

\section{A solution to the control problem and its
properties}\label{Solution} In this section we discuss the solution
to the control problem and study its properties.

\subsection{Alternative representations for the last terms
in the objective functions and their properties}\label{S7.1} The series of
objective functions $J^{\mathrm{upper}}[L,C(L)]$ and
$J^{\mathrm{lower}}[L,C(L)]$ are given by \eqref{SP11.4*} and,
respectively, by \eqref{SP12.1*}, and the last terms in these
functionals are given by \eqref{SP11.5*} and, respectively, by
\eqref{SP12.2*}. For our further analysis we need in other
representations for these last terms.

Recall that when $\rho_1(L)=1+\delta(L)$, the parameter $C$ defined
in \eqref{Def_of_C} is positive, while in the opposite case
$\rho_1(L)=1-\delta(L)$ it is negative. For the following study of
the properties of the possible limits of
$c^{\mathrm{upper}}[L,C(L)]$ and $c^{\mathrm{lower}}[L,C(L)]$ as
$L\to\infty$, we need to introduce the function
\begin{equation}\label{MXG1.18}
\psi(C)=\lim_{L\to\infty}\frac{\sum_{j=0}^{L-1}c_{L-j}\left(1-\frac{2C\mathsf{E}\varsigma}
{[{\rho}_{1,2}(L)(\mathsf{E}\varsigma)^{3}+\mathsf{E}\varsigma^{2}-\mathsf{E}\varsigma]L}\right)^j}
{\sum_{j=0}^{L-1}\left(1-\frac{2C\mathsf{E}\varsigma}
{[{\rho}_{1,2}(L)(\mathsf{E}\varsigma)^{3}+\mathsf{E}\varsigma^{2}-\mathsf{E}\varsigma]L}\right)^j}
\end{equation}
and establish its connection with the aforementioned limits.
However, for the purposes it is profitable to split this function
into two different functions in order to distinguish two case
studies. So, instead of \eqref{MXG1.18}, we consider two functions
both defined for a positive argument. Specifically, denoting $D=|C|$
consider the following two functions
\begin{equation}\label{SCP3}
\psi(D)=\lim_{L\to\infty}\frac{\sum_{j=0}^{L-1}c_{L-j}\left(1-\frac{2D\mathsf{E}\varsigma}
{[{\rho}_{1,2}(L)(\mathsf{E}\varsigma)^{3}+\mathsf{E}\varsigma^{2}-\mathsf{E}\varsigma]L}\right)^j}
{\sum_{j=0}^{L-1}\left(1-\frac{2D\mathsf{E}\varsigma}
{[{\rho}_{1,2}(L)(\mathsf{E}\varsigma)^{3}+\mathsf{E}\varsigma^{2}-\mathsf{E}\varsigma]L}\right)^j},
\end{equation}
and
\begin{equation}\label{SCP4}
\eta(D)=\lim_{L\to\infty}\frac{\sum_{j=0}^{L-1}c_{L-j}\left(1+\frac{2D\mathsf{E}\varsigma}
{[{\rho}_{1,2}(L)(\mathsf{E}\varsigma)^{3}+\mathsf{E}\varsigma^{2}-\mathsf{E}\varsigma]L}\right)^j}
{\sum_{j=0}^{L-1}\left(1+\frac{2D\mathsf{E}\varsigma}
{[{\rho}_{1,2}(L)(\mathsf{E}\varsigma)^{3}+\mathsf{E}\varsigma^{2}-\mathsf{E}\varsigma]L}\right)^j}.
\end{equation}
Since $\{c_i\}$ is a nonincreasing and bounded sequence, then the
limits of \eqref{SCP3} and \eqref{SCP4} do exist.

Denote
$$
\overline{c}^{\mathrm{upper}}(C)=\lim_{L\to\infty}c^{\mathrm{upper}}[L,C(L)],
$$
and
$$
\underline{c}^{\mathrm{lower}}(C)=\lim_{L\to\infty}c^{\mathrm{lower}}[L,C(L)],
$$

The relations between $\overline{c}^{\mathrm{upper}}(C)$ and
$\psi(D)$ and, respectively, between
$\underline{c}^{\mathrm{lower}}(C)$ and $\eta(D)$ are given in the
lemma below.

\begin{lem}\label{lem5} We have:
\begin{equation}\label{SCP5}
\overline{c}^{\mathrm{upper}}(C)=\psi(D),
\end{equation}
and
\begin{equation}\label{SCP6}
\underline{c}^{\mathrm{lower}}(C)=\eta(D).
\end{equation}
\end{lem}

The proof of this lemma is given in Section \ref{Proofs}.

\smallskip
The next lemma establishes the main properties of functions
$\psi(D)$ and $\eta(D)$.

\begin{lem}\label{lem6}
The function $\psi(D)$ is a nonincreasing function, and its maximum
is $\psi(0)=c^*$. The function $\eta(D)$ is a nondecreasing
function, and its minimum is $\eta(0)=c^*$.

(Recall that $c^*=\lim_{L\to\infty}({1}/{L})\sum_{i=1}^L c_i$ is
defined in Proposition \ref{prop1}.)
\end{lem}

The proof of this lemma is given in Section \ref{Proofs}.

\smallskip
In the following we need in stronger results that is given by Lemma
\ref{lem6}. Namely, we prove the following lemmas.

\begin{lem}\label{lem7}
If the sequence $\{c_i\}$ contains at least two distinct values,
then the function $\psi(D)$ is a strictly decreasing function, and
the function $\eta(D)$ is a strictly increasing function.
\end{lem}

The proof of this lemma is given in Section \ref{Proofs}.

\begin{lem}\label{lem8}
Under assumption of Lemma \ref{lem7} the function $\psi(D)$ is convex and the function $\eta(D)$ is concave.
\end{lem}

The proof of this lemma is given in Section \ref{Proofs}.

\subsection{Proof of the main result and discussion of optimal solution}
\subsubsection{Preliminaries}\label{S7.1.4}
Before starting the proof we discuss the structure of an optimal solution if it exists.

As it has already been discussed in Remark \ref{R1}, that a possible optimal solution falls into the category of the heavy traffic conditions, which are specified by the relation between the parameters $j_1$ and $j_2$ and the structure of costs $c_i$. Following Remark \ref{R1}, the existence of a solution is intuitively understandable, since according to relations \eqref{SP5.8} and \eqref{SP5.9} of Theorem \ref{M asymp}, under the conditions $\rho_1=1$, $\widetilde{\rho}_{1,2}<\infty$ and $\mathsf{E}\varsigma^2<\infty$ both limits $\lim_{L\to\infty}Lp_1(L)$ and $\lim_{L\to\infty}Lp_2(L)$ are finite, and the functional \ref{I1} must be finite. More rigorous arguments are given in Section \ref{S7.1.1}.

Here we classify different cases of an optimal solution.

\smallskip
The three possible cases of the heavy traffic conditions are as follows.

\smallskip
\textit{Case 1:} $\lim_{L\to\infty}L[\rho_1(L)-1]=C>0$. Analysis of this case is based on heavy-traffic assumption (j) of Condition \ref{cond1} and the statements of Theorem \ref{i1}.
This case is associated with Condition (i) and series of objective functions defined by \eqref{SP11.4*} and \eqref{SP11.5*}.

\smallskip
\textit{Case 2:} $\lim_{L\to\infty}L[\rho_1(L)-1]=C<0$. Analysis of this case is based on heavy-traffic assumption (jj) of Conditions \ref{cond2} and \ref{cond3} and the statements of Theorem \ref{i3}. This case is associated with Condition (ii) and series of objective functions defined by \eqref{SP12.1*} and \eqref{SP12.2*}.

\smallskip
\textit{Case 3:} $\lim_{L\to\infty}L[\rho_1(L)-1]=0$. Analysis of this case is based the statements of Theorem \ref{i2}. Note that asymptotic results under this heavy-traffic condition coincides with limit relations \eqref{SP5.8} and \eqref{SP5.9} of Theorem \ref{M asymp}. This case is associated with Condition (iii) and series of objective functions defined by \eqref{SP12.3*}.

\smallskip
In the proof given below we consider Case 1. The other cases can be studied similarly.

\subsubsection{Existence of a solution}\label{S7.1.1} Note first, that under the
assumptions made, there is a solution to the control problem
considered in this paper. Indeed, a solution contains two terms one
of them corresponds to the expression for $p_1J_1+p_2J_2$ in
\eqref{I1} and another one corresponds to the term
$\sum_{i=L^{\mathrm{lower}}+1}^{L^{\mathrm{upper}}}c_iq_i$ in
\eqref{I1}.

The first term of a solution is related to the models where the
water costs are not taken into account. For Case 1, this term can be extracted
from the function $J^{\mathrm{upper}}[L,C(L)]$ in \eqref{SP11.4*} by
setting $c^{\mathrm{upper}}[L,C(L)]=0$ and passing to the limit as
$L\to\infty$. Denoting this function by $J^*(C)$ we have the
following explicit expression
\begin{equation}\label{MXG1.19}
\begin{aligned}
J^*(C)=&C\left[j_1\frac{1}{\exp\left(\frac{2C\mathsf{E}\varsigma}{\widetilde\rho_{1,2}(\mathsf{E}\varsigma)^3
+\mathsf{E}\varsigma^2-\mathsf{E}\varsigma}\right)-1}\right.\\
&+\left.j_2\frac{\rho_2}{1-\rho_2}\cdot\frac{\exp\left(\frac{2C\mathsf{E}\varsigma}{\widetilde\rho_{1,2}(\mathsf{E}\varsigma)^3
+\mathsf{E}\varsigma^2-\mathsf{E}\varsigma}\right)}{\exp\left(\frac{2C\mathsf{E}\varsigma}{\widetilde\rho_{1,2}(\mathsf{E}\varsigma)^3
+\mathsf{E}\varsigma^2-\mathsf{E}\varsigma}\right)-1}\right].
\end{aligned}
\end{equation}

Taking derivative in $C$ and equating it to zero, we obtain the
equation
\begin{equation}\label{MXG1.20}
\begin{aligned}
&j_1\left[1-\frac{\frac{2C\mathsf{E}\varsigma}{\widetilde\rho_{1,2}(\mathsf{E}\varsigma)^3
+\mathsf{E}\varsigma^2-\mathsf{E}\varsigma}
\exp\left(\frac{2C\mathsf{E}\varsigma}{\widetilde\rho_{1,2}(\mathsf{E}\varsigma)^3
+\mathsf{E}\varsigma^2-\mathsf{E}\varsigma}\right)}{\exp\left(\frac{2C\mathsf{E}\varsigma}{\widetilde\rho_{1,2}(\mathsf{E}\varsigma)^3
+\mathsf{E}\varsigma^2-\mathsf{E}\varsigma}\right)-1}\right]\\
&+j_2\frac{\rho_2}{1-\rho_2}\exp\left(\frac{2C\mathsf{E}\varsigma}{\widetilde\rho_{1,2}(\mathsf{E}\varsigma)^3
+\mathsf{E}\varsigma^2-\mathsf{E}\varsigma}\right)\\
&\times\left[1-\frac{\frac{2C\mathsf{E}\varsigma}{\widetilde\rho_{1,2}(\mathsf{E}\varsigma)^3
+\mathsf{E}\varsigma^2-\mathsf{E}\varsigma}
\exp\left(\frac{2C\mathsf{E}\varsigma}{\widetilde\rho_{1,2}(\mathsf{E}\varsigma)^3
+\mathsf{E}\varsigma^2-\mathsf{E}\varsigma}\right)}{\exp\left(\frac{2C\mathsf{E}\varsigma}{\widetilde\rho_{1,2}(\mathsf{E}\varsigma)^3
+\mathsf{E}\varsigma^2-\mathsf{E}\varsigma}\right)-1}\right]=0.
\end{aligned}
\end{equation}

Equation \eqref{MXG1.20} has a solution. Indeed, setting $C=0$ for
the left-hand side of \eqref{MXG1.20} transforms it to the
inequality
$$
j_1+j_2\frac{\rho_2}{1-\rho_2}>0.
$$
On the other hand, setting
$$
C=\frac{\widetilde\rho_{1,2}(\mathsf{E}\varsigma)^3
+\mathsf{E}\varsigma^2-\mathsf{E}\varsigma}{2\mathsf{E}\varsigma}
$$
we obtain the inequality
$$
j_1\left[1-\frac{\mathrm{e}}{\mathrm{e}-1}\right]+j_2\frac{\rho_2}{1-\rho_2}
\left[1-\frac{\mathrm{e}}{\mathrm{e}-1}\right]<0.
$$
Thus, \eqref{MXG1.20} has a solution.

\subsubsection{Uniqueness of a solution}\label{S7.1.2}
To prove that the solution that is discussed in Section \ref{S7.1.1} is unique, we are to prove that the second derivative of the function $J^*(C)$ defined in \eqref{MXG1.19} is positive. Indeed,
the derivative of the function on left-hand side of \eqref{MXG1.20} is
$$
\left(j_1+j_2\frac{\rho_2}{1-\rho_2}\right)\frac{\frac{2C(\mathsf{E}\varsigma)^2}
{(\widetilde\rho_{1,2}(\mathsf{E}\varsigma)^3
+\mathsf{E}\varsigma^2-\mathsf{E}\varsigma)^2}
\exp\left(\frac{2C\mathsf{E}\varsigma}{\widetilde\rho_{1,2}(\mathsf{E}\varsigma)^3
+\mathsf{E}\varsigma^2-\mathsf{E}\varsigma}\right)}{\left[\exp\left(\frac{2C\mathsf{E}\varsigma}{\widetilde\rho_{1,2}(\mathsf{E}\varsigma)^3
+\mathsf{E}\varsigma^2-\mathsf{E}\varsigma}\right)-1\right]^2}>0,
$$
and taking into account that the left-hand side of \eqref{MXG1.20} is presented as
$$
\left[\exp\left(\frac{2C\mathsf{E}\varsigma}{\widetilde\rho_{1,2}(\mathsf{E}\varsigma)^3
+\mathsf{E}\varsigma^2-\mathsf{E}\varsigma}\right)-1\right]\frac{\mathrm{d}J^*}{\mathrm{d}C},
$$
we arrive at the conclusion that the second derivative of the function $J^*(C)$ is positive. Hence, the function $J^*(C)$ is convex.

The second term, which is the limit
$\overline{c}^{\mathrm{upper}}(C)$ is defined in Section \ref{S7.1}.
According to Lemma \ref{lem5},
$\overline{c}^{\mathrm{upper}}(C)=\psi(C)$. According to Lemma \ref{lem8} the function $\psi(C)$ is a convex function in $C$ and its maximum is $\psi(0)=c^*$. The function $\eta(D)=\eta(-C)$ is a concave function in $D$ (convex in $C$) with $\eta(0)=c^*$. Hence, the solution to control problem is unique.

\subsubsection{Structure of the optimal solution and corollary}\label{S7.1.3}
Now we discuss the structure of the optimal solution to the control
problem.
It is associated with three possible cases considered in Section \ref{S7.1.4}.

\smallskip
\textit{Case 1:} $\lim_{L\to\infty}L[\rho_1(L)-1]>0$. This case is associated with Condition (i). In this case, the minimum of $\overline{J}^{\mathrm{upper}}(C)$ occurs for $C=\overline{C}>0$.
Then, $\overline{c}^{\mathrm{upper}}(\overline{C})<c^*$, and the value of the limiting term for $p_1J_1+p_2J_2+\sum_{i=L^{\mathrm{lower}+1}}^{L^{\mathrm{upper}}}c_iq_i$ of the series of objective functions in
\eqref{I1} is given by $\overline{J}^{\mathrm{upper}}(\overline{C})$.

\smallskip
\textit{Case 2:} $\lim_{L\to\infty}L[\rho_1(L)-1]=C<0$. This case is associated with Condition (ii) and under this condition the inequality $j_1< j_2\rho_2/(1-\rho_2)$ is satisfied. In this case, the minimum of $\underline{J}^{\mathrm{lower}}(C)$ occurs for $C=\underline{C}>0$. Then,
$\underline{c}^{\mathrm{lower}}(0)>c^*$, and the value of the limiting term for $p_1J_1+p_2J_2+\sum_{i=L^{\mathrm{lower}+1}}^{L^{\mathrm{upper}}}c_iq_i$ of the series of objective function in
\eqref{I1} is given by $\underline{J}^{\mathrm{under}}(\underline{C})$.

\smallskip
\textit{Case 3:} $\lim_{L\to\infty}L[\rho_1(L)-1]=0$. Since $\overline{c}^{\mathrm{upper}}(0)=c^*>\overline{c}^{\mathrm{upper}}(\overline{C})$ for any positive $C$, then this case must belong to the case $j_1\leq j_2\rho_2/(1-\rho_2)$, where the only equality holds in the only case of same constant water cost for any level of the dam.
In this case, the minimum of $\underline{J}^{\mathrm{lower}}(C)$ occurs for $C=\underline{C}=0$. Then, $\underline{c}^{\mathrm{lower}}(0)=c^*$ and the limiting term for $p_1J_1+p_2J_2+\sum_{i=L^{\mathrm{lower}+1}}^{L^{\mathrm{upper}}}c_iq_i$ of the series of objective functions in
\eqref{I1} is
\begin{equation*}
j_1\frac
{\widetilde{\rho}_{1,2}(\mathsf{E}\varsigma)^{3}+\mathsf{E}\varsigma^{2}-\mathsf{E}\varsigma}
{2\mathsf{E}\varsigma} +j_2\frac{\rho_2}{1-\rho_2}\cdot \frac
{\widetilde{\rho}_{1,2}(\mathsf{E}\varsigma)^{3}+\mathsf{E}\varsigma^{2}-\mathsf{E}\varsigma}
{2\mathsf{E}\varsigma}+c^*.
\end{equation*}
(The result is mentioned in \eqref{SP11.1}.)

\smallskip
Note, that as in the first case so in the third one, the optimal value of the limiting term for $p_1J_1+p_2J_2+\sum_{i=L^{\mathrm{lower}+1}}^{L^{\mathrm{upper}}}c_iq_i$ of the series of objective functions must be less then it is given by \eqref{SP11.1}, since otherwise the strategy with $\rho_1=1$ would be also selected, which is impossible since the optimal solution is unique.

So, Theorem \ref{thm3*} is proved and the cases are discussed.

\smallskip
From Theorem \ref{thm3*} we have the following evident property of
the optimal control.

\begin{cor}\label{cor2}
The solution to the control problem can be $\rho_1=1$ only in the
case $j_1\leq j_2{\rho_2}/(1-\rho_2)$. Specifically, the only equality is
achieved for $c_i\equiv c$, $i=1,2,\ldots,L$, where $c$ is an
arbitrary positive constant.
\end{cor}

Although Corollary \ref{cor2} provides a result in the form of
simple inequality, this result is not really useful, since it is an
evident extension of the result of \cite{Abramov 2007}. A more
constructive result is obtained for the special case considered in
the next section.

\section{Example of linear costs}\label{Examples}

In this section we study an example related to the case of linear
costs.

Assume that $c_1$ and $c_L<c_1$ are given.  Then the assumption that
the costs are linear means, that
\begin{equation}\label{E1}
c_i = c_1 - \frac{i-1}{L-1}(c_1-c_L), \ \ i=1,2,\ldots, L.
\end{equation}
It is assumed that as $L$ is changed, the costs are recalculated as
follows. The first and last values of the cost $c_1$ and $c_L$
remains the same. Other costs in the intermediate points are
recalculated according to \eqref{E1}.

Therefore, to avoid confusing with the appearance of the index $L$
for the fixed (unchangeable) values of cost  $c_1$ and $c_L$, we use
the other notation: $c_1=\overline{c}$ and $c_L=\underline{c}$. Then
\eqref{E1} has the form
\begin{equation}\label{E2}
c_i = \overline{c} - \frac{i-1}{L-1}(\overline{c}-\underline{c}), \
\ i=1,2,\ldots, L.
\end{equation}

In the following we shall also use the inverse form of \eqref{E2}.
Namely,
\begin{equation}\label{E3}
c_{L-i}=\underline{c}+\frac{i}{L-1}(\overline{c}-\underline{c}), \ \
i=0,1,\ldots, L-1.
\end{equation}

Apparently,
\begin{equation}\label{E4}
c^*=\frac{\overline{c}+\underline{c}}{2}.
\end{equation}
For $\overline{c}^{\mathrm{upper}}(C)$ we have
\begin{equation}\label{E5}
\begin{aligned}
&\overline{c}^{\mathrm{upper}}(C)=\psi(D)\\
&=\lim_{L\to\infty}\frac{\sum_{j=0}^{L-1}
\left(\underline{c}+\frac{j}{L-1}(\overline{c}-\underline{c})\right)
\left(1-\frac{2D\mathsf{E}\varsigma}
{[{\rho}_{1,2}(L)(\mathsf{E}\varsigma)^{3}+\mathsf{E}\varsigma^{2}-\mathsf{E}\varsigma]L}\right)^j}
{\sum_{j=0}^{L-1} \left(1-\frac{2C\mathsf{E}\varsigma}
{[{\rho}_{1,2}(L)(\mathsf{E}\varsigma)^{3}+\mathsf{E}\varsigma^{2}-\mathsf{E}\varsigma]L}\right)^j}\\
&=\underline{c}+(\overline{c}-\underline{c})\lim_{L\to\infty}\frac{1}{L-1}\cdot
\frac{\sum_{j=0}^{L-1} j \left(1-\frac{2C\mathsf{E}\varsigma}
{[{\rho}_{1,2}(L)(\mathsf{E}\varsigma)^{3}+\mathsf{E}\varsigma^{2}-\mathsf{E}\varsigma]L}\right)^j}
{\sum_{j=0}^{L-1} \left(1-\frac{2C\mathsf{E}\varsigma}
{[{\rho}_{1,2}(L)(\mathsf{E}\varsigma)^{3}+\mathsf{E}\varsigma^{2}-\mathsf{E}\varsigma]L}\right)^j}\\
&=\underline{c}+(\overline{c}-\underline{c}) \frac
{{\rho}_{1,2}(L)(\mathsf{E}\varsigma)^{3}+\mathsf{E}\varsigma^{2}-\mathsf{E}\varsigma}{2C\mathsf{E}\varsigma}\\
&\ \ \ \times \frac{-\frac{2C\mathsf{E}\varsigma}
{{\rho}_{1,2}(L)(\mathsf{E}\varsigma)^{3}+\mathsf{E}\varsigma^{2}-\mathsf{E}\varsigma}+
\exp\left(\frac{2C\mathsf{E}\varsigma}
{{\rho}_{1,2}(L)(\mathsf{E}\varsigma)^{3}+\mathsf{E}\varsigma^{2}-\mathsf{E}\varsigma}\right)-1}
{\exp\left(\frac{2C\mathsf{E}\varsigma}
{{\rho}_{1,2}(L)(\mathsf{E}\varsigma)^{3}+\mathsf{E}\varsigma^{2}-\mathsf{E}\varsigma}\right)-1}.
\end{aligned}
\end{equation}
For example, as $C=D=0$ in \eqref{E5}, then
$\overline{c}^{\mathrm{upper}}(0)$ is
$\underline{c}+{1}/{2}(\overline{c}-\underline{c})=c^*$. This is in
agreement with the statement of Proposition \ref{prop1}.

In turn, for $\underline{c}^{\mathrm{lower}}(C)$ ($C$ is negative),
we have
\begin{equation}\label{E6}
\begin{aligned}
&\underline{c}^{\mathrm{lower}}(C)=\eta(D)\\
&=\lim_{L\to\infty}\frac{\sum_{j=0}^{L-1}
\left(\underline{c}+\frac{j}{L-1}(\overline{c}-\underline{c})\right)
\left(1-\frac{2C\mathsf{E}\varsigma}
{[{\rho}_{1,2}(L)(\mathsf{E}\varsigma)^{3}+\mathsf{E}\varsigma^{2}-\mathsf{E}\varsigma]L}\right)^j}
{\sum_{j=0}^{L-1} \left(1-\frac{2C\mathsf{E}\varsigma}
{[{\rho}_{1,2}(L)(\mathsf{E}\varsigma)^{3}+\mathsf{E}\varsigma^{2}-\mathsf{E}\varsigma]L}\right)^j}\\
&=\underline{c}+(\overline{c}-\underline{c})\lim_{L\to\infty}\frac{1}{L-1}\cdot
\frac{\sum_{j=0}^{L-1} j \left(1-\frac{2C\mathsf{E}\varsigma}
{[{\rho}_{1,2}(L)(\mathsf{E}\varsigma)^{3}+\mathsf{E}\varsigma^{2}-\mathsf{E}\varsigma]L}\right)^j}
{\sum_{j=0}^{L-1} \left(1-\frac{2C\mathsf{E}\varsigma}
{[{\rho}_{1,2}(L)(\mathsf{E}\varsigma)^{3}+\mathsf{E}\varsigma^{2}-\mathsf{E}\varsigma]L}\right)^j}\\
&=\underline{c}-(\overline{c}-\underline{c}) \frac
{{\rho}_{1,2}(L)(\mathsf{E}\varsigma)^{3}+\mathsf{E}\varsigma^{2}-\mathsf{E}\varsigma}
{2C\mathsf{E}\varsigma}\\
&\ \ \ \times \frac{-\frac{2C\mathsf{E}\varsigma}
{{\rho}_{1,2}(L)(\mathsf{E}\varsigma)^{3}+\mathsf{E}\varsigma^{2}-\mathsf{E}\varsigma}-1+
\exp\left(\frac{2C\mathsf{E}\varsigma}
{{\rho}_{1,2}(L)(\mathsf{E}\varsigma)^{3}+\mathsf{E}\varsigma^{2}-\mathsf{E}\varsigma}\right)}
{1-\exp\left(\frac{2C\mathsf{E}\varsigma}
{{\rho}_{1,2}(L)(\mathsf{E}\varsigma)^{3}+\mathsf{E}\varsigma^{2}-\mathsf{E}\varsigma}\right)}.
\end{aligned}
\end{equation}
Again, as $C=0$ in \eqref{E6}, then
$\underline{c}^{\mathrm{lower}}(0)$ is
$\underline{c}+{1}/{2}(\overline{c}-\underline{c})=c^*$. So, we
arrive at the agreement with the statement of Proposition
\ref{prop1}.

\section{Numerical study}\label{Numerical results}
The explicit solution in the case of linear costs is very routine
and cumbersome. We provide below a numerical study. For simplicity,
the input flow in the numerical example is assumed to be ordinary
Poisson, that is we set $\mathrm{E}\varsigma=1$ and
$\mathrm{E}\varsigma^{2}=1$ in our calculations.

Following Corollary \ref{cor2}, take first
$j_1=j_2{\rho_2}/(1-\rho_2)$. Clearly, that for this relation
between parameters $j_1$ and $j_2$ the minimum of
$\underline{J}^{\mathrm{lower}}(C)$ must be achieved for $C=0$,
while the minimum of $\overline{J}^{\mathrm{upper}}(C)$ must be
achieved for a positive $C$. Now, keeping $j_1$ fixed assume that
$j_2$ increases. Then, the problem is to find the value of parameter
$j_2$ such that the value $C$ corresponding to the minimization of
$\overline{J}^{\mathrm{upper}}(C)$ reaches zero.

In our example we take $j_1=1$, $\rho_2={1}/{2}$, $\underline{c}=1$,
$\overline{c}=2$, $\widetilde{\rho}_{1,2}=1$. In the table below we
outline some values $j_2$ and the corresponding value $C$ for the
optimal solution of $\overline{J}^{\mathrm{upper}}(C)$. It is seen
from the table that the optimal value is achieved in the case
$j_2\approx1.34$. Therefore, in the present example $j_1=1$ and
$j_2\approx1.34$ lead to the optimal solution $\rho_1=1$.

\begin{table}
    \begin{center}
            \caption{The values of parameter $j_2$ and corresponding
        arguments of optimal value $C$}
        \begin{tabular}{c|c|c}\hline
         Parameter & Argument of optimal value & Difference \\
        $j_2$ & $C$&$C(j_2+1)-C(j_2)$\\
        \hline
         1.06 & 0.200&  \\
         1.08 & 0.180&  0.020\\
         1.10 & 0.162&  0.018\\
         1.12 & 0.144&  0.018\\
         1.14 & 0.128&  0.016\\
         1.16 & 0.112&  0.016\\
         1.18 & 0.096&  0.016\\
         1.20 & 0.081&  0.015\\
         1.22 & 0.067&  0.014\\
         1.24 & 0.054&  0.013\\
         1.26 & 0.042&  0.012\\
         1.28 & 0.030&  0.012\\
         1.30 & 0.019&  0.011\\
         1.32 & 0.009&  0.010\\
         1.34 & 0    &  0.009\\
         \hline
        \end{tabular}
    \end{center}
\end{table}

\begin{figure}
\includegraphics[width=8cm, height=10cm]{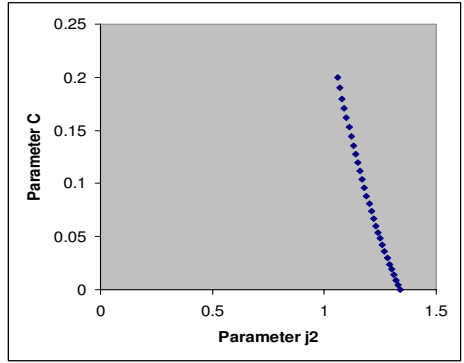}
\caption{Graph of the dependence of $C$ on $j_2$}
\end{figure}

It is seen from Table 1 that as $j_2$ increases, the value of cost
$C$ is monotonically decreases. The inverse proportion between $j_2$
and $C$ implies that the optimal value of
$\overline{J}^{\mathrm{upper}}(C)$ will be lower when the value of
$j_2$ is larger. That is, when the damage of flooding is kept
increased, the value $C$ and consequently the value of
$\overline{J}^{\mathrm{upper}}(C)$ must be decreased. The third
column in the table shows how the the parameter $C$ decreases when
the parameter $j_2$ increases. It shows almost linear dependence of
these parameters. The graph of the dependence of $C$ on $j_2$ is
shown in Figure 5.

\section{Proofs}\label{Proofs}

\subsection*{\textsc{Proof of Lemma} \ref{Asymp behav 1}}
Asymptotic relations \eqref{SP5.3} and \eqref{SP5.5} follow by
application of those \eqref{T.1} and, respectively, \eqref{T.3} of
Lemma \ref{lem.T}.

In order to prove asymptotic relation \eqref{SP5.4} we should apply
a Tauberian theorem by Postnikov (Lemma \ref{lem.P}). Then
asymptotic relation \eqref{SP5.4} is to follow from \eqref{P.1} if
we prove that the Tauberian condition $f_0+f_1<1$ of Lemma
\ref{lem.P} is satisfied. In the case of the present model, we must
prove that for some $\lambda_0>0$ the equality
\begin{equation}\label{SP5.5+1}
\int_0^\infty\mathrm{e}^{-\lambda_0x}(1+\lambda_0r_1x)\mathrm{d}B_1(x)=1
\end{equation}
is not the case. Without loss of generality $r_1$ in \eqref{SP5.5+1}
can be set to be equal to 1, since
\begin{equation*}\label{SP5.5+0}
\int_0^\infty\mathrm{e}^{-\lambda_0x}(1+\lambda_0r_1x)\mathrm{d}B_1(x)\leq
\int_0^\infty\mathrm{e}^{-\lambda_0x}(1+\lambda_0x)\mathrm{d}B_1(x).
\end{equation*}
Thus, we have to prove the inequality
\begin{equation*}\label{SP5.5+2}
\int_0^\infty\mathrm{e}^{-\lambda x}(1+\lambda x)\mathrm{d}B_1(x)<1.
\end{equation*}
Indeed, $\int_0^\infty\mathrm{e}^{-\lambda x}(1+\lambda
x)\mathrm{d}B_1(x)$ is an analytic function in $\lambda$, and hence,
according to the theorem on the maximum module of an analytic
function, equality \eqref{SP5.5+1} where $r_1=1$ must hold for
\textit{all} $\lambda_0\geq0$. This means that \eqref{SP5.5+1} is
valid if and only if $$
\int_0^\infty\mathrm{e}^{-\lambda_0x}\frac{(\lambda_0
x)^i}{i!}\mathrm{d}B_1(x)=0 $$ for all $i\geq2$ and
$\lambda_0\geq0$. Since $\int_0^\infty\mathrm{e}^{-\lambda
x}{(-x)^i}\mathrm{d}B_1(x)$ is the $i$th derivative of the
Laplace-Stieltjes transform $\widehat B_1(\lambda)$, then in this
case the Laplace-Stieltjes transform $\widehat B_1(\lambda)$ must be
a linear function in $\lambda$, i.e. $\widehat
B_1(\lambda)=d_0+d_1\lambda$, where $d_0$ and $d_1$ are some
constants. However, since $|\widehat B_1(\lambda)|\leq1$, we must
have $d_0=1$ and $d_1=0$. This is a trivial case where $B_1(x)$ is
concentrated in point 0, and therefore it is not a probability
distribution function having a positive mean.
 Thus \eqref{SP5.5+1} is not the case, and the
aforementioned Tauberian conditions are satisfied.

Now, the final part of the proof of \eqref{SP5.4} reduces to an
elementary algebraic calculations:
$$
\gamma_2:=\frac{\mathrm{d}^{2}}{\mathrm{d}z^{2}}\widehat{B}_1(\lambda-\lambda\widehat{R}(z))|_{z=1}
=\frac{\mathsf{E}\varsigma^{2}}{\mathsf{E}\varsigma}-1+\rho_{1,2}(\mathsf{E}\varsigma)^{2}.
$$
The lemma is proved.

\subsection*{\textsc{Proof of Theorem} \ref{M asymp}}
Let us first find asymptotic representation for
$\mathsf{E}\nu_L^{(1)}(\zeta_1\wedge L)$ as $L\to\infty$.
According to Lemma \ref{Asymp behav 1} and explicit representation
\eqref{MXG1.15}, we obtain as follows.

If $\rho_1<1$, then
\begin{equation}\label{SP5.12}
\begin{aligned}
\lim_{L\to\infty}\mathsf{E}\nu_L^{(1)}(\zeta_1\wedge L)&=
\frac{1}{1-\rho_1}\lim_{L\to\infty}\sum_{i=1}^{L}i\mathsf{Pr}\{\zeta_1\wedge L=i\}\\
&=1+\frac{\mathsf{E}\zeta_1}{1-\rho_1}.
\end{aligned}
\end{equation}
If $\rho_1=1$, $\rho_{1,2}<\infty$ and
$\mathsf{E}\varsigma^{2}<\infty$, then
\begin{equation}\label{SP5.13}
\begin{aligned}
\lim_{L\to\infty}\frac{\mathsf{E}\nu_L^{(1)}(\zeta_1\wedge L)}{L}=&\frac{2\mathsf{E}\varsigma}{\rho_{1,2}(\mathsf{E}\varsigma)^{3}+\mathsf{E}\varsigma^{2}
-\mathsf{E}\varsigma}\\
&\times\lim_{L\to\infty}
\sum_{i=1}^{L}i\mathsf{Pr}\{\zeta_1\wedge L=i\}\\
=&\frac{2\mathsf{E}\varsigma\mathsf{E}\zeta_1}{\rho_{1,2}(\mathsf{E}\varsigma)^{3}+\mathsf{E}\varsigma^{2}-
\mathsf{E}\varsigma}.
\end{aligned}
\end{equation}

If $\rho_1>1$, then
\begin{equation}\label{SP5.14}
\begin{aligned}
\lim_{L\to\infty}\mathsf{E}\nu_L^{(1)}(\zeta_1\wedge L)\varphi^L&=
\frac{1}{1+\lambda\widehat{B}_1^\prime(\lambda-\lambda\widehat{R}(\varphi))\widehat{R}^\prime(\varphi)}\\
&\ \ \ \times
\lim_{L\to\infty}\sum_{i=1}^Li\mathsf{Pr}\{\zeta_1\wedge L=i\}\\
&=\frac{\mathsf{E}\zeta_1}{1+\lambda\widehat{B}_1^\prime(\lambda-\lambda\widehat{R}(\varphi))\widehat{R}^\prime(\varphi)}.
\end{aligned}
\end{equation}
Therefore, taking into account these limiting relations
\eqref{SP5.12}, \eqref{SP5.13} and \eqref{SP5.14} by virtue of the
equality $\mathsf{E}\nu_L^{(1)}(\zeta_1\wedge L)=\mathsf{E}\nu_L^{(1)}(\zeta_1)$ and explicit representations
\eqref{SP4.19} and \eqref{SP4.20} (Lemma \ref{Prel asymp}) for $p_1$
and $p_2$, we finally arrive at the statements of the theorem. The
theorem is proved.

\subsection*{\textsc{Proof of Theorem} \ref{i1}}
Note first, that under Condition \ref{cond1} there is the following
expansion for $\varphi$:
\begin{equation}\label{SP6.3}
\varphi=1-\frac{2\delta\mathsf{E}\varsigma}{\widetilde{\rho}_{1,2}(\mathsf{E}\varsigma)^{3}+\mathsf{E}\varsigma^{2}-\mathsf{E}\varsigma}+O(\delta^{2}),
\end{equation}
where $\varphi$ itself is the limiting root of the series of functional equations, based on application of Lemma \ref{lem9}.
This expansion is similar to that given originally in the book of
Subhankulov \cite{Subhankulov 1976}, p.362, and its proof is
provided as follows. Write the equation
$\varphi=\widehat{B}_1(\lambda-\lambda\widehat{R}(\varphi))$ and
expand the right-hand side by Taylor's formula taking $\varphi=1-z$,
where $z$ is small enough when $\delta$ is small. We obtain:
\begin{equation}\label{SP6.4}
\begin{aligned}
1-z&=1-(1+\delta)z
+\frac{\widetilde\rho_{1,2}(\mathsf{E}\varsigma)^{3}+(1+\delta)\left({\mathsf{E}\varsigma^{2}}-
{\mathsf{E}\varsigma}\right)}{2\mathsf{E}\varsigma}z^{2} +O(z^{3}).
\end{aligned}
\end{equation}
Disregarding the small term $O(z^{3})$ in \eqref{SP6.4} we arrive at
the quadratic equation
\begin{equation}\label{SP6.5}
\delta
z-\frac{\widetilde\rho_{1,2}(\mathsf{E}\varsigma)^{3}+(1+\delta)\left({\mathsf{E}\varsigma^{2}}-
{\mathsf{E}\varsigma}\right)}{2\mathsf{E}\varsigma}z^{2}=0.
\end{equation}
The positive solution of \eqref{SP6.5},
$$z=\frac{2\delta\mathsf{E}\varsigma}{\widetilde\rho_{1,2}(\mathsf{E}\varsigma)^{3}+(1+\delta)\left({\mathsf{E}\varsigma^{2}}-
{\mathsf{E}\varsigma}\right)},$$ leads to the expansion given by
\eqref{SP6.3}.

Let us now expand the right-hand side of \eqref{SP5.14} when
$\delta$ is small. For the term
$1+\lambda\widehat{B}_1^\prime(\lambda-\lambda
\widehat{R}(\varphi))\widehat{R}^\prime(\varphi)$ we have the
expansion
\begin{equation}\label{SP6.6}
1+\lambda\widehat{B}_1^\prime(\lambda-\lambda
\widehat{R}(\varphi))\widehat{R}^\prime(\varphi)=\delta+O(\delta^{2}),
\end{equation}
Then, according to the l'Hospitale rule
$$
\lim_{u\uparrow1}\frac{1-\widehat{R}(u)}{1-u}=\mathsf{E}\varsigma.
$$
Hence
\begin{equation}\label{SP6.7}
\frac{1-\widehat{R}(\varphi)}{1-\varphi}=\mathsf{E}\varsigma[1+o(1)].
\end{equation}

 Substituting \eqref{SP6.3}, \eqref{SP6.6} and \eqref{SP6.7} into
 \eqref{SP5.14} we obtain the expansion
\begin{equation}\label{SP6.8}
\mathsf{E}\nu_L^{(1)}(\zeta_1\wedge L)=\frac{\exp\left(\frac{2C\mathsf{E}\varsigma}{\widetilde{\rho}_{1,2}(\mathsf{E}\varsigma)^{3}
+\mathsf{E}\varsigma^{2}-\mathsf{E}\varsigma}\right)-1}{\delta}\mathsf{E}\varsigma[1+o(1)].
\end{equation}
Hence, relations \eqref{SP6.1} and \eqref{SP6.2} of the theorem
follow by virtue of the equality
$\mathsf{E}\nu_L^{(1)}(\zeta_1\wedge L)=\mathsf{E}\nu_L^{(1)}(\zeta_1)$
and explicit representations \eqref{SP4.19} and \eqref{SP4.20}
(Lemma \ref{Prel asymp}) for $p_1$ and $p_2$.

\subsection*{\textsc{Proof of Theorem} \ref{i3}}
The explicit representation for the generating function for
$\mathsf{E}\widetilde\nu_j^{(1)}$ is given by \eqref{SP5.2}. Since
the sequence $\{\mathsf{E}\widetilde\nu_j^{(1)}\}$ is increasing,
then the asymptotic behavior of $\mathsf{E}\nu_L^{(1)}(\zeta_1)$ as
$L\to\infty$ under the assumptions $L[\rho_1(L)-1]\to C$ as
$L\to\infty$ can be found according to a Tauberian theorem of Hardy
and Littlewood (see e.g. \cite{Postnikov 1980}, \cite{Takacs
1967}, p.203 and \cite{SzF}). Namely, according to that theorem, the behaviour of
$\mathsf{E}\widetilde\nu_L^{(1)}$ as $L\to \infty$ and
$L[\rho_1(L)-1]\to C$ can be found from the asymptotic expansion of
\begin{equation}\label{SP6.13}
(1-z)\frac{\widehat{B}_1(\lambda-\lambda\widehat{R}(z))}{\widehat{B}_1(\lambda-\lambda\widehat{R}(z))-z}
\end{equation}
as $z\uparrow1$. Similarly to the evaluation given in the proof of
Theorem 4.3 \cite{Abramov 2007}, we have:
\begin{equation}\label{SP6.14}
\begin{aligned}
&(1-z)\frac{\widehat{B}_1(\lambda-\lambda\widehat{R}(z))}{\widehat{B}_1(\lambda-\lambda\widehat{R}(z))-z}\\&=
\frac{1-z}{1-z-\rho_1(1-z)+\frac
{\widetilde{\rho}_{1,2}(\mathsf{E}\varsigma)^{3}+\mathsf{E}\varsigma^{2}-\mathsf{E}\varsigma}{2\mathsf{E}\varsigma}(1-z)^2+O((1-z)^{3})}\\
&=\frac{1}{\delta+\frac
{\widetilde{\rho}_{1,2}(\mathsf{E}\varsigma)^{3}+\mathsf{E}\varsigma^{2}-\mathsf{E}\varsigma}{2\mathsf{E}\varsigma}(1-z)+O((1-z)^{2})}\\
&=\frac{1}{\delta\left[1+\frac
{\widetilde{\rho}_{1,2}(\mathsf{E}\varsigma)^{3}+\mathsf{E}\varsigma^{2}-\mathsf{E}\varsigma}{2\delta\mathsf{E}\varsigma}(1-z)\right]+O((1-z)^{2})}\\
&=\frac{1}{\delta\exp\left[\frac
{\widetilde{\rho}_{1,2}(\mathsf{E}\varsigma)^{3}+
\mathsf{E}\varsigma^{2}-\mathsf{E}\varsigma}{2\delta\mathsf{E}\varsigma}(1-z)\right]}[1+o(1)],
\end{aligned}
\end{equation}
where $\delta=\delta(L)$ denoted the difference $1-\rho_1(L)$.

Therefore, assuming that $z=\frac{L-1}{L}\to1$ as $L\to\infty$, from
\eqref{SP6.14} we arrive at the following estimate:
\begin{equation}\label{SP6.15}
\mathsf{E}\widetilde\nu_L^{(1)}=\frac{1}{\delta}\exp\left(-\frac
{\widetilde{\rho}_{1,2}(\mathsf{E}\varsigma)^{3}+\mathsf{E}\varsigma^{2}-\mathsf{E}\varsigma}{2C\mathsf{E}\varsigma}\right)[1+o(1)].
\end{equation}

Comparing \eqref{SP5.5} with \eqref{SP5.14} and taking into account
\eqref{SP6.7}, which holds true in the case of this theorem as well,
we obtain:
 \begin{equation}\label{SP6.16}
\mathsf{E}\nu_L^{(1)}(\zeta_1\wedge L)=\frac{\mathsf{E}\varsigma}{\delta}\exp\left(-\frac
{\widetilde{\rho}_{1,2}(\mathsf{E}\varsigma)^{3}+\mathsf{E}\varsigma^{2}-\mathsf{E}\varsigma}{2C\mathsf{E}\varsigma}\right)[1+o(1)].
\end{equation}
Hence, relations \eqref{SP6.11} and \eqref{SP6.12} of the theorem
follow by virtue of the equality
$\mathsf{E}\nu_L^{(1)}(\zeta_1\wedge L)=\mathsf{E}\nu_L^{(1)}(\zeta_1)$
and explicit representations \eqref{SP4.19} and \eqref{SP4.20}
(Lemma \ref{Prel asymp}) for $p_1$ and $p_2$.

\subsection*{\textsc{Proof of Theorem} \ref{thm1}}
Based on Lemma \ref{lem9}, we have the following expansion of \eqref{SP5.5} for large $L$:
\begin{equation}\label{SP9.2}
\mathsf{E}\widetilde\nu_{L-j}^{(1)}=\frac{\varphi^j}{\varphi^L[1+\lambda
\widehat{B}_1^\prime(\lambda-\lambda\widehat{R}(\varphi))\widehat{R}^\prime(\varphi)]}[1+o(1)].
\end{equation}
In turn, from \eqref{SP9.2} for large $L$ we obtain:
\begin{equation}\label{SP9.3}
\mathsf{E}\widetilde\nu_{L-j}^{(1)}-\mathsf{E}\widetilde\nu_{L-j-1}^{(1)}=\frac{(1-\varphi)\varphi^j}
{\varphi^L[1+\lambda
\widehat{B}_1^\prime(\lambda-\lambda\widehat{R}(\varphi))\widehat{R}^\prime(\varphi)]}[1+o(1)].
\end{equation}
From \eqref{SP9.3}, similarly to \eqref{SP5.14}, we further have:
\begin{equation*}
\begin{aligned}
&\mathsf{E}\nu_{L-j}^{(1)}(\zeta_1\wedge L)-\mathsf{E}
\nu_{L-j-1}^{(1)}(\zeta_1\wedge L)\\
&=\frac{(1-\widehat{R}(\varphi))(1-\varphi)\varphi^j}{[1+\lambda\widehat{B}_1^\prime(\lambda-\lambda\widehat{R}(\varphi))
\widehat{R}^\prime(\varphi)](1-\varphi)}[1+o(1)],
\end{aligned}
\end{equation*}
and, according to the equality
$\mathsf{E}\nu_L^{(1)}(\zeta_1\wedge L)=\mathsf{E}\nu_L^{(1)}(\zeta_1)$,
\begin{equation}\label{SP9.5}
\begin{aligned}
&\mathsf{E}\nu_{L-j}^{(1)}(\zeta_1)-\mathsf{E}\nu_{L-j-1}^{(1)}(\zeta_1)\\
&=\frac{(1-\widehat{R}(\varphi))(1-\varphi)\varphi^j}{[1+\lambda\widehat{B}_1^\prime(\lambda-\lambda\widehat{R}(\varphi))
\widehat{R}^\prime(\varphi)](1-\varphi)}
[1+o(1)].
\end{aligned}
\end{equation}

Next, under the conditions of the theorem, asymptotic expansions
\eqref{SP6.3} \eqref{SP6.6} and \eqref{SP6.7} hold. Taking into
consideration these expansions we arrive at the following asymptotic
relations for $j=0,1,\ldots$:
 \begin{equation*}\label{SP9.7}
 \begin{aligned}
 \mathsf{E}\nu_{L-j}^{(1)}(\zeta_1)-\mathsf{E}\nu_{L-j-1}^{(1)}(\zeta_1)&=
 \exp\left(\frac{2C\mathsf{E}\varsigma}{\widetilde{\rho}_{1,2}(\mathsf{E}\varsigma)^{3}+\mathsf{E}\varsigma^{2}-\mathsf{E}\varsigma}\right)\\
& \ \ \ \times
 \left(1-\frac{2\delta\mathsf{E}\varsigma}{\widetilde{\rho}_{1,2}
 (\mathsf{E}\varsigma)^{3}+\mathsf{E}\varsigma^{2}-\mathsf{E}\varsigma}\right)^j\\
& \ \ \ \times
\frac{2\delta\mathsf{E}\varsigma}{\widetilde{\rho}_{1,2}(\mathsf{E}\varsigma)^{3}+
 \mathsf{E}\varsigma^{2}-\mathsf{E}\varsigma}[1+o(1)].
 \end{aligned}
 \end{equation*}
 Now, taking into account asymptotic relation \eqref{SP6.1} of Theorem \ref{i1} and the explicit formula given by \eqref{SP7.1} (Lemma \ref{Explicit qi}) we arrive at the statement of the theorem.

\subsection*{\textsc{Proof of Theorem} \ref{thm4}}
Under the assumptions of this theorem let us first derive the
following asymptotic expansion:
\begin{equation}\label{SP10.13}
\tau=1+\frac{2\delta\mathsf{E}\varsigma}{\widetilde{\rho}_{1,2}(\mathsf{E}\varsigma)^{3}+\mathsf{E}\varsigma^{2}-\mathsf{E}\varsigma}+O(\delta^{2}).
\end{equation}
Asymptotic expansion \eqref{SP10.13} is similar to that of
\eqref{SP6.3}, and its proof is also similar. Namely, taking into
account that the equation
$z=\widehat{B}_1(\lambda-\lambda\widehat{R}(z))$  has a unique
solution in the set $(1,\infty)$, and this solution approaches 1 as
$\delta$ vanishes. Therefore, by the Taylor expansion of this
equation around the point $z=1$, we have:
\begin{equation}\label{SP10.14}
1+z=1+(1-\delta)z+\frac{\widetilde{\rho}_{1,2}(\mathsf{E}\varsigma)^{3}+(1-\delta)(\mathsf{E}\varsigma^{2}-\mathsf{E}\varsigma)}
{2\mathsf{E}\varsigma}z^{2}+O(z^{3}).
\end{equation}
Disregarding the term $O(z^{3})$, from \eqref{SP10.14} we arrive at
the quadratic equation
$$
\delta
z-\frac{\widetilde{\rho}_{1,2}(\mathsf{E}\varsigma)^{3}+(1-\delta)(\mathsf{E}\varsigma^{2}-\mathsf{E}\varsigma)}
{2\mathsf{E}\varsigma}z^{2}=0,
$$
and obtain a positive solution
$$z=\frac{2\delta\mathsf{E}\varsigma}{\widetilde\rho_{1,2}(\mathsf{E}\varsigma)^{3}+(1+\delta)\left({\mathsf{E}\varsigma^{2}}-
{\mathsf{E}\varsigma}\right)}.$$ This proves \eqref{SP10.13}.

Next, from \eqref{SP10.9}, \eqref{SP10.13} and explicit formula
\eqref{SP7.1} we obtain
\begin{equation}\label{SP10.15}
q_{L-j}=q_L\left(1+\frac{2\delta\mathsf{E}\varsigma}{\widetilde{\rho}_{1,2}(\mathsf{E}\varsigma)^{3}+\mathsf{E}\varsigma^{2}-\mathsf{E}\varsigma}\right)^j[1+o(1)].
\end{equation}
Taking into consideration
\begin{equation*}\label{SP10.16}
\begin{aligned}
&\sum_{j=0}^{L-1}\left(1+\frac{2\delta\mathsf{E}\varsigma}{\widetilde{\rho}_{1,2}(\mathsf{E}\varsigma)^{3}+\mathsf{E}\varsigma^{2}-\mathsf{E}\varsigma}\right)^j\\
&=\frac{\widetilde{\rho}_{1,2}(\mathsf{E}\varsigma)^{3}+\mathsf{E}\varsigma^{2}-\mathsf{E}\varsigma}{2\delta\mathsf{E}\varsigma}\left[\left(1+
\frac{2\delta\mathsf{E}\varsigma}{\widetilde{\rho}_{1,2}(\mathsf{E}\varsigma)^{3}+\mathsf{E}\varsigma^{2}-\mathsf{E}\varsigma}\right)^L-1\right]\\
&=\frac{\widetilde{\rho}_{1,2}(\mathsf{E}\varsigma)^{3}+\mathsf{E}\varsigma^{2}-
\mathsf{E}\varsigma}{2\delta\mathsf{E}\varsigma}\left[\exp
\left(-\frac{2C\mathsf{E}\varsigma}{\widetilde{\rho}_{1,2}(\mathsf{E}\varsigma)^{3}
+\mathsf{E}\varsigma^{2}-\mathsf{E}\varsigma}\right)-1\right]\\
&\ \ \ \times[1+o(1)],
\end{aligned}
\end{equation*}
from the normalization condition $p_1+p_2+\sum_{i=1}^Lq_i=1$
and the fact that both $p_1$ and $p_2$ have the order $O(\delta)$,
we obtain:
\begin{equation}\label{SP10.17}
\begin{aligned}
q_L&=\frac{2\delta\mathsf{E}\varsigma}{\widetilde{\rho}_{1,2}
(\mathsf{E}\varsigma)^{3}+\mathsf{E}\varsigma^{2}-\mathsf{E}\varsigma}\\
&\  \ \ \times\frac{1}{ \exp
\left(-\frac{2C\mathsf{E}\varsigma}{\widetilde{\rho}_{1,2}(\mathsf{E}\varsigma)^{3}
+\mathsf{E}\varsigma^{2}-\mathsf{E}\varsigma}\right)-1}[1+o(1)].
\end{aligned}
\end{equation}
The desired statement of the theorem follows from \eqref{SP10.17}.

\subsection*{\textsc{Proof of Proposition} \ref{prop2}}
The representation for the term
$$
\begin{aligned}
&C\left[j_1\frac{1}{\exp\left(\frac{2C\mathsf{E}\varsigma}{{\rho}_{1,2}(L)
(\mathsf{E}\varsigma)^{3}+\mathsf{E}\varsigma^{2}-\mathsf{E}\varsigma}\right)-1}\right.\\
&\ \ \ \left.+j_2\frac{\rho_2\exp\left(\frac{2C\mathsf{E}\varsigma}
{{\rho}_{1,2}(L)(\mathsf{E}\varsigma)^{3}+\mathsf{E}\varsigma^{2}-\mathsf{E}\varsigma}\right)}
{(1-\rho_2)\left({\exp\left(\frac{2C\mathsf{E}\varsigma}{{\rho}_{1,2}(L)
(\mathsf{E}\varsigma)^{3}+\mathsf{E}\varsigma^{2}-\mathsf{E}\varsigma}\right)-1}\right)}\right]
\end{aligned}
$$
of the right-hand side of \eqref{SP11.4*} follows from \eqref{SP6.1}
and \eqref{SP6.2} (Theorem \ref{i1}). This term is similar to that
of (5.2) in \cite{Abramov 2007}. The new term which takes into
account the water costs is
$\overline{c}^{\mathrm{upper}}(C)=\lim_{L\to\infty}{c}^{\mathrm{upper}}[L,C(L)]$.
Taking into account representation \eqref{SP9.1}, for this term we
obtain:
$$
\begin{aligned}
\overline{c}^{\mathrm{upper}}(C)&=\lim_{L\to\infty}\sum_{j=0}^{L-1}
q_{L-j}c_{L-j}\\&=\lim_{L\to\infty}\sum_{j=0}^{L-1} c_{L-j}\cdot
\frac{\exp\left(\frac{2C\mathsf{E}\varsigma}{{\rho}_{1,2}(L)(\mathsf{E}\varsigma)^{3}+\mathsf{E}\varsigma^{2}-\mathsf{E}\varsigma}\right)}
{\exp\left(\frac{2C\mathsf{E}\varsigma}{{\rho}_{1,2}(L)(\mathsf{E}\varsigma)^{3}+\mathsf{E}\varsigma^{2}-\mathsf{E}\varsigma}\right)-1}\\
&\ \ \ \times\left(1-\frac{2\delta
L\mathsf{E}\varsigma}{[{\rho}_{1,2}(L)(\mathsf{E}\varsigma)^{3}+
\mathsf{E}\varsigma^{2}-\mathsf{E}\varsigma]L}\right)^j\\
&\ \ \ \times\frac{2\delta
L\mathsf{E}\varsigma}{[{\rho}_{1,2}(L)
(\mathsf{E}\varsigma)^{3}+\mathsf{E}\varsigma^{2}-\mathsf{E}\varsigma)L},
\end{aligned}
$$
and, since $\lim_{L\to\infty}\delta L=C$, representation
\eqref{SP11.5*} follows.

\subsection*{\textsc{Proof of Proposition} \ref{prop3}}
The representation for the term
\begin{equation*}
\begin{aligned}
&-C\left[j_1\exp\left(-\frac{{\rho}_{1,2}(L)(\mathsf{E}\varsigma)^{3}+\mathsf{E}\varsigma^{2}-\mathsf{E}\varsigma}
{2C\mathsf{E}\varsigma}\right)\right.\\
&\ \ \ +\left.j_2\frac{\rho_2}{1-\rho_2}
\left(\exp\left(-\frac{{\rho}_{1,2}(L)(\mathsf{E}\varsigma)^{3}+\mathsf{E}\varsigma^{2}-\mathsf{E}\varsigma}
{2C\mathsf{E}\varsigma}\right)-1\right)\right]
\end{aligned}
\end{equation*}
of the right-hand side of \eqref{SP12.1*} follows from
\eqref{SP6.11} and \eqref{SP6.12} (Theorem \ref{i3}). This term is
similar to that (5.3) in \cite{Abramov 2007}. The new term, which
takes into account the water costs, is
$\underline{c}^{\mathrm{lower}}(C)=\lim_{L\to\infty}{c}^{\mathrm{lower}}[L,C(L)]$.
Taking into account representation \eqref{SP10.12}, for this term we
obtain:
$$
\begin{aligned}
\underline{c}^{\mathrm{lower}}(C)&=\lim_{L\to\infty}\sum_{j=0}^{L-1}
q_{L-j}c_{L-j}\\&=\lim_{L\to\infty}\sum_{j=0}^{L-1}
c_{L-j}\cdot\frac{1}{\exp\left(-\frac{2C\mathsf{E}\varsigma}
{{\rho}_{1,2}(L)(\mathsf{E}\varsigma)^{3}+\mathsf{E}\varsigma^{2}-\mathsf{E}\varsigma}\right)-1}\\
&\ \ \ \times \left(1+\frac{2\delta L\mathsf{E}\varsigma}
{[{\rho}_{1,2}(L)(\mathsf{E}\varsigma)^{3}+\mathsf{E}\varsigma^{2}-\mathsf{E}\varsigma]L}\right)^j\\
&\ \ \ \times\frac{2\delta L\mathsf{E}\varsigma}
{[{\rho}_{1,2}(L)(\mathsf{E}\varsigma)^{3}+\mathsf{E}\varsigma^{2}-\mathsf{E}\varsigma]L},
\end{aligned}
$$
and, because of
$$\lim_{L\to\infty}L[\rho_1(L)-1]=C\leq0,$$
representation
\eqref{SP12.2*} follows.

\subsection*{\textsc{Proof of Lemma} \ref{lem5}}
From \eqref{SCP3} and \eqref{SCP4} we correspondingly have the
representations:
\begin{equation}\label{SCP7}
\begin{aligned}
&\lim_{L\to\infty}\frac{1}{L}\sum_{j=0}^{L-1}
c_{L-j}\left(1-\frac{2C\mathsf{E}\varsigma}
{[{\rho}_{1,2}(L)(\mathsf{E}\varsigma)^{3}+\mathsf{E}\varsigma^{2}-\mathsf{E}\varsigma]L}\right)^j\\
&=\psi(D) \lim_{L\to\infty}\frac{1}{L}\sum_{j=0}^{L-1}
\left(1-\frac{2C\mathsf{E}\varsigma}
{[{\rho}_{1,2}(L)(\mathsf{E}\varsigma)^{3}+\mathsf{E}\varsigma^{2}-\mathsf{E}\varsigma]L}\right)^j,
\end{aligned}
\end{equation}
and since $D=-C$ ($C$ is negative),
\begin{equation}\label{SCP8}
\begin{aligned}
&\lim_{L\to\infty}\frac{1}{L}\sum_{j=0}^{L-1}
c_{L-j}\left(1+\frac{2D\mathsf{E}\varsigma}
{[{\rho}_{1,2}(L)(\mathsf{E}\varsigma)^{3}+\mathsf{E}\varsigma^{2}-\mathsf{E}\varsigma]L}\right)^j\\
&=\lim_{L\to\infty}\frac{1}{L}\sum_{j=0}^{L-1}
c_{L-j}\left(1-\frac{2C\mathsf{E}\varsigma}
{[{\rho}_{1,2}(L)(\mathsf{E}\varsigma)^{3}+\mathsf{E}\varsigma^{2}-\mathsf{E}\varsigma]L}\right)^j\\
&=\eta(D) \lim_{L\to\infty}\frac{1}{L}\sum_{j=0}^{L-1}
\left(1-\frac{2C\mathsf{E}\varsigma}
{[{\rho}_{1,2}(L)(\mathsf{E}\varsigma)^{3}+\mathsf{E}\varsigma^{2}-\mathsf{E}\varsigma]L}\right)^j.
\end{aligned}
\end{equation}
The desired results follow by direct transformations of the
corresponding right-hand sides of \eqref{SCP7} and \eqref{SCP8}.

Indeed, for the right-hand side of \eqref{SCP7} we obtain:
\begin{equation}\label{SCP9}
\begin{aligned}
&\psi(D) \lim_{L\to\infty}\frac{1}{L}\sum_{j=0}^{L-1}
\left(1-\frac{2C\mathsf{E}\varsigma}
{[{\rho}_{1,2}(L)(\mathsf{E}\varsigma)^{3}+\mathsf{E}\varsigma^{2}-\mathsf{E}\varsigma]L}\right)^j\\
&=\psi(D)\lim_{L\to\infty}
\left[1-\exp\left(-\frac{2C\mathsf{E}\varsigma}
{[{\rho}_{1,2}(L)(\mathsf{E}\varsigma)^{3}+\mathsf{E}\varsigma^{2}-\mathsf{E}\varsigma]L}\right)\right]\\
&\ \ \ \times\frac
{[{\rho}_{1,2}(L)(\mathsf{E}\varsigma)^{3}+\mathsf{E}\varsigma^{2}-\mathsf{E}\varsigma]}{2C\mathsf{E}\varsigma}.
\end{aligned}
\end{equation}
On the other hand, from \eqref{SP11.5*} we have:
\begin{equation}\label{SCP10}
\begin{aligned}
&\overline{c}^{\mathrm{upper}}(C)\left[1-\exp\left(-\frac{2C\mathsf{E}\varsigma}
{{\rho}_{1,2}(L)(\mathsf{E}\varsigma)^{3}+\mathsf{E}\varsigma^{2}-\mathsf{E}\varsigma}\right)\right]\\
&\ \ \ \times\frac
{{\rho}_{1,2}(L)(\mathsf{E}\varsigma)^{3}+\mathsf{E}\varsigma^{2}-\mathsf{E}\varsigma}
{2C\mathsf{E}\varsigma}\\
&= \frac{1}{L}\sum_{j=0}^{L-1}
c_{L-j}\left(1-\frac{2C\mathsf{E}\varsigma}
{[{\rho}_{1,2}(L)(\mathsf{E}\varsigma)^{3}+\mathsf{E}\varsigma^{2}-\mathsf{E}\varsigma]L}\right)^j.
\end{aligned}
\end{equation}
Hence, from \eqref{SCP7}, \eqref{SCP9} and \eqref{SCP10} we obtain
\eqref{SCP3}. The proof of \eqref{SCP6} is completely analogous and
uses the representations of \eqref{SP12.2*} and \eqref{SCP8}.

\subsection*{\textsc{Proof of Lemma} \ref{lem6}}
Let us first prove that $\psi(0)=c^*$ is a maximum of $\psi(D)$. For
this purpose we use the following well-known inequality (e.g. Hardy,
Littlewood and Polya \cite{Hardy Littlewood Polya 1952} or Marschall
and Olkin \cite{Marschall Olkin 1979}). Let $\{a_n\}$ and $\{b_n\}$
be arbitrary sequences, one of them is increasing and another
decreasing. Then for any finite sum we have
\begin{equation}\label{SCP11}
\sum_{n=1}^l a_nb_n\leq\frac{1}{l}\sum_{n=1}^l a_n\sum_{n=1}^l b_n.
\end{equation}

Applying inequality \eqref{SCP11} to the finite sums of the
left-hand side of \eqref{SCP7} and passing to the limit as
$L\to\infty$, we have
\begin{equation}\label{SCP12}
\begin{aligned}
&\lim_{L\to\infty}\frac{1}{L}\sum_{j=0}^{L-1}
c_{L-j}\left(1-\frac{2D\mathsf{E}\varsigma}
{[{\rho}_{1,2}(L)(\mathsf{E}\varsigma)^{3}+\mathsf{E}\varsigma^{2}-\mathsf{E}\varsigma]L}\right)^j\\
&\leq \lim_{L\to\infty}\frac{1}{L}\sum_{j=0}^{L-1}
c_{L-j}\\
&\ \ \ \times\lim_{L\to\infty}\frac{1}{L}\sum_{j=0}^{L-1}
\left(1-\frac{2D\mathsf{E}\varsigma}
{[{\rho}_{1,2}(L)(\mathsf{E}\varsigma)^{3}+\mathsf{E}\varsigma^{2}-\mathsf{E}\varsigma]L}\right)^j\\
&=\psi(0)\lim_{L\to\infty}\frac{1}{L}\sum_{j=0}^{L-1}
\left(1-\frac{2D\mathsf{E}\varsigma}
{[{\rho}_{1,2}(L)(\mathsf{E}\varsigma)^{3}+\mathsf{E}\varsigma^{2}-\mathsf{E}\varsigma]L}\right)^j.
\end{aligned}
\end{equation}
Then, comparing \eqref{SCP7} with \eqref{SCP12} enables us to
conclude,
\begin{equation*}\label{SCP13}
\psi(0)=c^*\geq\psi(D),
\end{equation*}
i.e. $\psi(0)=c^*$ is the maximum value of $\psi(D)$.

A more delicate result we are going to prove is that $\psi(D)$ is a nonincreasing convex function. We first prove that the function $\psi(D)$ is nonincreasing and then, following formulated Lemmas \ref{lem7} and \ref{lem8}, we later prove that $\psi(D)$ is a \textit{nontrivial} convex function, if the sequence of costs $c_i$ is decreasing (that is, the values $c_i$ all are not equal to a same constant).

To prove the fact that $\psi(D)$ is a nonincreasing convex function we are to derive the function $\psi(D)$ in $D$ and show that the derivative is negative. From the explicit representation for $\psi(D)$
\begin{equation*}
\psi(D)=\lim_{L\to\infty}\frac{\frac{1}{L}\sum_{j=0}^{L-1}c_{L-j}\left(1-\frac{2D\mathsf{E}\varsigma}
{[{\rho}_{1,2}(L)(\mathsf{E}\varsigma)^{3}+\mathsf{E}\varsigma^{2}-\mathsf{E}\varsigma]L}\right)^j}
{\frac{1}{L}\sum_{j=0}^{L-1}\left(1-\frac{2D\mathsf{E}\varsigma}
{[{\rho}_{1,2}(L)(\mathsf{E}\varsigma)^{3}+\mathsf{E}\varsigma^{2}-\mathsf{E}\varsigma]L}\right)^j}
\end{equation*}
we have
\begin{equation}\label{eq.0.3}
\begin{aligned}
&\frac{\mathrm{d}\psi}{\mathrm{d}D}=\lim_{L\to\infty}\left(\frac{1}{L}\sum_{j=0}^{L-1}\left(1-\frac{2D\mathsf{E}\varsigma}
{[{\rho}_{1,2}(L)(\mathsf{E}\varsigma)^{3}+\mathsf{E}\varsigma^{2}-\mathsf{E}\varsigma]L}\right)^j\right)^{-2}\\
\times&\left\{\lim_{L\to\infty}{\frac{1}{L}\sum_{j=0}^{L-1}\left(1-\frac{2D\mathsf{E}\varsigma}
{[{\rho}_{1,2}(L)(\mathsf{E}\varsigma)^{3}+\mathsf{E}\varsigma^{2}-\mathsf{E}\varsigma]L}\right)^j}\right.\\
&\times\lim_{L\to\infty}\frac{\mathrm{d}}{\mathrm{d}D}\left[\frac{1}{L}\sum_{j=0}^{L-1}c_{L-j}\left(1-\frac{2D\mathsf{E}\varsigma}
{[{\rho}_{1,2}(L)(\mathsf{E}\varsigma)^{3}+\mathsf{E}\varsigma^{2}-\mathsf{E}\varsigma]L}\right)^j\right]\\
&-\lim_{L\to\infty}\frac{1}{L}\sum_{j=0}^{L-1}c_{L-j}\left(1-\frac{2D\mathsf{E}\varsigma}
{[{\rho}_{1,2}(L)(\mathsf{E}\varsigma)^{3}+\mathsf{E}\varsigma^{2}-\mathsf{E}\varsigma]L}\right)^j\\
&\times\left.\lim_{L\to\infty}\frac{\mathrm{d}}{\mathrm{d}D}\left[\frac{1}{L}\sum_{j=0}^{L-1}\left(1-\frac{2D\mathsf{E}\varsigma}
{[{\rho}_{1,2}(L)(\mathsf{E}\varsigma)^{3}+\mathsf{E}\varsigma^{2}-\mathsf{E}\varsigma]L}\right)^j\right]\right\}.
\end{aligned}
\end{equation}
The task is to prove that the expression in the arc brackets of \eqref{eq.0.3} is negative or zero. That is, we are to prove that for sufficiently large $L$
\begin{equation}\label{eq.0.7}
\begin{aligned}
&{\sum_{j=0}^{L-1}\left(1-\frac{2D\mathsf{E}\varsigma}
{[{\rho}_{1,2}(L)(\mathsf{E}\varsigma)^{3}+\mathsf{E}\varsigma^{2}-\mathsf{E}\varsigma]L}\right)^j}\\
&\times\frac{\mathrm{d}}{\mathrm{d}D}\left[\sum_{j=0}^{L-1}c_{L-j}\left(1-\frac{2D\mathsf{E}\varsigma}
{[{\rho}_{1,2}(L)(\mathsf{E}\varsigma)^{3}+\mathsf{E}\varsigma^{2}-\mathsf{E}\varsigma]L}\right)^j\right]\\
&-\sum_{j=0}^{L-1}c_{L-j}\left(1-\frac{2D\mathsf{E}\varsigma}
{[{\rho}_{1,2}(L)(\mathsf{E}\varsigma)^{3}+\mathsf{E}\varsigma^{2}-\mathsf{E}\varsigma]L}\right)^j\\
&\times\frac{\mathrm{d}}{\mathrm{d}D}\left[\sum_{j=0}^{L-1}\left(1-\frac{2D\mathsf{E}\varsigma}
{[{\rho}_{1,2}(L)(\mathsf{E}\varsigma)^{3}+\mathsf{E}\varsigma^{2}-\mathsf{E}\varsigma]L}\right)^j\right]
\end{aligned}
\end{equation}
is negative.

Note, that \eqref{eq.0.7} is associated with the representation
\begin{equation}\label{eq.0.13}
\begin{aligned}
&\frac{\mathrm{d}}{\mathrm{d}D}\left[\frac{\sum_{j=0}^{L-1}c_{L-j}\left(1-\frac{2D\mathsf{E}\varsigma}
{[{\rho}_{1,2}(L)(\mathsf{E}\varsigma)^{3}+\mathsf{E}\varsigma^{2}-\mathsf{E}\varsigma]L}\right)^j}
{\sum_{j=0}^{L-1}\left(1-\frac{2D\mathsf{E}\varsigma}
{[{\rho}_{1,2}(L)(\mathsf{E}\varsigma)^{3}+\mathsf{E}\varsigma^{2}-\mathsf{E}\varsigma]L}\right)^j}\right]\\
&\times\left(\sum_{j=0}^{L-1}\left(1-\frac{2D\mathsf{E}\varsigma}
{[{\rho}_{1,2}(L)(\mathsf{E}\varsigma)^{3}+\mathsf{E}\varsigma^{2}-\mathsf{E}\varsigma]L}\right)^j\right)^{2}.
\end{aligned}
\end{equation}

So, the technical task reduces to the following.
Let
\begin{equation}\label{eq.0.10}
\begin{aligned}
f_L(z)=\frac{\sum_{i=0}^{L}a_iz^i}{\sum_{i=0}^{L}z^i}&=\frac{(1-z)\sum_{i=0}^{L}a_iz^i}{1-z^{L+1}}\\
&=\frac{a_0+\sum_{i=1}^L(a_i-a_{i-1})z^i}{1-z^{L+1}}
\end{aligned}
\end{equation}
be the function, $z<1$, in which $a_i$ is an increasing sequence. For the derivative of this function we have
\begin{equation}\label{eq.0.12}
\begin{aligned}
\frac{\mathrm{d}f_L}{\mathrm{d}z}&=\frac{\sum_{i=1}^L i(a_i-a_{i-1})z^{i-1}}{1-z^{L+1}}\\
+&\frac{(L+1)z^L\left(a_0+\sum_{i=1}^L(a_i-a_{i-1})z^i\right)}{(1-z^{L+1})^2}.
\end{aligned}
\end{equation}
Since $a_i$ is an increasing sequence, then the derivative $\mathrm{d}f_L/\mathrm{d}z$ is positive. Then, for the function $f_L(1-y/L)$, in which the argument $z$ is replaced with $1-y/L$, $0<y<L$, now the derivative of this function in $y$ is negative. So, as $L$ increases to infinity, the derivative $\mathrm{d}f_L(1-y/L)/\mathrm{d}y$ tends to the negative value or zero. This enables us to arrive at the conclusion that \eqref{eq.0.3} is negative and, as $L\to\infty$, it tends to the negative value or zero.

The first statement of Lemma \ref{lem6} is proved. The proof of the
second statement of this lemma is similar.

\subsection*{\textsc{Proof of Lemma} \ref{lem7}}
 In order to prove this lemma it is sufficient to prove that if the
sequence $\{c_i\}$ is nontrivial, that is there are at least two
distinct values of this sequence, then for any distinct positive
real numbers $C_1\neq C_2$ the values of functions are also
distinct, that is, $\psi(C_1)\neq \psi(C_2)$ and $\eta(C_1)\neq
\eta(C_2)$. Let us prove the first inequality: $\psi(C_1)\neq
\psi(C_2)$. Rewrite \eqref{SCP3} as
\begin{equation}\label{SCP3'}
\psi(D)=\lim_{L\to\infty}\frac{\frac{1}{L}\sum_{j=0}^{L-1}c_{L-j}\left(1-\frac{2D\mathsf{E}\varsigma}
{[{\rho}_{1,2}(L)(\mathsf{E}\varsigma)^{3}+\mathsf{E}\varsigma^{2}-\mathsf{E}\varsigma]L}\right)^j}
{\frac{1}{L}\sum_{j=0}^{L-1}\left(1-\frac{2D\mathsf{E}\varsigma}
{[{\rho}_{1,2}(L)(\mathsf{E}\varsigma)^{3}+\mathsf{E}\varsigma^{2}-\mathsf{E}\varsigma]L}\right)^j}.
\end{equation}
The limit of the denominator is equal to
$$
\begin{aligned}
&\left[1-\exp\left(-\frac{2D\mathsf{E}\varsigma}
{\widetilde{\rho}_{1,2}(\mathsf{E}\varsigma)^{3}+\mathsf{E}\varsigma^{2}-\mathsf{E}\varsigma}\right)\right]\\
&\times\frac{\widetilde{\rho}_{1,2}(\mathsf{E}\varsigma)^{3}+\mathsf{E}\varsigma^{2}-\mathsf{E}\varsigma}{2D\mathsf{E}\varsigma}.
\end{aligned}
$$
The limit of the numerator does exist and finite, since the sequence
$\{c_i\}$ is assumed to be nonincreasing and bounded. As well,
according to the other representation following from Lemma
\ref{lem5} and relation \eqref{SP11.5*}, it is an analytic function
in $D$ taking a nontrivial set of values.

The  analytic function $\psi(C)$ is defined for all real $C\geq0$
and it can be extended analytically for the whole complex plane.
Extension to the negative values of $C$ enables us to arrive at the
function $\eta(C)=\psi(-C)$, which is now an analytic function for
all real values $C$ and due to analytic continuation is an analytic
function in whole complex plane. According to the maximum principle
for the module of an analytic function, if an analytic function
takes the same values in two distinct points inside the domain of
its definition, then the function must be the constant. If $c_i=c^*$
for all $i=1,2,\ldots$, then this is just the case where
$\psi(C)=c^*$ for all $C$. If there exist $i_0$ and $i_1$ such that
$c_{i_0}\neq c_{i_1}$, then the function $\psi(C)$ cannot be a
constant, because the analytic function is uniquely defined by the
coefficients in the series expansion.
 So, the inequality $\psi(C_1)\neq
\psi(C_2)$ for distinct values $C_1$ and $C_2$ follows. The proof of
the second inequality $\eta(C_1)\neq \eta(C_2)$ is the part of the
first one, since $\eta(C)=\psi(-C)$ is the same analytic function.

\subsection*{\textsc{Proof of Lemma} \ref{lem8}}
To prove the lemma, we are to derive the function $\psi(D)$ in $D$ twice and show that the second derivative is nonnegative. Note that in the nontrivial case when the sequence $\{c_i\}$ is a decreasing sequence containing distinct values, the first derivative of $\psi(D)$, according to Lemmas \ref{lem6} and \ref{lem7}, is strictly negative. The proof given in Lemma \ref{lem6} cannot guarantee this, since the limit, as $L\to\infty$, may reach 0. But the statement of Lemma \ref{lem7} does guarantee this, since the analytic function that is not a constant must take distinct values only.

The expression for the second derivative is hardly observable. Therefore, we do not write down the exact derivations, and instead of them we provide the scheme of calculations only.

We consider the function defined by \eqref{eq.0.10} that is used in the proof of Lemma \ref{lem6}. Its derivative is defined by \eqref{eq.0.12}. Deriving this function the second time, show that the second derivative is positive. Denote the expression in \eqref{eq.0.12} by $I_L^{(1)}(z)+I_L^{(2)}(z)$, where $I_L^{(1)}(z)$ is the first term-fraction of the expression and $I_L^{(2)}(z)$ is the second one. For the derivative of the first term, we obtain
\begin{equation}\label{eq.0.14}
\begin{aligned}
\frac{\mathrm{d}I_L^{(1)}}{\mathrm{d}z}&=\frac{\sum_{i=2}^{L}i(i-1)(a_i-a_{i-1})z^{i-2}}{1-z^{L+1}}\\
+&\frac{(L+1)z^L\sum_{i=1}^{L}i(a_i-a_{i-1})z^{i-1}}{(1-z^{L+1})^2},
\end{aligned}
\end{equation}
and for the derivative of the second term, we obtain
\begin{equation}\label{eq.0.15}
\begin{aligned}
\frac{\mathrm{d}I_L^{(2)}}{\mathrm{d}z}&=\frac{(L+1)z^L\sum_{i=1}^{L}i(a_i-a_{i-1})z^{i-1}}{(1-z^{L+1})^2}\\
+&\frac{(L+1)Lz^{L-1}\left(a_0+\sum_{i=1}^{L}(a_i-a_{i-1})z^i\right)}{(1-z^{L+1})^2}\\
+&\frac{2(L+1)z^L\left(a_0+\sum_{i=1}^{L}(a_i-a_{i-1})z^{i}\right)}{(1-z^{L+1})^3}.
\end{aligned}
\end{equation}
Keeping in mind that the sequence $a_i$ is increasing (we assume that all of $a_1,a_2,\ldots$ of the sequence are not equal to a same constant),
it is readily seen that all the terms-fractions on the right-hand sides of \eqref{eq.0.14} and \eqref{eq.0.15} are positive. That is, the derivatives $\mathrm{d}I_L^{(1)}/\mathrm{d}z$ and $\mathrm{d}I_L^{(2)}/\mathrm{d}z$ both are positive. With the following change of argument, the derivatives
$\mathrm{d}I_L^{(1)}(1-y/L)/\mathrm{d}y$ and $\mathrm{d}I_L^{(2)}(1-y/L)/\mathrm{d}y$, where $0<y<L$, both are negative. This means that the required second derivative
$\mathrm{d}^2f_L(1-y/L)/\mathrm{d}y^2$ is positive. This statement implies that $\psi(D)$ is a convex function. The proof of the fact that $\eta(D)$ is a concave function is similar.

\section*{Acknowledgements}
The author expresses his gratitude to all the people who made critical comments officially or privately.
The author thanks Prof. Phil Howlett (University of South
Australia), whose questions in a local seminar at the University of
South Australia in 2005 initiated the solution of this circle of
problems including the earlier paper by the author \cite{Abramov
2007}.

\end{document}